\documentclass[final]{siamltex}

\usepackage{amssymb}
\usepackage{amsmath}
\usepackage{psfrag}
\usepackage{multirow}
\usepackage[hyphens]{url}
\usepackage{hyperref}
\usepackage[lined, ruled, linesnumbered]{algorithm2e}
\usepackage{epstopdf}
\usepackage[normalem]{ulem}

\usepackage[T1]{fontenc}
\usepackage[ansinew]{inputenc}
\usepackage{lmodern}

\usepackage[letterpaper,centering]{geometry}
\usepackage[textsize=tiny]{todonotes}

\newcommand{\argmin}{\operatornamewithlimits{arg\ min}}

\title{NOWPAC: A provably convergent derivative-free nonlinear optimizer with path-augmented constraints}

\author{F.\ Augustin\footnotemark[1] \and Y.~M.\ Marzouk\footnotemark[1]}

\begin{document}

\maketitle

\renewcommand{\thefootnote}{\fnsymbol{footnote}}

\footnotetext[1]{Massachusetts Institute of Technology, Department of Aeronautics and Astronautics, Cambridge, MA
02139, USA, \texttt{\{fmaugust,ymarz\}@mit.edu}.}

\renewcommand{\thefootnote}{\arabic{footnote}}

% \{Should we still keep `noise' in the title? I think it's okay, since we are precise in the text.}

\begin{abstract}
This paper proposes the algorithm NOWPAC (Nonlinear Optimization With Path-Augmented Constraints) for nonlinear constrained derivative-free optimization. The algorithm uses a trust region framework based on fully linear models for the objective function and the constraints. A new constraint-handling scheme based on an inner boundary path allows for the computation of feasible trial steps using models for the constraints. We prove that the iterates computed by NOWPAC converge to a first-order critical point. We also discuss the convergence of NOWPAC in situations where evaluations of the objective function or the constraints are inexact, e.g., corrupted by numerical errors. We determine a rate of decay that the magnitude of these numerical errors must satisfy, while approaching the critical point, to guarantee convergence. In settings where adjusting the accuracy of the objective or constraint evaluations is not possible, as is often the case in practical applications, we introduce an error indicator to detect these regimes and prevent deterioration of the optimization results.
\end{abstract}

\section{Introduction}\label{sec:introduction}
In the design of industrial processes and other engineered systems, one often has to choose parameters $x \in \mathbb{R}^n$ in order to maximize performance while meeting prescribed requirements. The requirements and the performance objective may be available only as the result of black-box model evaluations, and the requirements may not be expressible in analytical form. To solve these problems, this paper introduces a new derivative-free approach for nonlinear constrained optimization. We generalize existing trust region methodologies by proposing a new scheme for handling general nonlinear constraints, and we prove convergence of the resulting algorithm to a first-order critical point. We also develop additional theory and an error indicator to account for inexact evaluations of the objective and constraints.

More precisely, we are interested in solving optimization programs of the form: 
\begin{equation}\label{eq:general_opt_problem}
\begin{split}
&\;\;\;\;\; \; \; \min f(x)\\
&\mbox{s.t.} \quad c_i(x) \leq 0,\ i=1 \ldots r
\end{split}
\end{equation}
where $f: \mathbb{R}^n \rightarrow \mathbb{R}$ is the objective function defining the quantity of interest and $c_i: \mathbb{R}^n \rightarrow \mathbb{R}$, $i=1 \ldots r$, model the constraints imposed on the design parameters $x$. The constraints define the set of feasible points $X := \{x \in \mathbb{R}^n \, : \, c_i(x) \leq 0, \ i=1 \ldots r \}$, consisting of all admissible designs. There exist many approaches for approximating the solutions of (\ref{eq:general_opt_problem}); see for instance \cite{Bertsekas1976, Bertsekas1999, Boyd2009, Conn2000, Powell1977}. The \textit{constraints}, in particular, can be handled in various ways. One approach to enforcing the constraints is to replace the objective function by a merit function that penalizes the violation of the constraints \cite{Bertsekas1999, Boyd2009, Conn2000}. The merit function can also be built using an inner barrier, penalizing proximity to the boundary of $X$ and thus guaranteeing strict feasibility of the optimal design. In either case, good penalty parameters must be chosen to obtain an efficient algorithm. Current implementations use iterative approaches with increasing penalty parameters; see, for example, \cite{Plantenga2009}. If the constraints are expensive to evaluate, reduced-order models can be used to reduce the computational costs \cite{Agarwal2013, Fletcher2002, March2012}. Instead of using merit functions, the constraints can also be enforced via a Lagrange approach, which is often implemented in combination with sequential quadratic programming methods \cite{Bertsekas1999, Liang2008}. Alternatively, \cite{Fletcher2002, Fletcher2002a} introduce a filtering technique whose aim is to minimize the objective function and establish feasibility of the optimal point using a filter set of non-dominating pairs of objective values and constraint violations. There are also methods for nonlinear constrained optimization that do not rely on penalties or filters; for instance, \cite{Curtis2013, Gould2010} introduce a trust-funnel method that sequentially reduces 
the value of the objective function and reduces constraint violations by taking steps tangential and normal to the feasible set.

Despite this wide array of constraint-handling approaches, most require derivative information on the objective function and the constraints, a situation which we do not wish to pursue in this paper. We will consider settings in which we have access to the objective function and the constraints only as black-box evaluations. Our setting is therefore derivative-free: we assume that derivatives of the objective and constraints are either unavailable or computationally too expensive to obtain.
Moreover, we are interested in situations where we are not able to evaluate the objective function and the constraints exactly. Different methodologies have been proposed in this context, yet they typically assume increasing accuracy of the computations while approaching a critical point; see, e.g., \cite{Carter1991, Choi2000, Heinkenschloss2002}. In situations where the objective function and the constraints are only available as black-box evaluations, however, we may not be able to adjust or even bound the accuracy of these evaluations---i.e., the magnitude of the numerical errors or other perturbations to the functions $f$ and $c_i$.  This paper will therefore address the regime where we have neither control nor \textit{a priori} knowledge on the inexactness of function evaluations.

The development of derivative-free optimization methods began in the 1960s, when Hooke and Jeeves \cite{Hooke1961} and Nelder and Mead \cite{Nelder1965}, see also \cite{Spendley1962}, were among the first to propose local direct search methods which only require black-box evaluations of the objective function. Since then, many other direct search methods have been proposed. For unconstrained programming we refer to \cite{Torczon1991, Torczon1997} and the references therein. Some direct search methods have been applied in the context of inexact function evaluations: for instance, the implicit filtering method \cite{Bortz1998, Kelley1999}. 

There are also derivative-free algorithms for computing the solutions of constrained nonlinear programs. Jones et al.\ \cite{Jones1993}
proposed the DIRECT method for derivative-free optimization with box constraints and studied its convergence. Later, this method and its convergence theory were extended to general nonlinear constraints; see \cite{Finkel2004}. Other direct search methods for constrained optimization problems have been proposed in \cite{DiPillo2013, Liuzzi2006, Liuzzi2010}; these have been progressively generalized from local optimization with linear constraints, to local optimization with nonlinear constraints, and finally to global optimization. 
%We will compare the performance of the optimization code SDPEN from \cite{Liuzzi2010} with our proposed algorithm in Section \ref{sec:numerical_results}. 
Another direct search method is the mesh adaptive direct search method (MADS) as introduced in \cite{Abramson2006}; in MADS the constraints are treated by an extreme barrier \cite{Audet2008, Audet2006} or by a progressive barrier \cite{Audet2009} method.

A different class of derivative-free optimization methods approximate the objective function using local surrogate models. COBYLA \cite{Powell1994, Powell1998}, for instance, is a widely used algorithm based on linear models of the objective. It handles constraints using a penalty approach based on a linear approximation. In \cite{March2012}, a low-fidelity surrogate model for the constraints is used to build a merit function based on a quadratic penalty. However, a rigorous proof that these methods converge to solutions of \eqref{eq:general_opt_problem} is not available.  Surrogate models based on radial basis functions have also been successfully applied in a derivative-free trust region setting \cite{Regis2011, Regis2014, Wild2008, Wild2011}. A recent effort \cite{Regis2014} in this line of work discusses the derivative-free method COBRA for nonlinear constrained optimization, wherein a safety margin is added to the constraints. A prominent difference between COBRA and the present work lies in our adaptive handling of constraints via an inner boundary path, as opposed to the constant offset used in COBRA. A generalization of the trust-funnel method to derivative-free constrained optimization is presented in \cite{Sampaio2015}. Here, a key difference with the present work lies in our feasibility requirement for every trial step; the trust-funnel approach, on the other hand, allows for constraint violations which are reduced while approaching a critical point. We refer to \cite{Conn2009a, Lewis2000, Rios2013} for further overviews of derivative-free approaches.

The main contributions of this paper are threefold. First, we present the algorithm NOWPAC ({N}onlinear {O}ptimization {w}ith {P}ath-{A}ugmented {C}onstraints), which is based on a trust region framework. The algorithm introduces a new way of handling nonlinear black-box constraints using an \textit{inner boundary path}, which is an offset function to the constraints whose purpose is to locally convexify the feasible domain. Besides this convexification, the inner boundary path guides the next trial step to become feasible, and therefore acceptable according to our proposed algorithm. Second, we develop a rigorous proof of convergence of the intermediate points computed by NOWPAC to a first-order critical point. Third, we analyze the behavior of our algorithm in the presence of inexact evaluations of the objective function and the constraints. Specifically, we show that the magnitude of the errors must respect a particular asymptotic decay rate in order to guarantee convergence. Moreover, we provide an error indicator to detect corrupted evaluations in cases where the adjustment of the accuracy level is not possible. In the latter case we propose early termination of the optimization to avoid deterioration of the approximated optimal designs, and to save unnecessary evaluations of the objective function and constraints. Other than via early termination, our analysis does not attempt to reduce the impact of inexact evaluations, or to quantify the error in the approximated optimal designs that is due to inexact evaluations; we leave these developments to future work.

The remainder of the paper is organized as follows. Section \ref{sec:trust_region_algorithm} gives a brief introduction to the trust region methodology; for more details, the reader is referred to \cite{Conn1993, Conn2009, Conn2009a, Powell2002, Powell2003}. Section \ref{sec:NOWPAC} presents the algorithm NOWPAC. In Section \ref{sec:convergence_proof}, we prove the convergence of the intermediate points computed by NOWPAC to a first-order critical point. Thereafter, in Section \ref{sec:practical_aspects}, we discuss practical aspects of the implementation of our proposed algorithm. The proof of convergence is presented under the assumption of accurate evaluations of the objective function and the constraints. In practice, however, we may be faced with irreducible errors in the evaluations. Section \ref{sec:noisy_function_evaluations} thus discusses the behavior of NOWPAC in the latter setting, deriving the asymptotic bounds and error indicator described above. In Section \ref{sec:numerical_results} we numerically demonstrate NOWPAC using two test examples, the Schittkowski benchmark set \cite{Schittkowski1987, Schittkowski2008}, and a model of an industrial tar removal process used in converting biomass to liquid fuel. Concluding remarks and a sketch of future work are given in Section \ref{sec:conclusions}.

\section{The trust region framework}\label{sec:trust_region_algorithm} In this section we introduce the derivative-free trust region framework used to approximate the solution of (\ref{eq:general_opt_problem}). Trust region methods start from an initial point $x_0$ and compute a series of intermediate points $\{x_k\}_{k\in \mathbb{N}_0}$ that converge to a critical point $x^\ast$. To help compute $x_{k+1}$, trust region methods build surrogates of the objective function $f$ and the constraints $\{c_i\}_{i=1}^r$, denoted by $m_{x_k}^{f}$ and $\{m_{x_k}^{c_i}\}_{i=1}^r$ respectively, within a neighborhood of the current point $x_k$. The point $x_{k+1}$ is then determined from the surrogates as a point that suitably reduces the objective function while staying within the neighborhood of $x_k$ and satisfying the constraints. The neighborhood is called the trust region, $B(x_k, \rho_k) := \{ x \in \mathbb{R}^n \; : \; \|x - x_k\| \leq \rho_k\}$, with trust region radius $\rho_k$, $k \in \mathbb{N}_0$. Note that there are many possible choices of surrogate; among these, polynomial response surfaces \cite{Conn2009, Powell2000, Powell2002, Powell2004} are widely used. But other approximation methods can be employed as well; for example, radial basis functions are used to create the surrogates in \cite{Wild2011}. The particular choice of surrogate models $m_{x_k}^{f}$ and $\{m_{x_k}^{c_i}\}_{i=1}^r$ for the objective function and the constraints is beyond the scope of this work, and we will not go into details on how to compute them. In general, any surrogates that are twice continuously differentiable and that satisfy
\begin{subequations}\label{eq:fully_linear_equations}
\begin{align}
\Big| f(x + s) - m_x^f(x + s)\Big| & \leq \kappa_{f} \rho^{2} \label{eq:fully_linear_f}\\
\Big| c_i(x + s) - m_x^{c_i}(x + s)\Big| & \leq \kappa_{c} \rho^{2}\label{eq:fully_linear_c}\\
\Big\| \nabla f(x + s) - \nabla m_x^f(x + s)\Big\| & \leq \kappa_{df} \rho\label{eq:fully_linear_df}\\
\Big\| \nabla c_i(x + s) - \nabla m_x^{c_i}(x + s)\Big\| & \leq \kappa_{dc} \rho\label{eq:fully_linear_dc}
%\Big\| \nabla c_i(x_k + s) - \nabla m_k^{c_i}(x_k + s)\Big\| & \leq \kappa_{edc} \rho_k^{1-p}\label{eq:fully_linear_dc}
\end{align}
\end{subequations}
with constants $\kappa_{f}$,  $\kappa_{c}$, $\kappa_{df}$, $\kappa_{dc} > 0$, for all $x + s \in B(x, \rho)$, $i = 1 \ldots r$, are admissible. In our implementation of NOWPAC, we use quadratic minimum-Frobenius-norm surrogates; see \cite{Powell2004}. Models satisfying (\ref{eq:fully_linear_equations}) are called \textit{fully linear} within the trust region $B(x, \rho)$. All of the surrogates previously mentioned satisfy these conditions, if certain geometry conditions on the sampling points of the model are satisfied; see \cite{Conn2008, Powell2004, Wild2011}. With the surrogates denoted by $m_{x}^f(x + s)$ and $m_{x}^{c_i}(x + s)$, the corresponding gradients and Hessians at $s = 0$ are $g_x^f$, $g_x^{c_i}$ and $H_x^f$, $H_x^{c_i}$, respectively, for $i = 1 \ldots r$. Additionally, we assume the following:
\medskip
\begin{assumption}\label{ass:envelope_of_fully_linear_models}
For every constraint $i=1\ldots r$, there exists a bounding function 
\[
b_{c_i} \,:\,\left\{\begin{array}{lcl}
\mathbb{R}^n \times \mathbb{R}^{n} \times [0, 1] &\rightarrow &\mathbb{R}\\
(s; x, \rho) & \mapsto & b_{c_i}(s; x, \rho)
\end{array}\right.
\]
such that 
$
m_x^{c_i}(x+s) \leq c_i(x+s) + b_{c_i}(s; x, \rho)
$
for all $s \in B(0, \rho)$ and $(x, \rho) \in X \times [0,1]$. The bounding function is continuous in $(x, \rho)$ and satisfies $b_{c_i}(0, x, \rho) = 0$ as well as $b_{c_i} \leq \kappa_{\lambda_1} \rho$ for a constant $\kappa_{\lambda_1} > 0$ sufficiently large. At every $(x, \rho) \in X \times [0, 1]$, we also assume that $b_{c_i}(\,\cdot\,;x,\rho )$ is continuously differentiable with Lipschitz continuous gradient in $B(0, 1)$ and convex in $B(0, \rho)$.
\end{assumption}

\smallskip
Assumption~\ref{ass:envelope_of_fully_linear_models} is satisfied by many surrogate models. In particular, we point to Lemma \ref{lem:minFrobeniusnormModelsSatisfyAssumpion2.1}, where we explicitly construct the bounding function $b_c$ for quadratic minimum-Frobenius-norm models.

Before we state the trust region algorithm in the next section, we introduce a few general assumptions on the objective function and the constraints.
\medskip
\begin{assumption}\label{ass:general_assumptions}
The objective function $f$ and the constraints $\{c_i\}_{i=1}^r$ satisfy:
\begin{enumerate}
\item[(a)] $\mathcal{L} := X \cap \{x \in \mathbb{R}^n \; : \; f(x) \leq f(x_0)\}$ is compact, 
\item[(b)] $f$ and $c_i$ are continuously differentiable on $\mathcal{L}$ and have Lipschitz continuous gradients,
\item[(c)] $\left\| \nabla c_i(x) \right\| \geq \kappa_{bdc}$ for all $x$ on the boundary of $X$,
\item[(d)] at every critical point $x^\ast$ of $f$ in $X$ the Abadie constraint qualification condition, see \cite{Bazaraa2006}, holds,
\item[(e)] at any point on the boundary of the feasible domain $X$ with more than one active constraint, any two normals to the active constraints enclose an angle strictly less than $\pi$.
\end{enumerate}
\end{assumption}
\medskip

\begin{assumption}\label{ass:bound_on_model_Hessian}
There exists a constant $\kappa_{bh} > 0$ such that $\|H_x^f\| \leq \kappa_{bh}$ and $\|H_x^{c_i}\| \leq \kappa_{bh}$ for all $i = 1 \ldots r$.
\end{assumption}

\medskip

We note that Assumptions \ref{ass:general_assumptions}(a)--(d) ensure the existence of a solution of the optimization problem (\ref{eq:general_opt_problem}), which can be identified by first-order criticality conditions
%\todo{Is it okay to replace `criteria' with `conditions?' Or does this subtly change the meaning? More generally, do you mean that these assumptions ensure not only the existence of a solution in general, but rather the existence of a solution that can be identified by first-order criticality conditions using linearized constraints? Answer: Yes, conditions is okay. The Assumptions 2.2(a-d) ensure both, the existence of an optimum and that this solution can be identified using linearized constraints.}
% 
using linearized constraints. Assumption \ref{ass:general_assumptions}(e) excludes constraints that have the same tangent at a point on the boundary of the feasible domain, and hence excludes situations where two active constraints touch, as illustrated in Figure \ref{fig:exampleforinnerboundarypath} (right). We note that Assumption \ref{ass:general_assumptions}(e) in particular excludes equality constraints. 
We further remark that, due to the continuity of $\nabla f$ and $\{\nabla c_i\}_{i=1}^r$ in Assumption \ref{ass:general_assumptions}, we also have $\|\nabla f\| \leq \kappa_{bdf}$ and $\|\nabla c_i\| \leq \kappa_{ubdc}$, $i = 1 \ldots r$, for all $x$ in the compact set $\mathcal{L}$.

\section{The algorithm NOWPAC}\label{sec:NOWPAC}
In this section we introduce the derivative-free algorithm NOWPAC for approximating local critical solutions of the nonlinear constrained problem (\ref{eq:general_opt_problem}). The notation and basic structure follow closely along the lines of \cite{Conn1993, Conn2009a}; however we introduce significant changes in order to treat the nonlinear constraints as black-box evaluations. The outline of this section is as follows: In Section \ref{subsec:preliminaries} we introduce some necessary notation and state assumptions on the sufficient descent of the objective model in every trust region step. In Section \ref{subsec:thealgorithm} we describe NOWPAC itself as Algorithm \ref{alg:algorithm1}.

\subsection{Preliminaries}\label{subsec:preliminaries}

First we introduce an offset to the constraint function, the {\it inner boundary path},
\begin{equation}\label{eq:definition_of_inner_boundary_path}
h_{x}(x+d) :\begin{cases}
\;\mathbb{R}^n &\rightarrow \;\mathbb{R}\\
\;x+d &\mapsto \;\varepsilon_b\left\| d\right\|^{\frac{2}{1+p}},
\end{cases}
\end{equation}
with order reduction $p \in\: ]0, 1[$ and define the \textit{inner-boundary-path-augmented local feasible domain} at $x \in X$ as
\begin{equation}\label{eq:definition_of_Xxibp}
X_x^{ibp} := \left\{ x+d \;:\; c(x + d) + h_x(x + d) \leq 0 \right\} \cap B(x, 1).
\end{equation}
Note that the abbreviated notation $c$ represents the corresponding expressions for all $\{c_i\}_{i=1}^r$, e.g., $c(x) \leq 0$ means $c_i(x) \leq 0$ for all $i = 1 \ldots r$. Subsequently, we will use this abbreviation also for the surrogate models, i.e., $m_k^c(x) + h_k(x) \leq 0$. We illustrate the inner-boundary-path-augmented local feasible domains $X_x^{ibp}$ and the role of the inner boundary path $h_x$ in Figure \ref{fig:exampleforinnerboundarypath}, where the inner boundary path constant $\varepsilon_b > 0$ is chosen large enough to achieve a local convexification of $X_x^{ibp}$  (cf.\ Assumption~\ref{ass:convexification_through_ibp} below).
\begin{figure}[t]
\psfrag{T1}[c][cb][.8]{$x_{k_1}$}
\psfrag{T2}[c][t][.8]{$X_{x_{k_1}}^{ibp}$}
\psfrag{T3}[c][cb][.8]{$x_{k_2}$}
\psfrag{T4}[c][c][.8]{$X_{x_{k_2}}^{ibp}$}
\psfrag{T5}[c][c][.8]{$X$}
\begin{center}
\includegraphics[bb= 30 90 540 370, scale=.48]{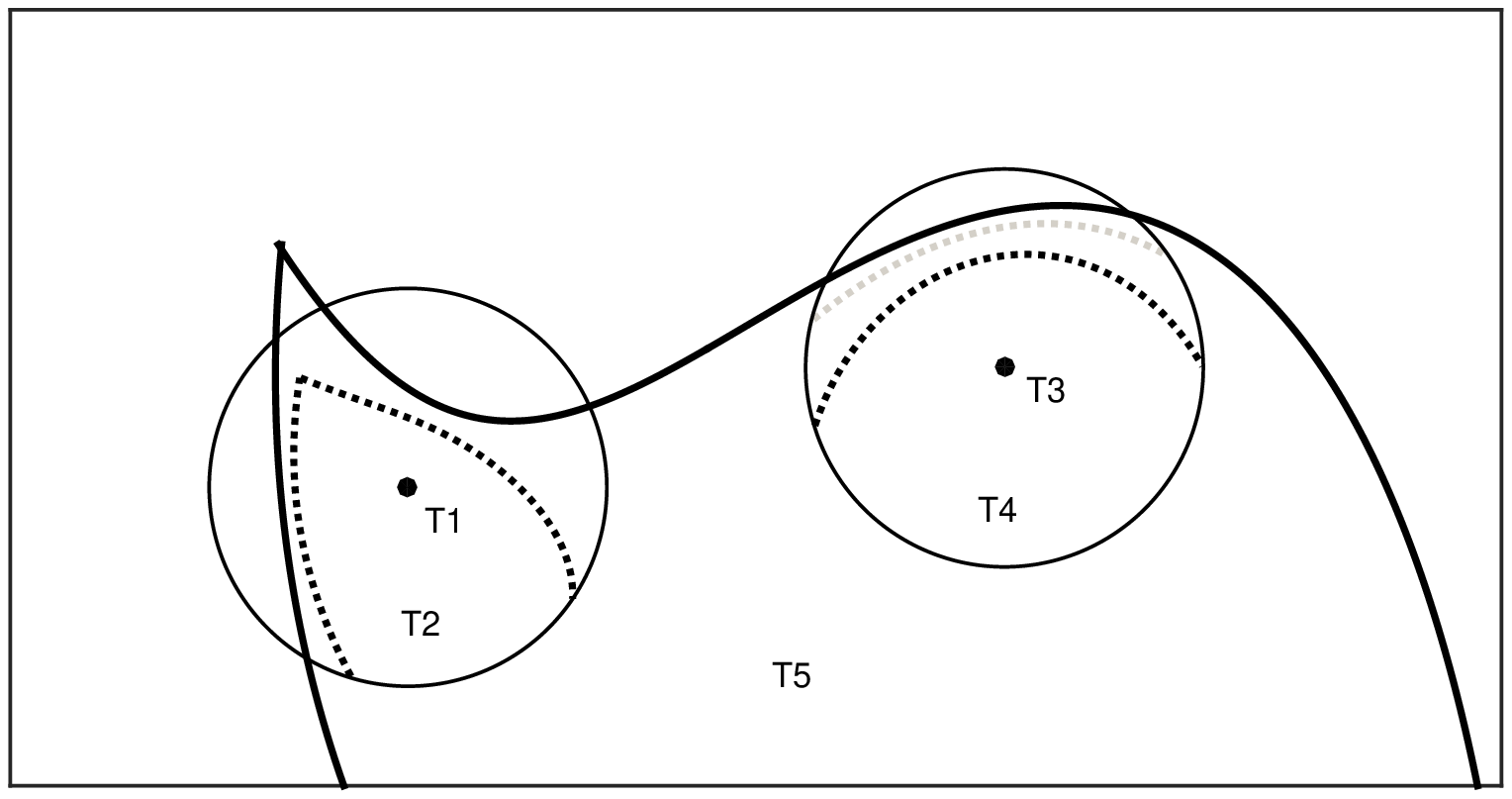}
\psfrag{T1}[c][c][.8]{$X$}
\psfrag{T2}[c][lb][.8]{$x$}
\psfrag{T3}[c][rb][.8]{$x_{k}$}
\psfrag{T4}[c][c][.75]{$X_{x_k}^{ibp}$}
\includegraphics[bb=20 22 530 420, scale=0.315]{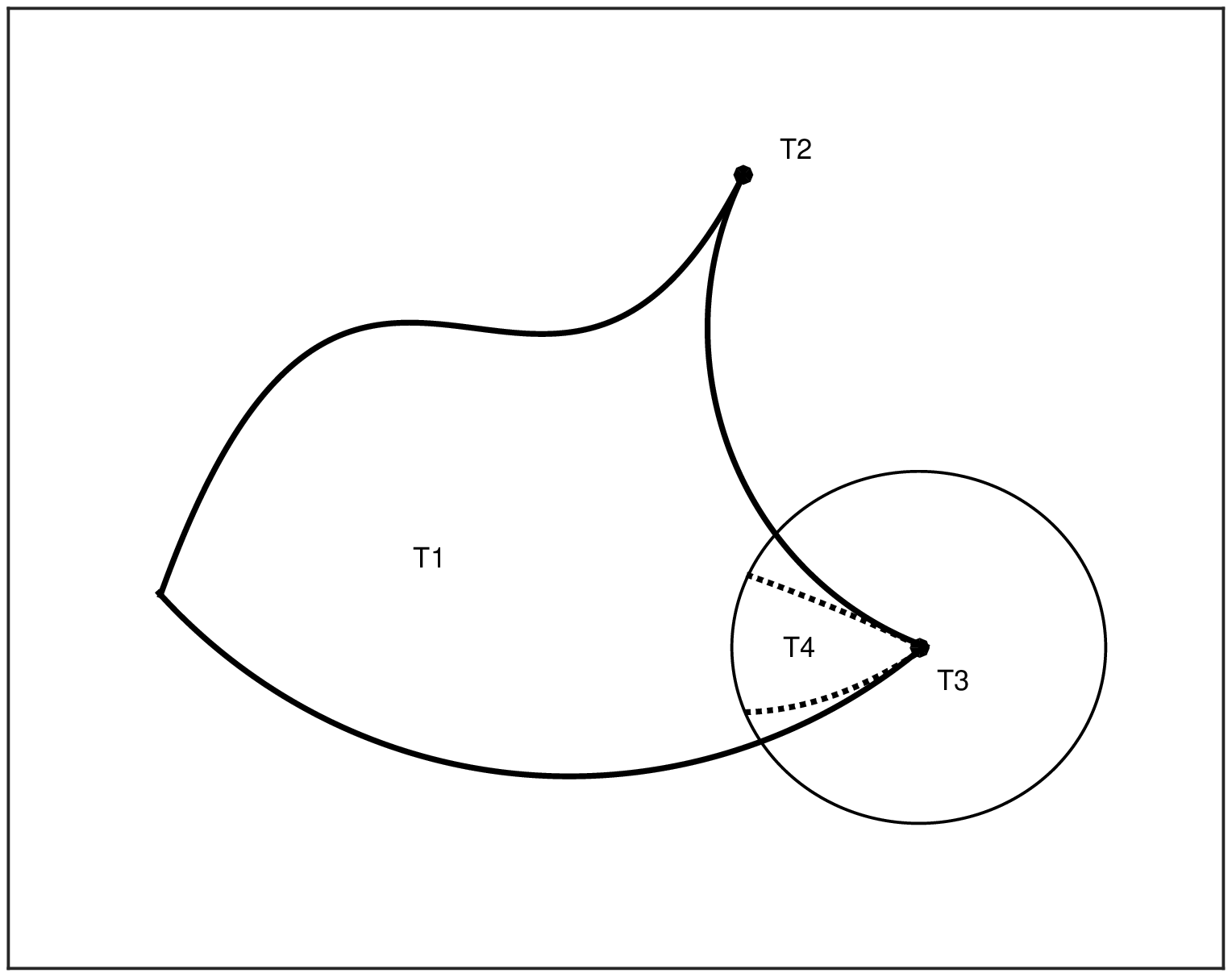}
\caption{Left: Local convexification (black dotted lines) of the feasible domain $X$ (solid line) around two designs $x_{k_1}$ and $x_{k_2}$. Here an inner boundary path constant $\varepsilon_b = 10$ with $p = 0.2$ is chosen to overcome the negative curvature, i.e., the concavity, of the exact constraints $c$. The circle represents a ball around $x_k$ with radius $1$. The gray dotted line represents the local convexification resulting from a scaled inner boundary path constant $\varepsilon_b = 10 \left(\frac12\right)^2$, cf.\ \eqref{eq:adjusted_inner_boundary_constant}. Right: Example of a feasible domain excluded by Assumption \ref{ass:general_assumptions}(e). Point $x$ is not admissible since the interior of the associated inner-boundary-path-augmented feasible domain is empty due to the touching constraints.
}\label{fig:exampleforinnerboundarypath}
\end{center}
\end{figure}
The importance of adding the inner boundary path $h_x$ to $m_x^c$ will become obvious in the discussion of the convergence of Algorithm \ref{alg:algorithm1} in Section \ref{sec:convergence_proof}. To quickly motivate the inner boundary path, however, we remark that it serves mainly two purposes: first, it locally convexifies the constraints around the point $x \in X$, as described in Assumption \ref{ass:convexification_through_ibp}; second, it helps push all iterates away from the boundary and towards the inner part of the feasible domain $X$. We only use one common inner boundary path constant $\varepsilon_b$ for all constraints to keep notation simple, but it is straightforward to extend the subsequent analysis to a set $\{\varepsilon_b^i\}_{i=1}^r$ of individual inner boundary path constants for each constraint $c_i$, $i = 1  \ldots r$.

Next we define the extended model
\begin{equation}\label{eq:cont_convex_extension_of_constraint_models}
M_x^c(x) := \begin{cases} 
m_x^c(x) & \mbox{if}\; x \in B(x, \rho)\\
\bar{m}_x^c(x) & \mbox{if}\; x \in B(x, 1) \backslash B(x, \rho),
\end{cases}
\end{equation}
where $\bar{m}_x^c$ is a smooth extension of the local surrogates $m_x^c$ to the ball $B(x, 1)$, if $\rho < 1$. If $\rho \geq 1$, then $M_x^c$ is simply the fully linear surrogate model $m_x^c$ within $B(x, \rho)$. 
Based on the extended surrogate model (\ref{eq:cont_convex_extension_of_constraint_models}) we define the \textit{approximated feasible domain}
\begin{equation}\label{eq:extended_approx_ibp_augmented_local_feasible_domain}
X_x := \left\{ x + d \;:\; M_x^c(x + d) + h_{x}(x + d) \leq 0\right\} \cap B(x, \max\{\rho, 1\}).
\end{equation} 
In order to have a reasonable approximation of $X_x^{ibp}$ by the approximated feasible domain $X_x \, \cap \, B(x, 1)$ for vanishing trust region radii $\rho \rightarrow 0$, we assume that $M_x^c$ is continuously differentiable with Lipschitz continuous gradient in $B(x, 1)$ with $|\bar{m}_x^c(x) - c(x)| \leq \kappa_{\lambda_2} \rho_k$ for a sufficiently large constant $\kappa_{\lambda_2} > 0$.
Finally we define the {\it inner and outer approximation sets}
\begin{align}\label{eq:extended_approx_ibp_augmented_local_feasible_domain_plus}
X_{(x, \rho)}^{+} &:= \left\{ x + d \;:\; c(x + d) + h_x(x + d) + b_c(d; x, \rho) \leq 0\right\} \cap B(x, 1),\\
X_{(x,\rho)}^{-} &:= \left\{ x + d \;:\; c(x + d) + h_x(x + d) - \kappa_{\lambda_3}\rho \quad\;\;\;\; \leq 0\right\} \cap B(x, 1),\nonumber
\end{align}
where $\kappa_{\lambda_3} > 0$ is chosen large enough such that $X_{(x, \rho)}^+ \subseteq X_x \subseteq X_{(x, \rho)}^-$, justifying the labels of `inner' and `outer' approximation. We will make use of this inclusion of the sets  $X_{(x, \rho)}^+$, $X_x$, $X_{(x, \rho)}^-$ in Lemma~\ref{lem:quality_of_approx_criticality_measure}. Note that the inner and outer approximations $X_{(x, \rho)}^+$ and $X_{(x, \rho)}^-$ to the approximated feasible domain $X_x$ are defined by lower and upper bounds on \textit{all} possible local surrogate models and are therefore independent of the particular surrogate model that defines $X_x$.
Henceforth we will use a single subscript $k$ for any of the quantities introduced above when referring to a particular point $x_k \in X$, i.e., $X_k^{ibp} := X_{x_k}^{ibp}$, $X_k := X_{x_k}$, as well as $m_k^f := m_{x_k}^f$ and $m_k^c := m_{x_k}^c$, etc.

\medskip

\begin{assumption}\label{ass:convexification_through_ibp}
Assume that $\varepsilon_b > 0$ is large enough such that the sets $X_x^{ibp}$ and $X_k$ and the inner and outer approximations $X_{(x,\rho)}^{\pm}$ are strictly convex.
\end{assumption}

\smallskip

Lemma \ref{lem:local_convexification} shows the existence of an inner boundary path constant $\varepsilon_b$ such that Assumption~\ref{ass:convexification_through_ibp} is satisfied.

Before we state the algorithm we have to define a measure for criticality, i.e., a measure for the proximity of the current iterate $x_k$ to a critical point: 
\begin{equation}\label{eq:approx_criticality_measure}
\alpha_k(\rho_k) := \frac{1}{\rho_k} \left| \min\limits_{\substack{x_k+d \in X_k\\ \|d\| \leq \rho_k}} \left\langle g_k^f, d\right\rangle\right|.
\end{equation}
Since $\alpha_k$ uses models for both the objective and the constraints, we call it the \textit{approximated criticality measure}.  We refer to Section~\ref{sec:crit_measure} for a detailed discussion of $\alpha_k$.
Besides a measure for criticality, we also need to specify an initial point $x_0$ and a maximal trust region radius $\rho_{\max} > 0$. Moreover, we set the initial trust region radius to $\rho_0 \in \;]\rho_{\min}, \rho_{\max}]$, $\rho_{\min} > 0$, and specify technical parameters as shown in Table \ref{tab:technical_parameters}.
\begin{table}%[!htb]
\caption{Technical parameters and their admissible ranges as used in Algorithm~\ref{alg:algorithm1}; see also Sections \ref{sec:practical_aspects} and \ref{sec:numerical_results}. $^{(\star)}$Additional restrictions: $\eta_1 > 0$, $\varepsilon_b$ large enough (see Assumption \ref{ass:convexification_through_ibp}).}\label{tab:technical_parameters}
\begin{center}\footnotesize
\renewcommand{\arraystretch}{1.3}
\begin{tabular}{|c|c|c|}\hline
symbol & range & description\\\hline\hline
$\varepsilon_b$ & $]0, \infty[$ & inner boundary path constant$^{(\star)}$\\
$\eta_0$        & $[0, 1[$      & step rejection parameter\\
$\eta_1$        & $[\eta_0, 1[$ & step acceptance parameter$^{(\star)}$\\
$\gamma_{inc}$  & $]1, \infty[$ & increment factor for trust region\\
$\gamma$        & $]0, 1[$      & decrease factor for trust region\\
$\varepsilon_c$ & $]0, \infty[$ & lower bound on trusted criticality measure\\
$\omega$        & $]0, 1[$      & decrease factor for trust region\\
$\mu$           & $]0, \infty[$ & factor for bound on trust region radius by $\alpha$\\\hline
\end{tabular}
\end{center}
\end{table}
As already mentioned in Section~\ref{sec:trust_region_algorithm}, our trust region algorithm computes, starting from the current iterate $x_k$, a trial step $x_k + s_k \in X_k \cap B(x_k, \rho_k)$. For the computation of the trial step $s_k$ we impose the following assumptions:
\smallskip
\begin{assumption}\label{ass:stepassumptions}
The trial step $s_k$ computed by Algorithm \ref{alg:algorithm1} satisfies
\begin{enumerate}
\item[(a)] $m_k^f(x_k) - m_k^f(x_k + s_k) \geq \mu_1 \alpha_k(\rho_k)\rho_k$, \ \text{and} 
\item[(b)] $\|s_k\| \geq \min\left\{ \mu_2 \rho_k^{1+q}, \mu_3\right\}$
\end{enumerate}
for $q < p$, $\mu_1, \mu_2, \mu_3 \in \;]0, 1]$, as well as the feasibility and trust region condition $x_k + s_k \in X_k \cap B(x_k, \rho_k)$.
\end{assumption}

\smallskip
The first assumption ensures that the step yields a sufficient descent in the objective model, whereas the second assumption keeps the step sizes from becoming too small. Note that the parameter $p$ in Assumption \ref{ass:stepassumptions} is the same as in the definition of the inner boundary path  (\ref{eq:definition_of_inner_boundary_path}).
We defer discussion of the existence of a step $s_k$ satisfying Assumptions \ref{ass:stepassumptions} to Section \ref{sec:practical_aspects}.

\subsection{The algorithm}\label{subsec:thealgorithm}
We present the complete derivative-free trust region algorithm NOWPAC in Algorithm~\ref{alg:algorithm1}.
% For an in-depth discussion of the general framework we refer to \cite{Conn2009a}.
%
\begin{algorithm}[!h]
\DontPrintSemicolon
Construct the initial fully linear models $m_0^{f}(x_0 + s)$, $m_0^{c}(x_0 + s)$.\;

\For{k = 0, 1, \ldots}{\hrule
\tcc{STEP 1: Criticality step}
\If{$\alpha_k(\rho_k) \leq \varepsilon_c$}
{
%Call algorithm \ref{alg:algorithm2} to certify full linearity of $m_k^{f, \mbox{icb}}$ and $m_k^{c, \mbox{icb}}$\;
\If{$m_k^{f}$ and $m_k^{c}$ are not fully linear in $B\left(x_k, \rho_k\right)$ or $\rho_k > \mu \alpha_k(\rho_k)^{\frac1{q}}$\vspace{1mm}}
{
Set $\rho_k = \omega \rho_k$\;\vspace{-0.5mm}
Construct fully linear models $m_k^f$ and $m_k^c$\;
Go to line {\tt 4}\;
}
}

\hrule\tcc{STEP 2: Step calculation}
Compute a trial step $s_k$ that satisfies Assumptions \ref{ass:stepassumptions}\;

\hrule\tcc{STEP 3: Check feasibility of trial point}
\If{$c(x_k + s_k) > 0$}
{
Set $\rho_k = \gamma \rho_k$ and update $m_k^f$ and $m_k^c$ accordingly to obtain fully linear models\;
%Go to line {\tt 10}
Go to {\tt STEP 1}
}

\hrule\tcc{STEP 4: Check acceptance of trial point}
Compute $r_k = \frac{f(x_k) - f(x_k + s_k)}{m_k^f(x_k) - m_k^f(x_k + s_k)}$\;
\eIf{$r_k \geq \eta_0$}
{
Set $x_{k+1} = x_k + s_k$\;
Include $x_{k+1}$ into the node set and update the models to $m_{k+1}^{f}$ and $m_{k+1}^{c}$\;
}{
Set $x_{k+1} = x_k$, $m_{k+1}^{f} = m_k^f$ and $m_{k+1}^{c} = m_k^c$\;
}

\hrule\tcc{STEP 5: Trust region update}
Set 
$
\rho_{k+1} = \begin{cases}
\gamma_{inc}\rho_k & \mbox{if} \quad r_k \geq \eta_1\\
\rho_k             & \mbox{if} \quad \eta_0 \leq r_k < \eta_1,\\
\gamma\rho_k       & \mbox{if} \quad r_k < \eta_0.\\
\end{cases}
$\vspace{2mm}\;

\hrule\tcc{STEP 6: Model improvement}
\If{$r_k < \eta_0$}{
Improve the quality of the models $m_{k+1}^{f}$ and $m_{k+1}^{c}$\;} 
\hrule
}

\caption{Nonlinear Optimization With Path-Augmented Constraints\label{alg:algorithm1}}
\end{algorithm}
We call iteration $k$ {\it successful} if the acceptance ratio $r_k$ exceeds the threshold $\eta_1$, whereas we call it {\it acceptable} if $\eta_0 \leq r_k < \eta_1$. Note that we do not specify a stopping criterion to terminate the algorithm. The usual approach in derivative-free trust region algorithms is to stop whenever the trust region radius falls below a prescribed threshold $\rho_{\min} > 0$. We will see in Section~\ref{subsec:conv_to_critical_point} that this is a reasonable stopping criterion for NOWPAC as well, since $\rho_k \rightarrow 0$ for the sequence $\{x_k\}_k$ converging to a critical point $x^\ast$. We therefore insert the line\\

$\,${\tt {\scriptsize \bf 22a}}$\;| \;\;$  {\bf if} $\rho_{k+1} < \rho_{min}$ {\bf then} {\it stop}.\\

\noindent into Algorithm \ref{alg:algorithm1} below for the actual implementation of NOWPAC. However, since we want to examine the asymptotic behavior of the iterates as $k \rightarrow \infty$, we do not include a stopping criterion in the forthcoming theoretical investigations. For more practical aspects of the implementation of NOWPAC see Section \ref{sec:practical_aspects}.

Finally we note that the construction of fully linear models in lines 6, 18, and 24 of Algorithm \ref{alg:algorithm1} can be completed in a finite number of steps. In particular, the model improvement {\tt STEP 6} computes at most finitely many intermediate points before the models become fully linear; see \cite{Conn2009a}.

\section{Convergence to first-order critical points}\label{sec:convergence_proof}
In this section we prove convergence of the intermediate points $\{x_k\}_k$ generated by Algorithm \ref{alg:algorithm1} to a first-order critical point $x^\ast$ of (\ref{eq:general_opt_problem}). We subdivide this section into three parts. First, in Section \ref{sec:crit_measure}, we show that if the approximated criticality measure $\alpha_k(\rho_k)$ evaluated at intermediate points $x_k$ vanishes to zero, then $x^\ast = \lim_{k\rightarrow \infty} x_k$ is a first-order critical point. For convergence, we then have to prove that Algorithm \ref{alg:algorithm1} indeed generates a sequence of intermediate points on which the approximated criticality measure converges to zero. In Section \ref{sec:successful_steps}, we show that Algorithm \ref{alg:algorithm1} computes $\{x_k\}_k$ without getting trapped in an infinite loop of infeasible or rejected steps in {\tt STEP~3} and {\tt STEP~4}. Thereafter, in Section \ref{subsec:conv_to_critical_point}, we complete the proof of convergence by showing that the sequence of intermediate points generated by Algorithm \ref{alg:algorithm1} converges to a first-order critical point.
The general ideas within Section \ref{sec:convergence_proof} follow along the lines of \cite{Conn1993}. But we develop additional arguments in order to show convergence in the case of approximated constraint functions. 

%In Section \ref{sec:successful_steps} we show that if the trust region radius $\rho_k$ is sufficiently small, then Alrorithm \ref{alg:algorithm1} accepts the trial step $s_k$. Having shown that the algorithm does not get stuck in an infinite loop of rejected steps, we prove convergence of the intermediate points $\{x_k\}_k$ to a critical point in Section \ref{subsec:conv_to_critical_point}.

\subsection{The criticality measure}\label{sec:crit_measure}
%In this section we discuss the criticality measure \ref{eq:approx_criticality_measure} used in algorithm \ref{alg:algorithm1}. 
When an optimal point $x^\ast$ is located at the boundary of the feasible set $X$, it is well known that the gradient is not necessarily an appropriate indicator for criticality. We therefore rely on the fact that $x^\ast \in X$ is a critical point if and only if
\begin{equation}\label{eq:criticality_criteria}
 -\nabla f(x^\ast) \in N(x^\ast),
\end{equation}
where $N(x) := \{ y \in \mathbb{R}^n \, : \, \left\langle y, u - x\right\rangle \leq 0, \; \forall u \in X\}$ denotes the normal cone of the set of feasible points $X$ at point $x$; due to Assumption \ref{ass:general_assumptions}(d) we can use the linearized constraints at $x^\ast$ to characterize the normal cone. Moreover, note that $N(x) = \{0\}$ whenever $x$ is an inner point of $X$ and (\ref{eq:criticality_criteria}) reduces to the gradient criterion $\|\nabla f(x^\ast)\| = 0$ for local first-order optimality. We now define the exact criticality measure,
\begin{equation}\label{eq:exact_criticality_measure}
 \mathfrak{A}[x] := \left| \min\limits_{x+d \in X_x^{ibp}} \left\langle \nabla f(x), d \right\rangle\right|,
\end{equation}
which gives the maximal possible decrease of the linearized objective function within $X_x^{ibp}$, i.e., the inner-boundary-path-augmented local feasible domain \eqref{eq:definition_of_Xxibp}. Note that the criticality subproblem (\ref{eq:exact_criticality_measure}) is linear in $d$ with $d=0$ always a feasible point, since $x \in X_x^{ibp}$ by definition. Thus the optimal value of the criticality subproblem, $\min_{x+d \in X_x^{ibp}} \left\langle \nabla f(x), d \right\rangle$, is necessarily less than or equal to zero and we can write 
\[
\mathfrak{A}[x] = -\left\langle \nabla f(x), d^\ast \right\rangle
\]
for $d^\ast := \argmin_{x+d \in X_x^{ibp}} \left\langle \nabla f(x), d \right\rangle$. We will frequently use this relation between the absolute value and the negation of the optimal value of the criticality subproblem within the subsequent proofs.

The following lemma shows that $\mathfrak{A}[x] = 0$ is a reasonable indicator for criticality of the point $x \in X$.

\medskip
\begin{lemma}\label{lem:exact_criticality_measure}
Under Assumptions \ref{ass:general_assumptions}, the point $x \in X$ is critical if and only if $\mathfrak{A}[x] = 0$.
\end{lemma}

\begin{proof} 
First note that due to Assumptions \ref{ass:general_assumptions}(c, e) and the inner boundary path $h_x( \cdot )$ vanishing superlinearly (with exponent $\frac{2}{1+p}$) around the center point, the interior of $X_x^{ibp}$ is always non-empty, as illustrated in Figure \ref{fig:exampleforinnerboundarypath} (right). 
Thus, because of the convexity of $X_x^{ibp}$, the Slater condition holds for $X_x^{ibp}$. Now observe that the gradient of $h_x(x)$ is zero, which means that the normal cones to $X$ and $X^{ibp}_x$ at $x$ are identical. The criticality conditions for the criticality subproblem in (\ref{eq:exact_criticality_measure}) can be expressed as
\begin{equation}\label{eq:opt_conditions_for_exact_criticality_measure}
\begin{split}
0 \in 2 \lambda d + \nabla f(x) + N(x + d),
\quad x + d \in X_x^{ibp},\\ \|d\| \leq 1 \quad \mbox{and} \quad \lambda(d^Td - 1) = 0.\qquad\qquad
\end{split}
\end{equation}
We see that if $\mathfrak{A}[x] = 0$, then $d = 0$ is an optimal solution to the criticality subproblem in (\ref{eq:exact_criticality_measure}). Inserting $d = 0$ into (\ref{eq:opt_conditions_for_exact_criticality_measure}) yields  (\ref{eq:criticality_criteria}), i.e., $-\nabla f(x) \in N(x)$. On the other hand, if (\ref{eq:criticality_criteria}) holds, then the above conditions  (\ref{eq:opt_conditions_for_exact_criticality_measure}) are satisfied with $d = 0$ and $\lambda = 0$, implying $\mathfrak{A}[x] = 0$.
\end{proof}
\medskip

At this point the exact criticality measure $\mathfrak{A}[x_k]$ depends on the gradient $\nabla f$ of the objective function and the exact constraint functions $c$. In the context of derivative-free optimization we know neither the gradient nor the exact algebraic structure of the constraint functions. Thus, we are not able to evaluate the exact criticality measure and must modify it in order to obtain a criticality measure that we can evaluate. Accordingly, we replace $X_k^{ibp}$ with $X_k$ and substitute the model gradient $g_k^f$ for the exact gradient $\nabla f(x_k)$. The result is the approximated criticality measure $\alpha_k(\rho_k)$ given in (\ref{eq:approx_criticality_measure}) and used in Algorithm~\ref{alg:algorithm1}. At this point it remains to show that the approximated criticality measure serves our need to drive iterates of the algorithm to a critical point of (\ref{eq:general_opt_problem}). We address this problem with the following lemma.

\medskip
\begin{lemma}[Relation between the exact and approximated criticality measures]\label{lem:quality_of_approx_criticality_measure}
Under Assumption \ref{ass:general_assumptions}, let $\{x_k\}_k \in X$ be a sequence of points in the feasible domain X and $\{\rho_k\}_k$ a sequence of trust region radii with $\lim_{k\rightarrow \infty} \rho_k = 0$. It holds that
\[
\lim\limits_{k \rightarrow \infty} \alpha_{k}(\rho_{k}) = 0 \quad \Rightarrow \quad\lim\limits_{k\rightarrow \infty} \mathfrak{A}[x_{k}] = 0.
\]
\end{lemma}

\begin{proof}
Due to $\lim_{k \rightarrow \infty}\alpha_{k}(\rho_{k}) = 0$ there exists an index $\hat{k}$ such that $\alpha_{k}(\rho_{k}) \leq \varepsilon_c$ for all $k \geq \hat{k}$. In this case {\tt STEP 1} of Algorithm \ref{alg:algorithm1} ensures that the models $m_{k}^f$ and $m_{k}^c$ are fully linear on $B(x_{k}, \rho_{k})$ with $\rho_{k} \leq \mu\alpha_{k}(\rho_{k})^{\frac1{p}}$ for all $k \geq \hat{k}$, in particular yielding $\lim_{ k \rightarrow \infty} \rho_{k} = 0$. Since we are only interested in the asymptotic behavior of the criticality measures, we assume without loss of generality that $\rho_k < 1$.

First, define the intermediate criticality measures
\[
\mathfrak{A}_1[x_k] := \left|\min\limits_{\substack{x_k + s \in X_k \\ \|s\|\leq 1}} \left\langle \nabla f(x_k), s \right\rangle\right|\quad \mbox{and} \quad
\mathfrak{A}_2[x_k] := \left|\min\limits_{\substack{x_k + s \in X_k \\ \|s\|\leq 1}} \left\langle g_k^f, s \right\rangle\right|.
\]
Note that the difference between $\mathfrak{A}_2[x_k]$ and $\rho_k\alpha_k(\rho_k)$ is that the former is constrained by $X_k$ rather than by $X_k \cap B(x_k, \rho_k)$. The difference between $\mathfrak{A}_2[x_k]$ and $\mathfrak{A}_1[x_k]$ is in the gradient of the criticality subproblem, and the difference between $\mathfrak{A}_1[x_k]$ and $\mathfrak{A}[x_k]$ lies in the introduction of the approximated feasible domain $X_k$. In order to prove the assertion of the lemma we use the triangle inequality, 
\[
\mathfrak{A}[x_k] \leq \left|\mathfrak{A}[x_k] - \mathfrak{A}_1[x_k]\right| + 
\left|\mathfrak{A}_1[x_k] - \mathfrak{A}_2[x_k]\right| + \mathfrak{A}_2[x_k],
\]
and show that each term on the right-hand side vanishes for decreasing trust region radius $\rho_k$ and for a vanishing approximated criticality measure $\alpha_k(\rho_k)$. For this we consider the combined sequence $\{(x_k, \rho_k)\}_k$ of intermediate points and trust region radii as computed by Algorithm \ref{alg:algorithm1}. Furthermore we define 
\[
\mathfrak{A}_{1}^{\pm}[x, \rho] := \left|\min\limits_{x + s \in X_{(x,\rho)}^{\pm}} \left\langle \nabla f(x), s \right\rangle\right|,
\]
which, according to Lemma \ref{lem:upper_lower_semicontinuity}, are continuous functions in $(x, \rho)$. Lemma \ref{lem:upper_lower_semicontinuity} establishes the continuity of $\mathfrak{A}_1^{\pm}$ on the whole domain $X \times [0, \rho_{max}]$, which then holds in particular for the path through $\{(x_k, \rho_k)\}_k$ as computed by Algorithm~\ref{alg:algorithm1}, as depicted in Figure~\ref{fig:rho_x_path}. 
\begin{figure}[t]
\psfrag{TagY}[cb][c][.8]{$x$}
\psfrag{TagX}[c][c][.8]{$\rho$}
\psfrag{T1}[c][c][.8]{$(x_0, \rho_0)$}
\psfrag{T2}[c][ct][.8]{$(x_k, \rho_k)$}
\psfrag{T3}[c][ct][.8]{$(x_{k+1}, \rho_{k+1})$}
\begin{center}
\includegraphics[bb= 0 90 560 420, scale=.48]{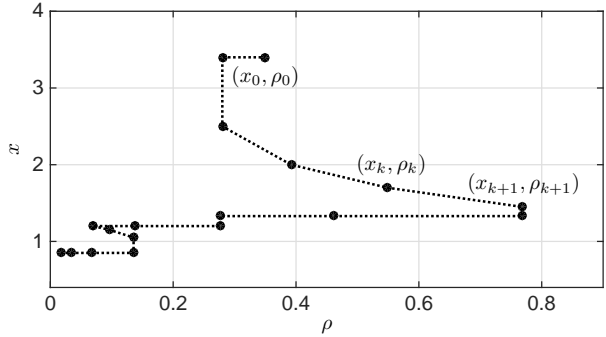}
\end{center}
\caption{Typical path in the domain $X \times [0, \rho_{max}]$ as computed by Algorithm~\ref{alg:algorithm1}.}\label{fig:rho_x_path}
\end{figure}
Note that $X_{(x_k, 0)}^{+} = X_{(x_k, 0)}^{-} = X_k^{ibp}$ for all $x_k \in X$, i.e., $\mathfrak{A}_1^{+}[x_k, \rho_k] = \mathfrak{A}_1^{-}[x_k, \rho_k] = \mathfrak{A}[x_k]$ for $\rho_k = 0$. It then follows from the continuity of $\mathfrak{A}_1^{+}$ and $\mathfrak{A}_1^{-}$ that for every $\varepsilon > 0$ there exists a $k_{\varepsilon}$ such that 
\[
|\mathfrak{A}_{1}^{+}[x_k, \rho_k] - \mathfrak{A}[x_k]| \leq \varepsilon \quad \mbox{and}\quad
|\mathfrak{A}_{1}^{-}[x_k, \rho_k] - \mathfrak{A}[x_k]| \leq \varepsilon
\]
for all $k \geq k_{\varepsilon}$. Due to $X_{(x_k, \rho_k)}^+ \subseteq X_{k} \subseteq X_{(x_k, \rho_k)}^-$ we have\footnote{It holds that $\min_{x_k+s \in X_{(x_k,\rho_k)}^+} \left\langle \nabla f(x_k), s \right\rangle
\geq
\min_{x_k+s \in X_k, \|s\|\leq1} \left\langle \nabla f(x_k), s \right\rangle
\geq 
\min_{x_k+s \in X_{(x_k,\rho_k)}^-} \left\langle \nabla f(x_k), s \right\rangle$. Multiplication with $-1$ yields $\mathfrak{A}_{1}^{+}[x_k, \rho_k] \leq \mathfrak{A}_{1}[x_k, \rho_k] \leq \mathfrak{A}_{1}^{-}[x_k, \rho_k]$.
}
\[
\mathfrak{A}_{1}^{+}[x_k, \rho_k] \leq \mathfrak{A}_{1}[x_k] \leq \mathfrak{A}_{1}^{-}[x_k, \rho_k].
\]
Subtracting $\mathfrak{A}[x_k]$ from this inequality yields
\[
-\varepsilon \leq \mathfrak{A}_{1}^{+}[x_k, \rho_k]-\mathfrak{A}[x_k] \leq \mathfrak{A}_{1}[x_k] -\mathfrak{A}[x_k]\leq \mathfrak{A}_{1}^{-}[x_k, \rho_k]-\mathfrak{A}[x_k] \leq \varepsilon,
\]
and thus
$
|\mathfrak{A}_{1}[x_k] -\mathfrak{A}[x_k]| \leq \varepsilon
$
for $k \geq k_{\varepsilon}$.

Next we derive an upper bound on the difference between the two intermediate criticality measures $\mathfrak{A}_1[x_k]$ and $\mathfrak{A}_2[x_k]$.
Since for $k \geq \hat{k}$ the model $m_k^f$ is fully linear within the trust region $B(x_k, \rho_k)$ and we have that $\|\nabla f(x) - \nabla m_k^f(x)\| \leq \kappa_{df}\rho_k$ for all $x \in B(x_k, \rho_k)$, it holds in particular that $\|\nabla f(x_k) - g_k^f\| \leq \kappa_{df}\rho_k$. We now follow along the lines of \cite[Lem.\ 3.5]{Conn1993} to show that $| \mathfrak{A}_1[x_k] - \mathfrak{A}_2[x_k]| \leq \kappa_{df} \rho_k$. Denote the solutions of the first and second intermediate criticality subproblems, $\mathfrak{A}_1[x_k]$ and $\mathfrak{A}_2[x_k]$, by
\[
s_k := \argmin\limits_{x_k + s \in X_k, \|s\|\leq1} \left\langle \nabla f(x_k), s \right\rangle \quad \mbox{and} \quad
\hat{s}_k := \argmin\limits_{x_k + s \in X_k, \|s\|\leq1} \left\langle g_k^f, s \right\rangle.
\]
Let us first assume the case $\mathfrak{A}_1[x_k] - \mathfrak{A}_2[x_k] \geq 0$. It follows that
\begin{align*}
0 \leq \mathfrak{A}_1[x_k] - \mathfrak{A}_2[x_k] &= \left\langle g_k^f, \hat{s}_k \right\rangle - \left\langle \nabla f(x_k), s_k \right\rangle\\
& = \left\langle g_k^f, \hat{s}_k\right\rangle - \left\langle g_k^f, s_k\right\rangle + \left\langle g_k^f - \nabla f(x_k), s_k\right\rangle\\
& \leq \left\langle g_k^f, \hat{s}_k\right\rangle - \left\langle g_k^f, s_k\right\rangle + \left\| g_k^f - \nabla f(x_k)\right\|\left\| s_k\right\|\\
& \leq \left\langle g_k^f, \hat{s}_k\right\rangle - \left\langle g_k^f, s_k\right\rangle + \kappa_{df}\rho_k,
\end{align*}
where we used the Cauchy-Schwarz inequality, the full linearity property (\ref{eq:fully_linear_df}), and the fact that $\|s_k\| \leq 1$. Noting that $\langle g_k^f, \hat{s}_k\rangle \leq \langle g_k^f, s_k\rangle$ (since $\hat{s}_k$ is a minimum of the intermediate criticality subproblem associated with $\mathfrak{A}_2[x_k]$) yields the bound $\mathfrak{A}_1[x_k] - \mathfrak{A}_2[x_k] \leq \kappa_{df} \rho_k$. The upper bound in the case where $\mathfrak{A}_1[x_k] - \mathfrak{A}_2[x_k] \leq 0$ can be shown analogously by replacing the first line in the above inequality chain with $0 \leq \mathfrak{A}_2[x_k] - \mathfrak{A}_1[x_k] = \langle \nabla f(x_k), s_k\rangle - \langle \nabla f(x_k), \hat{s}_k\rangle + \langle \nabla f(x_k) - g_k^f, \hat{s}_k\rangle$.

In order to complete the proof we refer to Lemma \ref{lem:big_A_and_approx_crit_lemma_relation} where we show that $\mathfrak{A}_2[x_k] \leq \alpha_k(\rho_k)$. Then the assertion of the lemma follows from the assumption that $\lim_{k\rightarrow \infty}\alpha_k(\rho_k) = 0$.
 \end{proof}

\subsection{Successful iterations}\label{sec:successful_steps} From the definition of Algorithm \ref{alg:algorithm1} we see that all intermediate points $x_k$ that are either not feasible or not acceptable are discharged within {\tt STEP 3} or {\tt STEP 4}. In this section, we show that Algorithm \ref{alg:algorithm1} does not get trapped in an infinite loop of discharged iterations that result in premature convergence to a potentially non-critical point. First we show that {\tt STEP 3} in Algorithm~\ref{alg:algorithm1} determines a feasible point if the trust region radius is sufficiently small. Then, we examine the second hurdle for a successful iteration: the step acceptance condition in {\tt STEP 4}.

\medskip

\begin{lemma}\label{lem:successful_step_3}
Let $m_k^c$ be fully linear on $B(x_k, \rho_k)$. {\tt STEP 3} in Algorithm~\ref{alg:algorithm1} yields a feasible trial step if 
\[
%\varepsilon_b \geq \kappa_c \mu_2^{-\frac{2}{1+p}}.
\rho_k \leq \left(\frac{\varepsilon_b}{\kappa_c}\mu_2^{\frac{2}{1+p}}\right)^{\frac{1+p}{2(p - q)}}.
\]
\end{lemma}
\begin{proof}
Since $x_k + s_k \in X_k \cap B(x_k, \rho_k)$ we know that $m_k^c(x_k + s_k) + h_k(x_k + s_k) \leq 0$, i.e.,
\begin{equation}\label{eq:bound_on_constraint_model_by_alphaepsb}
m_k^c(x_k + s_k) \leq -h_k(x_k + s_k) = -\varepsilon_b\|s_k\|^{\frac{2}{1+p}} \leq -\varepsilon_b\mu_2^{\frac{2}{1+p}} \rho_k^{2\frac{1 + q}{1+p}},
\end{equation}
where we used Assumption \ref{ass:stepassumptions}(b) in the last inequality. Moreover, from (\ref{eq:bound_on_constraint_model_by_alphaepsb}) and the fully linear property of $m_k^c$, we get
\[
c(x_k + s_k) \leq m_k^c(x_k + s_k) + \kappa_c \rho_k^{2} \leq -\varepsilon_b\mu_2^{\frac{2}{1+p}}\rho_k^{2\frac{1 + q}{1+p}} + \kappa_c \rho_k^{2}.
\]
Thus {\tt STEP 3} in Algorithm~\ref{alg:algorithm1} yields a feasible trial step, i.e., $c(x_k + s_k) \leq 0$, if
\[
-\varepsilon_b\mu_2^{\frac{2}{1+p}}\rho_k^{2\frac{1 + q}{1+p}} + \kappa_c \rho_k^{2} \leq 0.
\]
Solving for $\rho_k$ yields the assertion of the lemma.
\end{proof}

\medskip

The next lemma shows that if the trust region radius falls below a certain threshold  (given by the criticality measure and the threshold in Lemma \ref{lem:successful_step_3}), then the trial step will be accepted. In other words, if the current design is not a critical point, then we can always find a successful trial step.

\medskip

\begin{lemma}\label{lem:successful_step}
If $m_k^f$ and $m_k^c$ are fully linear on $B(x_k, \rho_k)$ and $\rho_k \leq A_k(\alpha_k)$, for
\[
A_k(\alpha_k) := \min\left\{ \frac{\mu_1 \alpha_k(\rho_k) (1-\eta_0)}{2\kappa_{f}}, \, \left(\frac{\varepsilon_b}{\kappa_c}\mu_2^{\frac{2}{1+p}}\right)^{\frac{1+p}{2(p - q)}} \right\},
\]
then iteration $k$ will be acceptable or successful.
\end{lemma}

\begin{proof}
Since the assumptions of Lemma \ref{lem:successful_step_3} are satisfied, Algorithm~\ref{alg:algorithm1} passes {\tt STEP 3} with a feasible trial step. Moreover, from Assumption \ref{ass:stepassumptions}(a) we know that
\[
m_k^f(x_k) - m_k^f(x_k + s_k) \geq \mu_1 \alpha_k(\rho_k) \rho_k.
\]
Using the fully linear properties of the model $m_k^f$ on $B(x_k, \rho_k)$ we get
\begin{align*}
\left| r_k - 1\right| &\leq
\left| \frac{f(x_k + s_k) - m_k^f(x_k + s_k)}{m_k^f(x_k) - m_k^f(x_k + s_k)}\right| +
\left| \frac{f(x_k) - m_k^f(x_k)}{m_k^f(x_k) - m_k^f(x_k + s_k)}\right|\\
&\leq \frac{2\kappa_{f} \rho_k^{2}}{\mu_1 \alpha_k(\rho_k) \rho_k}
\leq 1-\eta_0.
\end{align*}
Therefore $r_k \geq \eta_0$ and iteration $k$ is acceptable or successful.
\end{proof}

\subsection{Proof of convergence}\label{subsec:conv_to_critical_point}
Having ensured that Algorithm~\ref{alg:algorithm1} always finds an acceptable or successful feasible trial step, we now show convergence of the intermediate points $\{x_k\}_k$ to a first-order critical point $x^\ast$. Following the ideas in \cite{Conn1993, Conn2009a} we establish a relation between the trust region radii $\{\rho_k\}_k$ and the criticality measures $\{\mathfrak{A}[x_k]\}_k$; we reason that $\lim_{k \rightarrow \infty} \rho_k = 0$, from which we eventually conclude $\lim_{k\rightarrow \infty} \mathfrak{A}[x_k] = 0$. We start by proving the technical auxiliary Lemma \ref{lem:bounded_crit_measure_means_bounded_tr_radius} where we show that if the approximated criticality measures $\{\alpha_k\}_k$ are bounded from below by a positive constant, then the sequence $\{\rho_k\}_k$ of trust region radii will also be bounded from below by a positive constant, cf.\ \cite[Lem.\ 10.7]{Conn2009a}. 

\medskip

\begin{lemma}\label{lem:bounded_crit_measure_means_bounded_tr_radius}
Suppose that there exists a constant $\kappa_1 > 0$ such that $\alpha_k(\rho_k) \geq \kappa_1$ for all $k$. Then there exists a constant $\kappa_2 > 0$ such that $\rho_k \geq \kappa_2$ for all $k$.
\end{lemma}

\begin{proof}
By Lemma \ref{lem:successful_step} (note that {\tt STEP 1}, {\tt STEP 3,} and {\tt STEP 6} in Algorithm~\ref{alg:algorithm1} ensure full linearity of the models $m_k^f$ and $m_k^c$ after every reduction of the trust region radius) it holds that whenever $\rho_k$ falls below the value 
\begin{equation}\label{eq:in_proof_lem_bounded_crit_measure_means_bounded_tr_radius1}
\bar{\kappa}_2 = A_k(\kappa_1),
\end{equation}
the $k$th iteration is either acceptable or successful, and hence it holds that $\rho_{k+1} \geq \rho_k$. We conclude from  (\ref{eq:in_proof_lem_bounded_crit_measure_means_bounded_tr_radius1}) and the rules of {\tt STEPS 1},  {\tt 3}, and {\tt 5} that 
$\rho_k \geq \min\{\mu\omega\kappa_1^{\frac1{q}}, \gamma\bar{\kappa}_2\} =: \kappa_2$.
%$\rho_k \geq \min\{\mu\omega\kappa_1^{\frac1{q}},$ $\gamma\kappa_1, \rho_0, \gamma\bar{\kappa}_2\} =: \kappa_2$.
\end{proof}

\medskip

For notational convenience we denote the set of indices of all acceptable or successful steps by $\mathcal{S}$. In the next lemma we show convergence of Algorithm~\ref{alg:algorithm1} to a first-order critical point if $|\mathcal{S}| < \infty$ (i.e., if there are only finitely many acceptable or successful steps).

\medskip

\begin{lemma}\label{lem:convergence_finite_number_of_successful_steps}
If $|\mathcal{S}| < \infty$, then $\lim\limits_{k \rightarrow \infty} \mathfrak{A}[x_k] = 0$.
\end{lemma}

\begin{proof}
First we show that $\lim_{k \rightarrow \infty}\rho_k = 0$ if $|\mathcal{S}| < \infty$. For this we note that {\tt STEP 6} in  Algorithm \ref{alg:algorithm1} ensures full linearity of the models for the objective function and the constraints within every iteration after the last acceptable or successful iteration. Therefore, after the last acceptable or successful step, Algorithm \ref{alg:algorithm1} never increases the trust region radius $\rho_k$ but reduces it either by a factor of some power of $\omega$ in {\tt STEP 1} or by a factor of $\gamma$ in {\tt STEP 3} or {\tt STEP 5}. It follows that $\lim_{k \rightarrow \infty}\rho_k = 0$.
This in turn implies that $\lim_{k\rightarrow \infty} \alpha_k(\rho_k) = 0$; if the approximate criticality measures were bounded away from zero, then Lemma \ref{lem:successful_step}, for small $\rho_{k+1}$, guarantees that step $k+1$ is either acceptable or successful, yielding a contradiction to $|\mathcal{S}| < \infty$. The assertion of this lemma now follows from Lemma \ref{lem:quality_of_approx_criticality_measure}.
\end{proof}

\medskip

Thus far we have proved the convergence of Algorithm \ref{alg:algorithm1} in the case of $|\mathcal{S}| < \infty$. In the remainder of this section we extend the proof of convergence to $\mathcal{S}$ being a countably infinite set. To this end, we first show that the sequence of trust region radii $\{\rho_k\}$ converges to zero even if $\mathcal{S}$ is infinite in Lemma \ref{lem:tr_region_radius_converges_to_zero}. This immediately implies the existence of a subsequence of approximated criticality measures $\{\alpha_k(\rho_k)\}_k$ that converges to zero. Finally we combine all results to prove the convergence of Algorithm \ref{alg:algorithm1} towards a first-order critical point in Theorem \ref{thm:overall_convergence_of_exact_criticality meausre}.

\medskip

\begin{lemma}\label{lem:tr_region_radius_converges_to_zero}
It holds that $\lim\limits_{k \rightarrow \infty} \rho_k = 0$.
\end{lemma}

\begin{proof}
The proof follows closely along the lines of \cite[Lem.\ 10.9]{Conn2009a}. First, note that if $|\mathcal{S}| < \infty$, the assertion follows from the first part of the proof of Lemma \ref{lem:convergence_finite_number_of_successful_steps}. So henceforth we assume that $\mathcal{S}$ is a countably infinite set. For every $k \in \mathcal{S}$ we have
\[
f(x_k) - f(x_{k+1}) \geq \eta_0\left(m_k^f(x_k) - m_k^f(x_{k+1})\right) 
 \geq \eta_0 \mu_1 \alpha_k(\rho_k)\rho_k,
\]
where we used Assumption \ref{ass:stepassumptions}(a) in the second inequality. Due to {\tt STEP 1} in Algorithm~\ref{alg:algorithm1} we have $\alpha_k(\rho_k) \geq \min\{\varepsilon_c,\, \mu^{-q}\rho_k^{q}\}$, yielding
\begin{equation}\label{eq:rhs_rho_goes_to_zero_in_case_of_infinitely_many_SA_steps}
f(x_k) - f(x_{k+1}) \geq \eta_0 \mu_1 \min\{\varepsilon_c,\, \mu^{-q}\rho_k^{q}\}\rho_k.
\end{equation}
Since $\mathcal{S}$ is countably infinite and the objective function $f$ is bounded from below within the feasible set $X$, the right-hand side of (\ref{eq:rhs_rho_goes_to_zero_in_case_of_infinitely_many_SA_steps}), i.e., the trust region radius $\rho_k$, has to converge to zero.
\end{proof}

\medskip

Lemma \ref{lem:tr_region_radius_converges_to_zero} shows that using the stopping criterion $\rho_{k+1} < \rho_{min}$ is reasonable and results in termination of NOWPAC after a finite number of steps. Another direct consequence of Lemma \ref{lem:tr_region_radius_converges_to_zero} is that
\begin{equation}\label{lem:iminf_of_alphak_goes_to_zero}
\liminf\limits_{k \rightarrow \infty} \alpha_k(\rho_k) = 0,
\end{equation}
since $\alpha_k(\rho_k) \geq \kappa_1$ for some $\kappa_1 > 0$ for all $k$ implies $\rho_k \geq \kappa_2$ for all $k$. The following theorem shows that the convergence of a subsequence of the approximated criticality measures $\{\alpha_{k_i}\}_{k_i}$ is carried over to the overall convergence of the exact criticality measures to zero.

\medskip

\begin{theorem}\label{thm:overall_convergence_of_exact_criticality meausre}
It holds that
\[
\lim\limits_{k \rightarrow \infty} \mathfrak{A}[x_k] = 0.
\]
\end{theorem}

\begin{proof}
Since the theorem holds for $|\mathcal{S}| < \infty$ by Lemma \ref{lem:convergence_finite_number_of_successful_steps}, we assume that $\mathcal{S}$ is a countably infinite set. Following the ideas of the proof of \cite[Thm.\ 10.13]{Conn2009a} we prove the assertion of the theorem by contradiction. Assume that there exists a subsequence $\{\hat{k}_i\}_i \subseteq \mathcal{S}$ such that
\begin{equation}\label{eq:in_convergenceproof_contradiction_assumption}
\mathfrak{A}[x_{\hat{k}_i}] \geq \varepsilon_0
\end{equation}
for some $\varepsilon_0 > 0$ for all $i$. It immediately follows from Lemma \ref{lem:quality_of_approx_criticality_measure}  that
\[
\alpha_{\hat{k}_i}(\rho_{\hat{k}_i}) \geq \varepsilon
\]
for some $\varepsilon > 0$ for all $i$ sufficiently large; in particular this holds true for
\begin{equation}\label{eq:inproof_convergence_theorem_eq2}
\varepsilon < \frac{\varepsilon_0}{4}.
\end{equation}
Based on the subsequence $\{\hat{k}_i\}_i$ we define two subsequences $\{k_i\}_i$ and $\{l_i\}_i$ of all steps as follows: starting from $k_1 = \hat{k}_1$ we choose the first index $l_1 > k_1$ for which $\alpha_{l_1}(\rho_{l_1}) < \varepsilon$ and define the remaining members of the two subsequences inductively. Determine $j_i := \min\{ j \in \mathbb{N} \; : \; \hat{k}_j > l_i\}$, set $k_{i+1} = \hat{k}_{j_i}$, and choose $l_{i+1} > k_{i+1}$ as being the first index for which $\alpha_{l_{i+1}}(\rho_{l_{i+1}}) < \varepsilon$. 
Note that the existence of $\{l_i\}_{i}$ is guaranteed by (\ref{lem:iminf_of_alphak_goes_to_zero}). 
We thus arrive at subsequences of indices satisfying
\begin{equation}\label{eq:inproof_convergence_theorem_eq1}
\alpha_k(\rho_k) \geq \varepsilon \quad \mbox{for} \quad k_i \leq k < l_i \quad \mbox{and} \quad \alpha_{l_i}(\rho_{l_i}) < \varepsilon.
\end{equation}

Before we conclude the proof of convergence of $\{x_k\}_k$ to a first-order stationary point, we first have to show that $\lim_{i \rightarrow \infty}\|x_{k_i} - x_{l_i}\| = 0$; cf.\ the proof of \cite[Thm.\ 10.13]{Conn2009}. For this we define the set of indices
\[
\mathcal{K} := \bigcup\limits_{i \in \mathbb{N}_0}\left\{k \in \mathbb{N}_0 \;:\; k_i \leq k < l_i\right\},
\]
with the sequences $\{k_i\}_i$ and $\{l_i\}_i$ as defined above. We know that $\alpha_k(\rho_k) \geq \varepsilon$ for $k \in \mathcal{K}$. Thus, since $\rho_k \rightarrow 0$ (see Lemma \ref{lem:tr_region_radius_converges_to_zero}) it follows from Lemma \ref{lem:successful_step} that the iteration $k$ is acceptable or successful for all $k \in \mathcal{K}$ large enough. For every $k \in \mathcal{K} \cap \mathcal{S}$ we have
\[
f(x_k) - f(x_{k+1}) \geq \eta_0 \left(m_k^f(x_k) - m_k^f(x_{k+1})\right) \geq \eta_0 \mu_1 \alpha_k(\rho_k) \rho_k \geq \eta_0 \mu_1 \varepsilon \rho_k.
\]
Thus we obtain 
\[
\|x_{k_i} - x_{l_i}\| \leq \sum\limits_{\substack{j=k_i\\j \in \mathcal{K} \cap \mathcal{S}}}^{l_i-1} \|x_j - x_{j+1}\| \leq  \sum\limits_{\substack{j=k_i \\ j \in \mathcal{K} \cap \mathcal{S}}}^{l_i-1} \rho_{j} \leq \frac{1}{\eta_1 \mu_1 \varepsilon} \left(f(x_{k_i}) - f(x_{l_i})\right)
\]
for $k_i$ sufficiently large. Noting that the sequence $\{f(x_k)\}_k$ is bounded from below (see Assumption \ref{ass:general_assumptions}) and monotonically decreasing, it follows that the left-hand side of the inequality above must converge to zero for $i \rightarrow \infty$.

Now, due to the continuity of the exact criticality measure (see Lemma \ref{lem:upper_lower_semicontinuity}) and $\|x_{k_i} - x_{l_i}\| \rightarrow 0$, it holds that $ \left| \mathfrak{A}[x_{l_i}]  -  \mathfrak{A}[x_{k_i}]\right| < \varepsilon$ for $i$ sufficiently large. Moreover, using the fact that $\rho_{l_i} \rightarrow 0$, Corollary \ref{cor:technical_corollary} along with Lemma \ref{lem:big_A_and_approx_crit_lemma_relation} yields
\begin{align*}
\mathfrak{A}[x_{k_i}] &\leq \left|\mathfrak{A}[x_{k_i}] - \mathfrak{A}[x_{l_i}]\right| + \left|\mathfrak{A}[x_{l_i}] - \mathfrak{A}_1[x_{l_i}]\right| + \left|\mathfrak{A}_1[x_{l_i}] - \mathfrak{A}_2[x_{l_i}]\right| + \mathfrak{A}_2[x_{l_i}]\\
& < \varepsilon + \varepsilon + \varepsilon + \alpha_{l_i}(x_{l_i}) < 4\varepsilon < \varepsilon_0
\end{align*}
for $i$ sufficiently large, which contradicts (\ref{eq:in_convergenceproof_contradiction_assumption}).
\end{proof}

\section{Implementation and choice of parameters}\label{sec:practical_aspects}
Having discussed the theoretical properties of Algorithm \ref{alg:algorithm1} in Section \ref{sec:convergence_proof}, we now comment on the practical implementation of NOWPAC. In particular we address the practical choice of the order reduction parameter $p$ as well as the existence of trial steps $\{s_k\}$ in {\tt STEP 2} satisfying Assumptions \ref{ass:stepassumptions}.

First we examine the existence of trial steps satisfying Assumptions \ref{ass:stepassumptions}; for this we consider the optimal solutions $\{\hat{d}_k\}$ of the criticality subproblem (\ref{eq:approx_criticality_measure}). We assume that the refinements in {\tt STEP 1} result in $\|g_k^f\| > 0$ eventually; otherwise, since {\tt STEP 1} ensures full linearity of the objective model, we have $\Vert \nabla f(x_k)\Vert  = 0$, i.e., $x_k$ is already a first-order critical point. It holds that
\[
\min\left\{ \varepsilon_c,\, \mu^{-q}\rho_k^{q}\right\} \leq \alpha_k(\rho_k) = -\frac1{\rho_k}\left\langle g_k^f, \hat{d}_k\right\rangle =
\frac1{\rho_k} \left\Vert g_k^f\right\Vert  \left\Vert \hat{d}_k\right\Vert  \cos(\phi_d),
\]
where $\phi_d$ denotes the angle between $-g_k^f$ and $\hat{d}_k$. The first inequality is a direct consequence of {\tt STEP 1} in Algorithm~\ref{alg:algorithm1} for $\rho_k$ sufficiently small. Moreover, 
since $\hat{d}_k$ is a descent direction
%\todo{Do you mean ``since $\hat{d}_k$ is a descent direction'' or instead ``if $\hat{d}_k$ is a descent direction?'' Answer: I meant ''for ...'' in the sense of ''because of ...''. I changed it to ''since ... ''}
 we know that $\phi_d < \frac{\pi}{2}$ and thus $\cos(\phi_d) > 0$. Thus, for every $x_k$ that is not an optimal solution of (\ref{eq:general_opt_problem}) we have 
\[
\left\|\hat{d}_k\right\| \geq  \frac{\min\left\{ \varepsilon_c, \, \mu^{-q} \rho_k^{q}\right\}\rho_k}{\left\|g_k^f\right\| \cos(\phi_d)},
\]
which justifies Assumption \ref{ass:stepassumptions}(b). For Assumption \ref{ass:stepassumptions}(a) we note that $m_k^f(x_k) - m_k^f(x_k + \hat{d}_k) = -\langle g_k^f, \hat{d}_k\rangle - t(\hat{d}_k)$ with the remainder term of the Taylor approximation $t(\hat{d}_k) \in \mathcal{O}(\| \hat{d}_k \|^2)$. It holds that
\begin{align*}
m_k^f(x_k) - m_k^f(x_k + \hat{d}_k) &= -\left\langle g_k^f, \hat{d}_k\right\rangle - t(\hat{d}_k)
\geq \rho_k \alpha_k(\rho_k) - \left|t(\hat{d}_k)\right|\\
&= \rho_k \alpha_k(\rho_k) - \|\hat{d}_k\|^2\left|t(\hat{d}_k)\|\hat{d}_k\|^{-2}\right|\\
&\geq \rho_k \alpha_k(\rho_k) - \rho_k^2\left|t(\hat{d}_k)\|\hat{d}_k\|^{-2}\right|
\geq \frac12 \rho_k \alpha_k(\rho_k),
\end{align*}
for $\rho_k$ sufficiently small. For the last inequality we used the fact that $\rho_k \in o(\alpha_k(\rho_k))$, which is ensured by {\tt STEP 1} in Algorithm~\ref{alg:algorithm1}. In our implementation of NOWPAC, however, we avoid repeating {\tt STEP 1} whenever the trial step $s_k$ is infeasible and go directly to line {\tt 10} in Algorithm~\ref{alg:algorithm1}. We do this to reduce the computational costs of computing the criticality measure at every infeasible step. 

NOWPAC is designed to work in settings with costly objective function evaluations that dominate the cost of computing a good trial step $s_k$. This suggests it may be beneficial (in terms of the overall computational costs) to invest effort in computing a good trial step rather than looking for a quick and crude approximation. In our implementation of NOWPAC we use the CCSA algorithm \cite{Svanberg2002}, as implemented in the NLopt library \cite{JohnsonNLopt}, to compute the trial steps in {\tt STEP 2} of Algorithm~\ref{alg:algorithm1}, i.e., to find $s_k = \argmin_{x_k+s \in X_k,\, \|s\| \leq \rho_k} m_k^f(x_k + s)$.

\medskip 

Next we discuss the choice of the order reductions $p$ and $q$. We briefly recall where we introduced $p$ and $q$:
\begin{itemize}
 \item The order reduction parameter $p$ appears in the definition of the inner boundary path (\ref{eq:definition_of_inner_boundary_path}); it is required within Lemma \ref{lem:successful_step_3} to prove that {\tt STEP 3} of Algorithm~\ref{alg:algorithm1} eventually finds a feasible trial step, when the trust region radius is small enough.
%\todo{I changed `if' to `when' because the trust region radius is guaranteed to eventually become small enough. Is this okay? Answer: Yes, that is okay.}
 \item Within Lemma \ref{lem:successful_step_3}, we used Assumption \ref{ass:stepassumptions}(b) on the step size being a fraction of order $\rho_k^{1+q}$; the latter assumption is ensured by {\tt STEP 1} of Algorithm~\ref{alg:algorithm1}.
\end{itemize}
In practical applications, Algorithm \ref{alg:algorithm1} is always terminated when the trust region radius falls below the threshold $\rho_{min}$. We therefore discuss the choice of order reductions $p=0$ and $q=0$ in the {\it pre-asymptotic regime} of $\rho_k \geq \rho_{min}$.

We note that, as long as {\tt STEP 2} computes a descent direction $s_k$, we can always find (potentially small) parameters $\mu_1$ and $\mu_2$ such that Assumption \ref{ass:stepassumptions} is satisfied with $q = 0$ for all $\rho_k \geq \rho_{min}$.
Thus, Assumption \ref{ass:stepassumptions} can be satisfied for $\rho_k \geq \rho_{min}$, regardless of the choice of $q$ in {\tt STEP 1} of Algorithm~\ref{alg:algorithm1}, allowing us simply to check for $\rho_k > \mu \alpha_k(\rho_k)$ in {\tt STEP 1} of Algorithm~\ref{alg:algorithm1}. 
Revisiting the proof of Lemma \ref{lem:successful_step_3} with $p = 0$ and using $\|s_k\| \geq \mu_2 \rho_k^1$, we see that {\tt STEP 3} computes a feasible trial step $s_k$ if $\varepsilon_b \geq \kappa_c \mu_2^{-2}$.
Thus we are guaranteed to find feasible trial steps by choosing the inner boundary path constant large enough, even for the choice of $p = 0$ in the definition of the inner boundary path. Finally, note that choosing the inner boundary path to be a quadratic offset ($p = 0$) 
is consistent with Assumption \ref{ass:convexification_through_ibp} and Lemma \ref{lem:big_A_and_approx_crit_lemma_relation} since its Hessian is positive definite.
%
%Thus, in the pre-asymptotic phase we can choose $q > 0$ freely in $\rho_k \leq \mu \alpha_k(\rho_k)^{\frac1{q}}$ in {\tt STEP 1} of NOWPAC to maximize its efficiency. On the one hand, this relation between $\alpha_k$ and $\rho_k$ for a small value $q$ results in an overly tight restriction of the step size due to small trust region radii. On the other hand, if we choose $p$ too large, we loosen the correlation between the criticality value of an iterate and the size of the trust region radius. We therefore set $q = 1$ in order to have the order-$1$ relation, $\rho_k \leq \mu \alpha_k(\rho_k)$, between the criticality value and the trust region radius.
%
Finally, we propose a heuristic for an adaptive choice of the inner boundary path constant $\varepsilon_b$. From Lemma \ref{lem:successful_step_3} we see that $\varepsilon_b$ has to be chosen sufficiently large in order to 
convexify the constraints and to guarantee that {\tt STEP 3} in Algorithm~\ref{alg:algorithm1} will be passed with a feasible trial step. In practical applications, however, we have found that NOWPAC works very well for an \textit{a priori} fixed value of $\varepsilon_b > 0$ along with the adaptive scaling 
\begin{equation}\label{eq:adjusted_inner_boundary_constant}
\varepsilon_{b,k} := \varepsilon_{b} \left(\frac{\|s_{k-1}\|}{\rho_{k-1}}\right)^{2},
\end{equation}
to increase efficiency by not overly constraining the trial steps due to the inner boundary path. The situation of a too-large inner boundary path constant and its relaxation (\ref{eq:adjusted_inner_boundary_constant}) is depicted in the left plot of Figure \ref{fig:exampleforinnerboundarypath}, where we see an unnecessary restriction of the possible step size.

\section{Inexact evaluations of the objective function and constraints}
\label{sec:noisy_function_evaluations}
In the preceding sections we assumed that we are able to evaluate the objective function and the constraints up to a prescribed tolerance, so that the models $m_k^f$ and $m_k^c$ are fully linear (\ref{eq:fully_linear_equations}). This assumption requires the function evaluations to become more and more accurate when approaching a critical point. As we noted in Section \ref{sec:introduction}, there exist theoretical results in the context of derivative-based trust region methods (see \cite{Carter1991, Heinkenschloss2002} and the references therein) showing convergence in case of increasing accuracy of the evaluations while approaching the optimal design. For corresponding results for derivative-free methods, see for example \cite{Bortz1998, Kelley1999}.
In practical applications, we are often faced with situations where we cannot avoid inexact evaluations of the objective function or the constraints. Inexactness may stem from numerical errors, limitations on the number of cycles in a recursive procedure, inaccurate measurements, and other factors. Particularly in cases where the objective function and constraints are given only as black-box evaluations of a simulation code, we are not likely to be able to tune the model tolerances. Figure~\ref{fig:NoisyTarRespSurfs} provides an example of the inexact function evaluations that we would like our method to address; shown are evaluations of the objective function and a constraint function in the tar removal process model of Section~\ref{sec:num_results_tar_removal_process}. 
The small-scale roughness is the result of numerical errors. 

To avoid any ambiguity, we contrast our focus on numerical errors with the case of an objective or constraint function that depends on uncertain parameters, where the parameters may be constrained to some interval or endowed with a probability distribution. In the latter case, one might account for uncertainty by replacing the objective function or constraint with its ``robust counterpart,'' yielding a task in stochastic programming. We refer the interested reader to \cite{Beyer2007, Bhatnagar2013, Kushner1997, Robbins1951, Shapiro2009} and references therein. Methods for stochastic programming require the exploration of the uncertain parameter space in some fashion, and are not our focus here. Of course, there is a link between the introduction of robust objectives and the issue of numerical error; for instance, numerical evaluation of an expectation with respect to the uncertain parameters is subject to error due to a finite number of Monte Carlo samples or finite quadrature resolution. But our focus here is on the presence and magnitude of numerical errors only, regardless of how they originated. In other words, we do not distinguish among different sources of inexactness in evaluations of $f$ and $c$.

This section first addresses the situation where increasing the accuracy of the evaluations of the objective function and the constraints is possible. In this case we quantify the rate of noise reduction needed to guarantee convergence. Thereafter, we discuss regimes where the accuracy level cannot be adjusted and propose an indicator to detect when inexact evaluations of $f$ or $c$ prevent NOWPAC from making progress. In this case, we propose early termination of the algorithm to save computational effort and to prevent corruption of the results.

For the error analysis in this section we assume that the objective function and the constraints can each be split into a sum of two terms. The first terms are the functions themselves, satisfying Assumptions \ref{ass:general_assumptions} and \ref{ass:bound_on_model_Hessian}. The second terms are the errors. These error terms are only observed (via summation with the exact function value) at the finite number of points where the objective function and the constraints are evaluated. We fill in the gaps between these points using quadratic extensions of the error, via minimum-Frobenius-norm models $\delta_k^f$ and $\delta_k^{c}$. We point out that $\delta_k^f$ and $\delta_k^{c}$ are simply extensions of the observed errors, rather than approximations of the actual error. We assume that the magnitudes of the errors are bounded by $\delta_{k, max}^f$ and $\delta_{k, max}^{c}$. Beyond this, we do not make any additional assumptions, e.g., on the distribution of the error or even whether it is stochastic or deterministic. In order to be detectable, however, the errors at different points in the design space must be sufficiently uncorrelated. For example, if the error term degenerates to a constant offset, it is impossible to separate it from the underlying objective function or constraint by simply observing its sum with one of the latter. The same holds true for errors satisfying equivalent smoothness properties as the objective function and the constraints. 
Summarizing, the perturbed observed functions are given by
\begin{align*}
f_{n,k}(x) &:= f(x) + \delta_k^f(x),\\
c_{n,k}(x) &:= c(x) + \delta_k^{c}(x).
\end{align*}
In the following theorem we prove the rate of decay---with respect to the trust region radii---that the errors $\delta_k^f$ and $\delta_k^c$ must obey in order to guarantee convergence of NOWPAC; cf.\ also \cite{Kannan2012}. 
To simplify notation in the following theorem, without loss of generality we assume the design variables to be properly scaled and set $\rho_{max} = 1$.

\medskip

\begin{theorem}\label{thm:error_decrease_rate}
Consider the fully linear minimum-Frobenius-norm surrogates $m_k^{f_{n,k}}$ and $m_k^{c_{n,k}}$ of the observed noisy objective function $f_{n,k}$ and constraints $c_{n,k}$. The intermediate points $\{x_k\}_k$ computed by Algorithm \ref{alg:algorithm1} converge to a first-order critical point if $
\delta_{k, max}^f, \delta_{k, max}^c \in o\left(\rho_k^{2}\right)$ and the inner boundary path constant
$\varepsilon_b$ is greater than $\mu_2^{-\frac{2}{1+p}}(\kappa_c + \delta_{k,max}^c\rho_k^{-2\frac{1+q}{1+p}})$.
\end{theorem}

\begin{proof}
First we show that if $\delta_{k, max}^f, \delta_{k, max}^c \in o\left(\rho_k^{2}\right)$, the full linearity of the noisy models is maintained. In \cite[Thm.\ 5.4]{Conn2009a} it is shown that for minimum-Frobenius-norm models we have
\begin{align*}
\left| f_{n,k}(x_k + s) - m_k^{f_{n,k}}(x_k + s)\right| & \leq  \kappa_{f_{n}} \rho_k^{2} \\
\left| c_{n,k}(x_k + s) - m_k^{c_{n,k}}(x_k + s)\right| & \leq  \kappa_{c_{n}} \rho_k^{2} \\
\left\| \nabla f_n(x_k + s) - \nabla m_k^{f_n}(x_k + s)\right\| & \leq \kappa_{df_n} \rho_k \\
\left\| \nabla c_n(x_k + s) - \nabla m_k^{c_n}(x_k + s)\right\| & \leq \kappa_{dc_n} \rho_k,
\end{align*}
with
\begin{align}\label{eq:pertubed_fully_linearity_coefficients}
\begin{split}
\kappa_{f_n} &= \left(\kappa_{n} + \frac12\right) \left(\nu_{n,k}^f + \left\|H_k^{f_{n,k}}\right\|\right)\\
\kappa_{c_n} &= \left(\kappa_{n} + \frac12\right) \left(\nu_{n,k}^c + \left\|H_k^{c_{n,k}}\right\|\right)\\
\kappa_{df_{n}} &= \kappa_{n} \left(\nu_{n,k}^f + \left\|H_k^{f_{n,k}}\right\|\right)\\[3mm]
\kappa_{dc_{n}} &= \kappa_{n} \left(\nu_{n,k}^c + \left\|H_k^{c_{n,k}}\right\|\right).
\end{split}
\end{align}
Here the constant $\kappa_{n}$ depends on the geometry of the interpolation points, but it does not depend on the trust region radius $\rho_k$. Also, $\nu_{n,k}^f$ and $\nu_{n,k}^c$ denote the Lipschitz constants of the gradients of $f_{n,k}$ and $c_{n,k}$, while $H_k^{f_{n,k}}$ and $H_k^{c_{n,k}}$ denote the Hessians of $m_k^{f_{n,k}}$ and $m_k^{c_{n,k}}$. Now we examine the Lipschitz constants of the gradients of $f_{n,k}$ and $c_{n,k}$. Using the triangle inequality we get
\begin{align*}
\left\| \nabla f_{n,k}(x_1) - \nabla f_{n,k}(x_2)\right\| 
&\leq \left\| \nabla f(x_1) - \nabla f(x_2)\right\| + \left\| \nabla \delta_k^f(x_1) + \nabla \delta_k^f(x_2)\right\|\\
&\leq \left(\nu^f + \left\|H_k^{\delta_k^f}\right\|\right)\left\| x_1 - x_2\right\|,\\
\left\| \nabla c_{n,k}(x_1) - \nabla c_{n,k}(x_2)\right\| 
&\leq \left\| \nabla c(x_1) - \nabla c(x_2)\right\| + \left\| \nabla \delta_k^c(x_1) + \nabla \delta_k^c(x_2)\right\|\\
&\leq \left(\nu^c + \left\|H_k^{\delta_k^c}\right\|\right)\left\| x_1 - x_2\right\|,
%\left\| \nabla c_{n,k}(x_1) - \nabla c_{n,k}(x_2)\right\| 
%&\leq \left(\nu^c + \left\|H_k^{\delta_k^c}\right\|\right)\left\| x_1 - x_2\right%\|,
\end{align*}
where $\nu^f$ and $\nu^c$ are the Lipschitz constants of $\nabla f$ and $\nabla c$, and $H_k^{\delta_k^f}$ and $H_k^{\delta_k^c}$ denote the Hessians of the error functions $\delta_k^f(x)$ and $\delta_k^c(x)$. From the above inequalities we obtain upper bounds on the Lipschitz constants:
\begin{equation}\label{eq:lipschitz_noisy_function}
\nu_{n,k}^f \leq \nu^f + \left\|H_k^{\delta_k^f}\right\| \quad \mbox{and} \quad
\nu_{n,k}^c \leq \nu^c + \left\|H_k^{\delta_k^c}\right\|.
\end{equation}
 Furthermore it holds that
\begin{align*}
\left\|H_k^{f_{n,k}}\right\| &= \left\|H_k^{f} + H_k^{\delta_k^f}\right\| \leq \left\|H_k^{f}\right\| + \left\|H_k^{\delta_k^f}\right\| \quad \mbox{and}\\
\left\|H_k^{c_{n,k}}\right\| &= \left\|H_k^{c} + H_k^{\delta_k^c}\right\| \leq \left\|H_k^{c}\right\| + \left\|H_k^{\delta_k^c}\right\|,
\end{align*}
which, together with (\ref{eq:lipschitz_noisy_function}) and  (\ref{eq:pertubed_fully_linearity_coefficients}), yields
\begin{align}\label{eq:error_reduction_proof_inequalities}
\begin{split}
\kappa_{f_n} & \leq \left(\kappa_{n} + \frac12\right) \left(\nu^f + \left\|H_k^{f}\right\| + 2\left\|H_k^{\delta_k^f}\right\|\right) \leq \kappa_1^f + 2\kappa_2^f \delta_{k, max}^f\rho_k^{-2} \\
\kappa_{c_n} & \leq \left(\kappa_{n} + \frac12\right) \left(\nu^c + \left\|H_k^{c}\right\| + 2\left\|H_k^{\delta_k^c}\right\|\right) \leq \kappa_1^c + 2\kappa_2^c \delta_{k, max}^c\rho_k^{-2}\\
\kappa_{df_n} & \leq \kappa_{n} \left(\nu^f + \left\|H_k^{f}\right\| + 2\left\|H_k^{\delta_k^f}\right\|\right) \leq \kappa_{3}^f + \kappa_2^f \delta_{k, max}^f\rho_k^{-2}\\
\kappa_{dc_n} & \leq \kappa_{n} \left(\nu^c + \left\|H_k^{c}\right\| + 2\left\|H_k^{\delta_k^c}\right\|\right) \leq \kappa_{3}^c + \kappa_2^c \delta_{k, max}^c\rho_k^{-2}.
\end{split}
\end{align}
for constants $\kappa_1^f$, $\kappa_1^c$, $\kappa_2^f$, $\kappa_2^c$, $\kappa_3^f$, $\kappa_3^c > 0$. In the second inequalities we used
\[
\left\|H_k^{\delta_k^f}\right\| \leq \bar{\kappa}_2^f \frac{\delta_{k, max}^f}{\rho_k^{2}} \quad \mbox{and} \quad
\left\|H_k^{\delta_k^c}\right\| \leq \bar{\kappa}_2^c \frac{\delta_{k, max}^c}{\rho_k^{2}},
\]
cf.\ \cite[Thm.\ 5.7]{Conn2009a} where we replace the upper bound mentioned in the proof of  \cite[Thm.\ 5.7]{Conn2009a} by $\max_{x\in B(x_k, \rho_k)} |\delta_k^f(x)| \leq \delta_{k, max}^f$ and  $\max_{x\in B(x_k, \rho_k)} |\delta_k^c(x)| \leq \delta_{k, max}^c$ respectively. %We remark that the right-hand sides in (\ref{eq:pertubed_fully_linearity_coefficients}) have the same asymptotic behavior as $\|H_k^{\delta_k^f}\|$ and $\|H_k^{\delta_k^c}\|$ respectively, for $\|H_k^{\delta_k^f}\|$ and $\|H_k^{\delta_k^c}\|$ large enough. 
Thus the full linearity properties (\ref{eq:fully_linear_equations}) and Assumption \ref{ass:bound_on_model_Hessian} hold, if we ensure that the right-hand sides of (\ref{eq:error_reduction_proof_inequalities}), i.e., the values of $\kappa_{f_n}$, $\kappa_{c_n}$, $\kappa_{df_n}$ and $\kappa_{bhf_n}$, do not grow unboundedly; in particular this is the case if 
\[
\delta_{k, max}^f \in o\left(\rho_k^{2}\right) \quad \mbox{ and } \quad \delta_{k, max}^c \in \left(\rho_k^{2}\right).
\]
To conclude the proof of convergence, we relate the noisy function evaluations $f_{n,k}$ and $c_{n,k}$ to the exact objective function $f$ and constraints $c$, i.e.,
\begin{align*}
\left| f(x_k + s) - f_{n,k}(x_k + s)\right| & \leq \delta_{k, max}^f \in o(\rho_k^2)\\
\left| c(x_k + s) - c_{n,k}(x_k + s) \right| & \leq \delta_{k, max}^c \in o(\rho_k^2) \\
\left\|\nabla f(x_k + s)  - \nabla f_{n,k}(x_k + s) \right\| & = \left\|\nabla \delta_k^f(x_k + s) \right\| \in o(\rho_k)\\
\left\|\nabla c(x_k + s)  - \nabla c_{n,k}(x_k + s) \right\| & = \left\|\nabla \delta_k^c(x_k + s) \right\| \in o(\rho_k).
\end{align*}
The latter inclusion follows from \cite[Thm. 5.4]{Conn2009} by interpreting $\delta_k^f$, or $\delta_k^c$ respectively, as an $o(\rho_k^2)$ approximation of the constant zero function and replacing the assumption on exact function evaluations in the proof of \cite[Thm. 5.4]{Conn2009} with the point-wise error $o(\rho_k^2)$ (see Lemma \ref{lem:completion_of_noise_lemma} for more details).
In summary, $m_k^{f_{n,k}}$ and $m_k^{c_{n,k}}$ are fully linear models for the exact objective function $f$ and the constraints $c$, respectively.

Finally we address the issue that Algorithm~\ref{alg:algorithm1} may pass {\tt STEP 3} with a trial step that is incorrectly designated as feasible. 
We have to ensure that no infeasible step is accepted because it appears to be feasible due to the noise. For this we revisit the proof of Lemma \ref{lem:successful_step_3} and now require the safety margin $c(x_k + s_k) \leq -\delta_{k,max}^c$. It follows that
\[
\rho_k \leq \left(\frac{1}{\kappa_c}\left(\varepsilon_b \mu_2^{\frac{2}{1+p}} - \delta_{k,max}^c\rho_k^{-2\frac{1+q}{1+p}}\right)\right)^{\frac{1+p}{2(p-q)}}.
\]
Thus, in the pre-asymptotic phase we have to choose $\varepsilon_b$ large enough so that the right-hand side in the inequality above is always greater than $\rho_{max} = 1$, yielding
\[
%\varepsilon_b \geq  \mu_2^{-\frac{2}{1+p}}\delta_{k,max}^c\rho_k^{-2\frac{1+q}{1+p}},
\varepsilon_b \geq \mu_2^{-\frac{2}{1+p}}\left(\kappa_c + \delta_{k,max}^c\rho_k^{-2\frac{1+q}{1+p}} \right),
\]
where this additional restriction vanishes for decreasing trust region radius $\rho_k$.
\end{proof}

\medskip

Note that even though we only have access to inexact evaluations of the constraints, we still have to be able to check feasibility in {\tt STEP 3} in Algorithm~\ref{alg:algorithm1}. However, we point out that in the asymptotic regime $k \rightarrow \infty$, $\delta_{k,max}^c = o(\rho_k^{2})$ will be dominated by the requirement from Assumption \ref{ass:convexification_through_ibp}. Thus, in practical applications, the inner boundary path constant $\varepsilon_b$ has to be adapted to the magnitude of the errors (or simply chosen sufficiently large) in order to ensure convergence of Algorithm \ref{alg:algorithm1} also in the pre-asymptotic regime where the trust region radii are not close to zero.

We now use Theorem \ref{thm:error_decrease_rate} to define an indicator that estimates the minimum trust region radius at which further progress of Algorithm \ref{alg:algorithm1} is expected not to improve the optimization result. In other words, the indicator detects regimes in which $\delta_{k, max}^f, \delta_{k, max}^c = o\left(\rho_k^{2}\right)$ is violated. The indicator is based on observing the increase of the norms of $H_k^{f_n}$ or $H_k^{c_n}$ as functions of the trust region radius. Note that these Hessians are computed in every iteration of Algorithm \ref{alg:algorithm1} and are therefore readily available without additional computational cost.
As described in the proof of Theorem \ref{thm:error_decrease_rate}, the norms of $H_k^{f_n}$ and $H_k^{c_n}$ show the same asymptotic behavior as $\delta_{k,max}^{f}\rho_k^{-2}$ and $\delta_{k,max}^{c}\rho_k^{-2}$ for decreasing trust region radii. Thus, on a logarithmic scale, the slope $\tau$ of the growth of $\|H_k^{f_n}\|$ and $\|H_k^{c_n}\|$ with respect to the trust region radii $\rho_k$ should not grow for vanishing $\rho_k$. We estimate this slope using linear regression of the corresponding norms of the Hessians at rejected steps, i.e., at steps where the intermediate point does not change. In these steps, the Hessian is supposed to be of the same order of magnitude, provided the geometry constant $\kappa_n$ does not grow unboundedly. The latter property is ensured by Algorithm \ref{alg:algorithm1} within the model improvement steps. The result is the approximated slope $\tau\left( \|H_k\| \right)$, where $H_k$ denotes either $H_k^{f_n}$ or $H_k^{c_n}$. The slope can be categorized as indicating a convergent or possibly non-convergent regime, according to the corresponding convergence properties of Algorithm \ref{alg:algorithm1}; in our numerical examples we set the noise threshold to $1$, i.e., we used the classification $\tau\left( \|H_k\| \right) < 1$ for the convergent regime and $\tau\left( \|H_k\| \right) \geq 1$ for the possibly non-convergent regime.
If possibly non-convergent iterations are detected, we suggest early termination of Algorithm \ref{alg:algorithm1} in order to prevent deterioration of the results.

\section{Numerical results}\label{sec:numerical_results}
In this section we apply NOWPAC to several optimization problems. In Sections \ref{sec:rosenbrock} and \ref{sec:exponential}, we discuss two model problems: the Rosenbrock function (\ref{eq:rosenbrock}) and a nonlinear constrained anisotropic exponential example (\ref{eq:exponential_obj}). We use these examples for validation of the algorithm, knowing the exact optimal points and the minimum objective values. In both examples, we also demonstrate the effectiveness of the error indicator proposed in Section \ref{sec:noisy_function_evaluations}. In Section \ref{sec:benchmark_test} we demonstrate NOWPAC's efficiency on test problems from the Schittkowski benchmark set \cite{Schittkowski2008}. Then in Section \ref{sec:num_results_tar_removal_process}, we apply NOWPAC to a large-scale black box model of tar removal in a biomass-to-liquid plant. In this example, the model enters evaluations of both the objective and the constraints. In all examples, we set the parameters in NOWPAC to the values shown in Table \ref{tab:parameters_example1}. We use the initial trust region radius $\rho_0 = 0.1$ throughout the examples.

\begin{table}[!htb]
\caption{Default parameters of NOWPAC used in the test examples of Sections \ref{sec:rosenbrock}--\ref{sec:num_results_tar_removal_process}.}\label{tab:parameters_example1}
\begin{center}\footnotesize
\renewcommand{\arraystretch}{1.3}
\begin{tabular}{|c||c|c|c|c|c|c|c|c|c|}\hline
parameter & $\varepsilon_b$ &  $\gamma_{inc}$ & $\gamma$ & $\omega$ & $\eta_0$ & $\eta_1$ \\\hline\hline
value     & $10.0$          &  $2.0$        & $0.8$    & $0.6$    & $0.1$    & $0.7$ \\\hline
\end{tabular}
\end{center}
\end{table}

For comparison we also compute the optimal points using the linear surrogate-based solver COBYLA and the direct search algorithms NOMAD, SDPEN, and GSS-NLC. We use the implementation of COBYLA \cite{Powell1994, Powell1998} from the NLopt optimization library for nonlinear optimization \cite{JohnsonNLopt} and the implementation of NOMAD as in \cite{LeDigabel2011}. The asynchronous parallel pattern search (APPS) method GSS-NLC \cite{Adams2013, Gray2006} is a pattern search method based on sampling of the objective function augmented by a penalty approach for the constraints. In all examples we picked an initial penalty parameter of $10^{3}$ and an initial step size of $0.1$, where we found GSS-NLC to perform best. We used the GSS-NLC code as provided by the hybrid optimization parallel search package (HOPSPACK) \cite{Plantenga2009}. All computations for the benchmark problems were performed on a $2.6$ GHz Intel Core i$7$ processor using the GNU compiler version $4.8$.

\subsection{Rosenbrock function}\label{sec:rosenbrock}
The first example is the unconstrained optimization of the Rosenbrock function,
\begin{equation}\label{eq:rosenbrock}
\min\limits_{(x_1, x_2)\in \mathbb{R}^2} (x_2 - x_1^2)^2 + (x_1 - 1)^2,
\end{equation}
where the optimal point $x^\ast = (1, 1)^T$ and the minimal objective value $f^\ast = 0$ are known analytically. We start the optimization at $x_0 = (1.5, 1.5)^T$. The Rosenbrock function exhibits small gradients in the neighborhood of the optimal point $x^\ast$, and thus constitutes a setting in which derivative-free trust region methods like NOWPAC are forced to take small steps; we recall that the size of the trust region, and therefore the step size, is tightly connected to the size of the gradients. For this reason, we find it worthwhile to include this unconstrained test case to discuss the performance of NOWPAC.

The performances of the solvers NOWPAC, COBYLA, NOMAD, SDPEN, and GSS-NLC are summarized in Table \ref{tab:rosenbrock_sum}. We see that all methods result in reasonable approximations of the exact optimum. Looking at the number of function evaluations, we see that NOWPAC requires fewer function evaluations than the other solvers. 
\begin{table}[htbp]
\caption{Summarized performance statistics of  NOWPAC, COBYLA, NOMAD, SDPEN, and GSS-NLC applied to the Rosenbrock minimization problem \ref{eq:rosenbrock}. $SC$ indicates the stopping criteria, which is the $\rho_{min}$ threshold for NOWPAC and the absolute distance in the coordinate directions for COBYLA, NOMAD, SDPEN, and GSS-NLC. $d_x$ and $d_f$ denote the Euclidean distances between the approximated and the analytical solution.}\label{tab:rosenbrock_sum}
\begin{center}\footnotesize
\renewcommand{\arraystretch}{1.3}
\begin{tabular}{|c||c|c|c|c|}\hline
						   & $SC$      & $\#eval$ & $d_x$              & $d_f$ \\\hline\hline
NOWPAC                    & $10^{-3}$ & $64$     & $2.27 \cdot 10^{-4}$  & $1.05 \cdot 10^{-8}$\\\hline
COBYLA                    & $10^{-3}$ & $81$     & $1.81 \cdot 10^{-2}$  & $6.29 \cdot 10^{-5}$\\\hline     
NOMAD                     & $10^{-3}$ & $70$     & $0$                   & $0$\\\hline
SDPEN                     & $10^{-3}$ & $69$     & $0$                   & $0$\\\hline
GSS-NLC                   & $10^{-3}$ & $129$    & $6.12 \cdot 10^{-3}$  & $7.24 \cdot 10^{-6}$\\\hline\hline
NOWPAC                    & $10^{-4}$ & $65$     & $2.27 \cdot 10^{-4}$  & $1.05 \cdot 10^{-8}$\\\hline
COBYLA                    & $10^{-4}$ & $150$    & $8.47 \cdot 10^{-4}$  & $2.81 \cdot 10^{-7}$\\\hline     
NOMAD                     & $10^{-4}$ & $81$     & $0$                   & $0$\\\hline
SDPEN                     & $10^{-4}$ & $85$     & $0$                   & $0$\\\hline
GSS-NLC                   & $10^{-4}$ & $184$    & $4.07 \cdot 10^{-4}$  & $3.47 \cdot 10^{-8}$\\\hline\hline
NOWPAC                    & $10^{-5}$ & $76$     & $1.09 \cdot 10^{-4}$  & $2.07 \cdot 10^{-9}$\\\hline
COBYLA                    & $10^{-5}$ & $199$    & $1.05 \cdot 10^{-4}$  & $2.37 \cdot 10^{-9}$\\\hline     
SDPEN                     & $10^{-5}$ & $97$     & $0$                   & $0$\\\hline
NOMAD                     & $10^{-5}$ & $97$     & $0$                   & $0$\\\hline
GSS-NLC                   & $10^{-5}$ & $228$    & $6.35 \cdot 10^{-5}$  & $7.28 \cdot 10^{-10}$\\\hline
\end{tabular}
\end{center}
\end{table}

Next, we introduce artificial errors into the objective function. In particular, we add independent errors to every evaluation of the objective function, randomly drawn from a uniform distribution on the interval $[-\delta^f_{\max}, \delta^f_{\max}]$ with magnitude $\delta^f_{\max} \in \{10^{-2}, 10^{-3}, 10^{-4}\}$. We show the norms of the Hessians $H_k^f$ for one realization of an optimization run of NOWPAC in Figure \ref{fig:rosenbrock_noise}. 
\begin{figure}[htbp]
\begin{center}
\includegraphics[width=42mm]{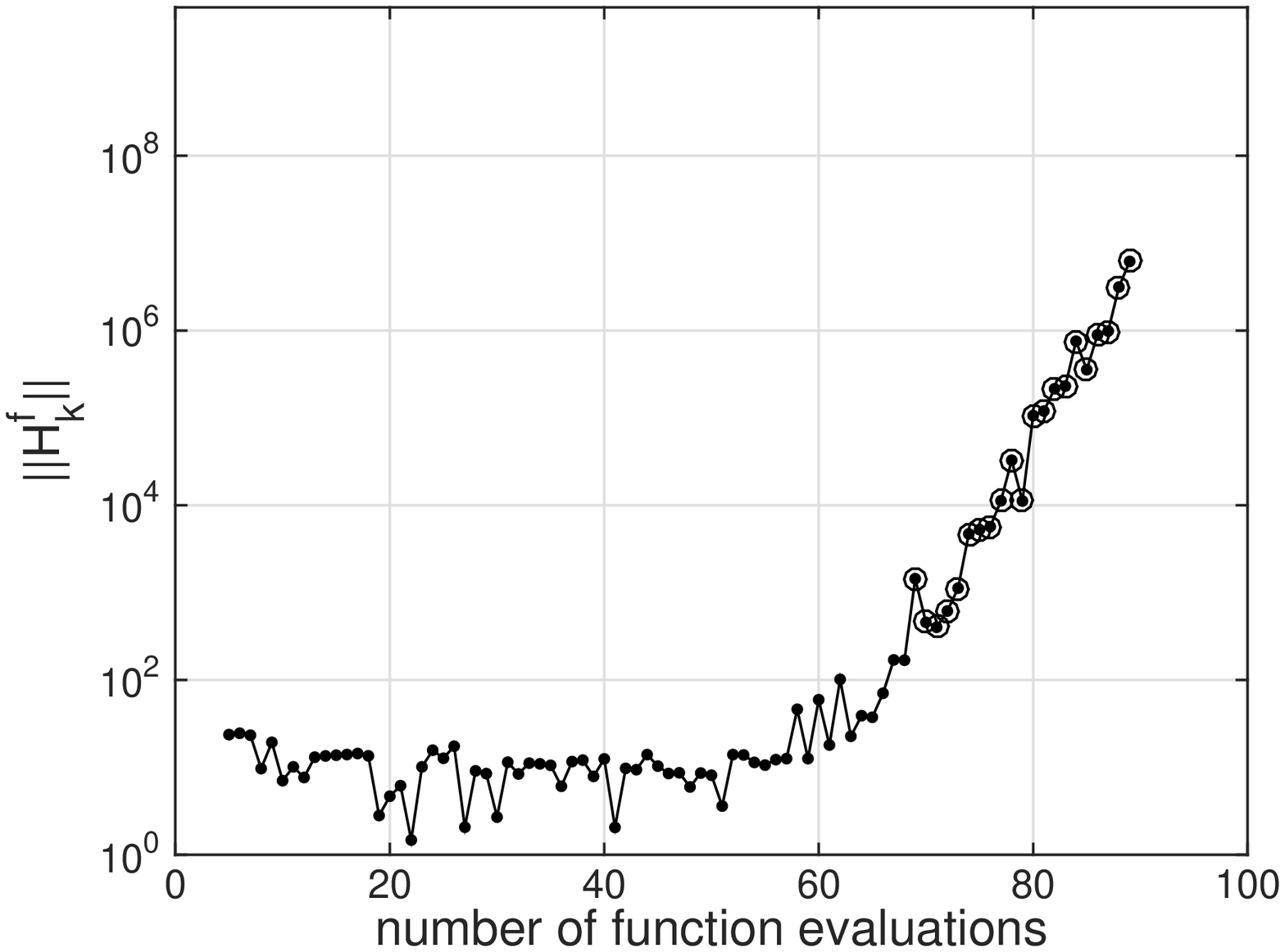}
\includegraphics[width=42mm]{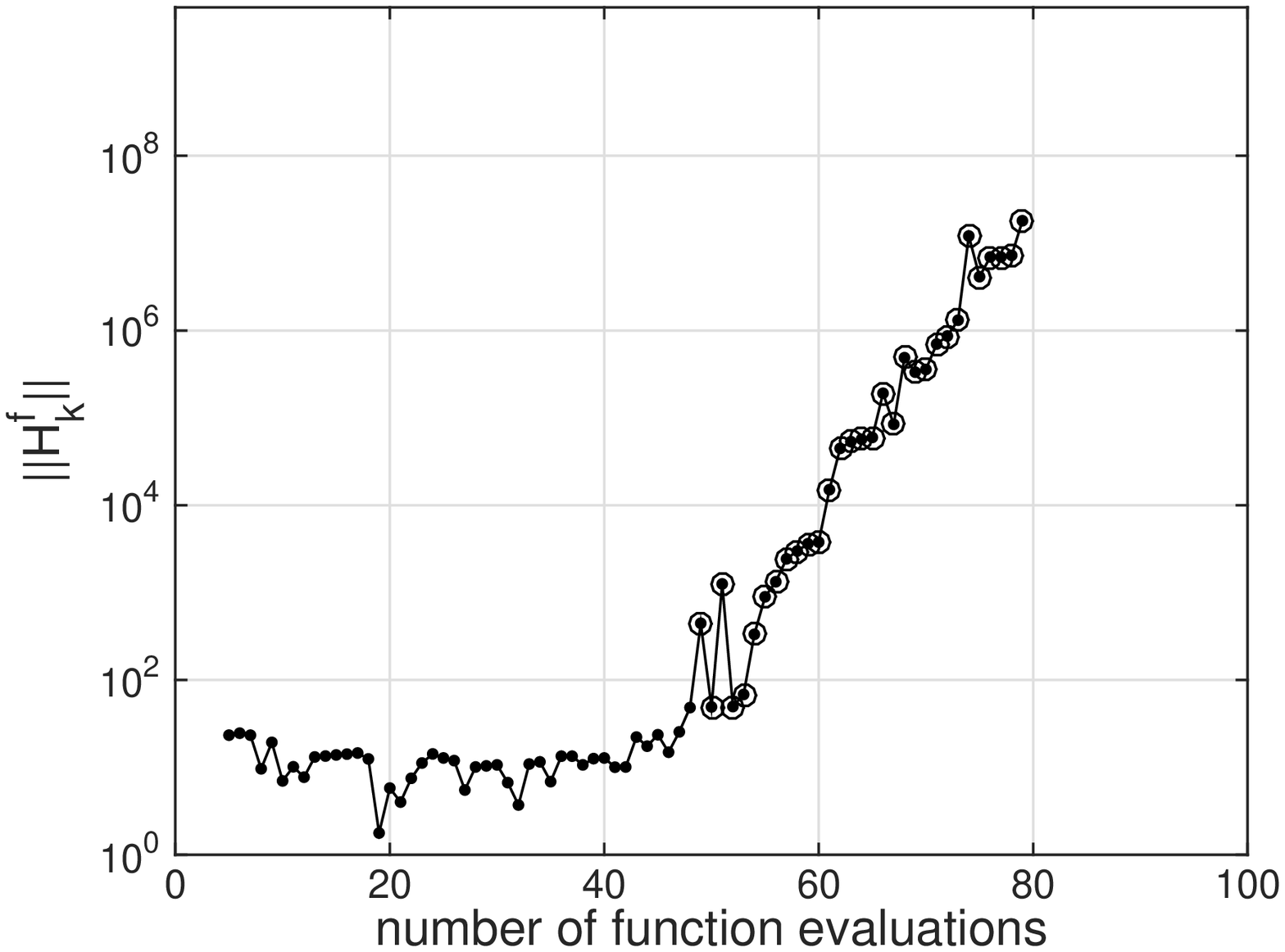}
\includegraphics[width=42mm]{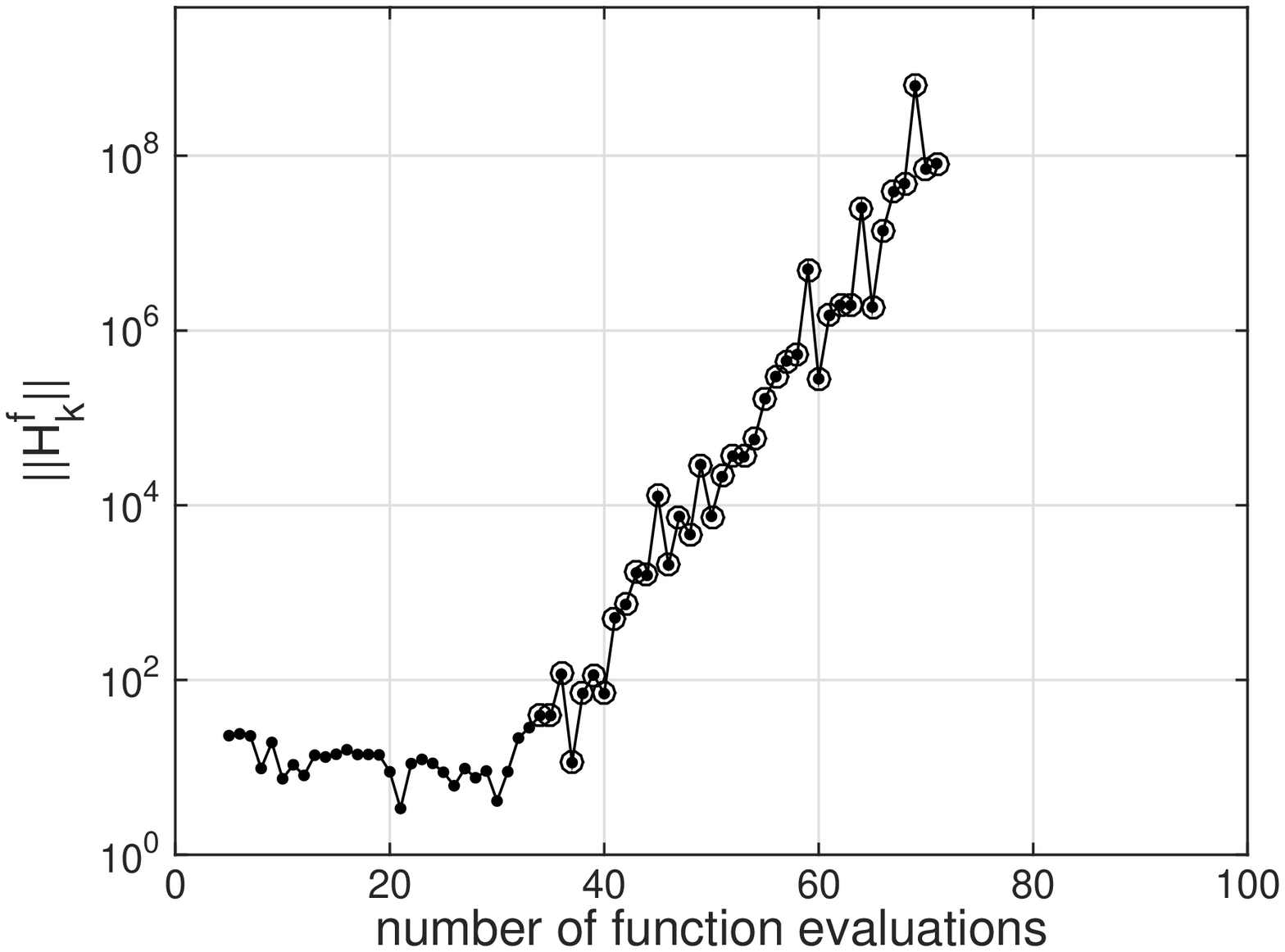}
\caption{Norms of the Hessians of the models of the objective function in (\ref{eq:rosenbrock}) for $\delta^f_{max} = 10^{-4}$ (left), $\delta^f_{max} = 10^{-3}$ (middle) and $\delta^f_{max} = 10^{-2}$ (right). Circles indicate the non-convergent regime.}\label{fig:rosenbrock_noise}
\end{center}
\end{figure}
To create these plots we switch off the early termination due to the detection of errors. (Note that NOWPAC would ordinarily stop after detecting a non-convergent iteration, as described in Section \ref{sec:noisy_function_evaluations}.) In Table \ref{tab:rosenbrock_sum_noise} we report the average distance from the approximated optimal point to the exact optimal design at early termination, as well as the corresponding average distances for the objective values. To compute these averages we ran NOWPAC $1000$ times with random samples for the errors in the objective function. All distances are computed at the iteration where NOWPAC detects the first non-convergent iteration. In the same table, we report the average number of saved function evaluations, i.e., the number of additional evaluations performed by a run with identical parameters but \textit{without} early termination due to inexact function evaluations. The numbers are rounded to the nearest integer.
\begin{table}[htbp]
\caption{Summarized performance of the application of NOWPAC to the noisy Rosenbrock minimization problem. The stopping criteria is set to $\rho_{min} = 10^{-5}$. $d_x$ and $d_f$ denote the average Euclidean distances of the approximations to the analytical solutions at early termination. $d_x^{end}$ and $d_f^{end}$ denote the average Euclidean distances of the approximations to the analytical solutions without early termination. The solution is computed $1000$ times and the average value of all outcomes is shown. $\#eval$ refers to the average number of function evaluations at early termination, whereas $\#saved$ is the average number of additional evaluations performed by the same run without early termination.}\label{tab:rosenbrock_sum_noise}
\begin{center}\footnotesize
\renewcommand{\arraystretch}{1.3}
\begin{tabular}{|c||c|c|c|c|c|c|}\hline
$\delta^f_{max}$  & $\#eval$ & $\#saved$ & $d_x$ & $d_f$ & $d_x^{end}$ & $d_f^{end}$\\\hline\hline
$10^{-2}$ & $ 33 $     &  $ 33 $      & $2.66\cdot10^{-1}$  & $1.01\cdot10^{-2}$  & $2.64\cdot10^{-1}$  & $9.48\cdot10^{-3}$\\\hline    
$10^{-3}$ & $ 56 $     &  $ 30 $      & $5.38\cdot10^{-2}$  & $7.75\cdot10^{-4}$  & $5.33\cdot10^{-2}$  & $7.79\cdot10^{-4}$\\\hline    
$10^{-4}$ & $ 71 $     &  $ 22 $      & $1.63\cdot10^{-2}$  & $8.75\cdot10^{-5}$  & $1.60\cdot10^{-2}$  & $8.60\cdot10^{-5}$\\\hline  
$10^{-5}$ & $ 80 $     &  $ 16 $      & $5.18\cdot10^{-3}$  & $9.01\cdot10^{-6}$  & $4.98\cdot10^{-3}$  & $8.59\cdot10^{-6}$\\\hline  
$0$               & $76$     &  $-$       &  $-$       &  $-$       & $1.09 \cdot 10^{-4}$  & $2.07 \cdot 10^{-9}$\\\hline
\end{tabular}
\end{center}
\end{table}
As expected, increasing the magnitude of the errors corrupts the optimal point. Moreover, the number of objective function evaluations declines significantly, even if early termination is switched off. To explain this trend, we point out that $d_f$ is roughly of the same order as the maximal magnitude of the errors, $\delta^f_{max}$. In this situation, the inexact function evaluations corrupt the acceptance ratio $r_k$ in {\tt STEP 4} of Algorithm \ref{alg:algorithm1}, misleading NOWPAC to reject steps. Using the noise indicator we are able to detect this regime and automatically terminate NOWPAC without evaluating many steps that will eventually be rejected. 
We should also point out that choosing a stopping criterion $\rho_{min}$ that is \textit{a priori} adjusted to the magnitude of the noise would also avoid many rejected steps, due to the noise overwhelming the shape of the objective. % for overly small trust region radii. 
But in the examples above we did not presume to know anything about the magnitude of the noise, i.e.,  we did not change parameters of the algorithm according to $\delta^f_{max}$. This setup is intended to mimic practical situations in which the noise is not well characterized, and hence manually tailoring the stopping criterion to the noise magnitude is not possible.

\subsection{Constrained anisotropic exponential}\label{sec:exponential}
Our second example is the constrained minimization problem
\begin{equation}\label{eq:exponential_obj}
\min\limits_{x \in X} -\exp\left(x^TDx\right)
\end{equation}
with the diagonal scaling matrix $D = \mbox{diag}(1, 2, 3, 4, 5)$ and the feasible domain
\[
X = \left\{x \in \mathbb{R}^5 \; : \; \sin\left(\|x\|^2\right) \leq \frac12, \; \left\|x - \frac{3}{8}e_5\right\| \leq \frac{3}{8}\right\},
\]
where $e_5 = (0, 0, 0, 0, 1)^T$ denotes the fifth canonical unit vector. Here we have $5$ design parameters and $2$ constraint functions. The optimal point is known to be $x^\ast = (0, 0, 0, 0, \sqrt{\arcsin(0.5)})^T  \approx (0, 0, 0, 0, 0.724)^T$ with optimal value $f^\ast =  -1.37\cdot 10^{1}$. At the optimal point, the first constraint $c_1(x) = \sin(\|x\|^2) - 0.5$ is active, so that the optimal point is critical, $\mathfrak{A}[x^\ast] = 0$, but not stationary, $\|\nabla f(x^\ast)\| = 3.96 \cdot 10^{2}$. Despite this steep gradient at $x^\ast$, the objective function exhibits a relatively flat region around the origin, where the greatest descent can be achieved by varying the last coordinate. Thus, when starting at the point $x_0 = (0.1, 0.1, 0.1, 0.1, 0.1)^T$, reducing the objective function drives the intermediate points towards the boundary of the feasible domain $X$. Once the boundary is reached, further progress towards the minimum can only be made by moving \textit{along} the boundary of $X$. The shape of this objective and feasible domain therefore constitute a useful setting in which to discuss the effectiveness of constraint handling in NOWPAC.

\begin{table}[htbp]
\caption{Summarized performance statistics of NOWPAC, COBYLA, NOMAD, SDPEN, and GSS-NLC applied to the constrained minimization problem (\ref{eq:exponential_obj}). $SC$ indicates the stopping criteria, which is the $\rho_{min}$ threshold for NOWPAC and the absolute distance in the coordinate directions for COBYLA, NOMAD, SDPEN, and GSS-NLC. $d_x$ and $d_f$ denote the Euclidean distances between the approximated and the analytical solutions.}\label{tab:exponential_sum}
\begin{center}\footnotesize
\renewcommand{\arraystretch}{1.3}
\begin{tabular}{|c||c|c|c|c|}\hline
													& $SC$      & $\#eval$        & $d_x$              & $d_f$ \\\hline\hline
NOWPAC                    & $10^{-3}$ & $59$     & $2.23 \cdot 10^{-3}$ & $2.87 \cdot 10^{-4}$\\\hline
COBYLA                    & $10^{-3}$ & $81$     & $2.50 \cdot 10^{-3}$ & $4.54 \cdot 10^{-4}$\\\hline
NOMAD                     & $10^{-3}$ & $574$    & $6.37 \cdot 10^{-2}$ & $6.57 \cdot 10^{-2}$\\\hline
SDPEN                     & $10^{-3}$ & $153$    & $4.18. \cdot 10^{-1}$ & $9.49 $\\\hline
GSS-NLC                   & $10^{-3}$ & $1610$   & $1.51 \cdot 10^{-2}$ & $3.70 \cdot 10^{-3}$\\\hline\hline
NOWPAC                    & $10^{-4}$ & $96$     & $3.67 \cdot 10^{-4}$ & $4.86 \cdot 10^{-6}$\\\hline
COBYLA                    & $10^{-4}$ & $130$    & $8.74 \cdot 10^{-5}$ & $1.86 \cdot 10^{-6}$\\\hline
NOMAD                     & $10^{-4}$ & $980$    & $4.10 \cdot 10^{-2}$ & $2.52 \cdot 10^{-2}$\\\hline
SDPEN                     & $10^{-4}$ & $192$    & $4.18 \cdot 10^{-1}$ & $9.49 $\\\hline
GSS-NLC                   & $10^{-4}$ & $4073$   & $9.60 \cdot 10^{-3}$ & $1.28 \cdot 10^{-3}$\\\hline\hline
NOWPAC                    & $10^{-5}$ & $128$    & $3.15 \cdot 10^{-5}$ & $1.63 \cdot 10^{-8}$\\\hline
COBYLA                    & $10^{-5}$ & $160$    & $1.83 \cdot 10^{-5}$ & $1.14 \cdot 10^{-8}$\\\hline   
NOMAD                     & $10^{-5}$ & $2842$   & $1.94 \cdot 10^{-2}$ & $5.28 \cdot 10^{-3}$\\\hline
SDPEN                     & $10^{-5}$ & $221$    & $4.18. \cdot 10^{-1}$ & $9.49 $\\\hline
GSS-NLC                   & $10^{-5}$ & $9321$   & $2.23 \cdot 10^{-3}$ & $4.49 \cdot 10^{-4}$\\\hline
\end{tabular}
\end{center}
\end{table}

The performances of NOWPAC, COBYLA, NOMAD, SDPEN, and GSS-NLC are summarized in Table \ref{tab:exponential_sum}. Almost all methods, except SDPEN which did not converge to a critical solution, produce reasonable approximations of the optimal point and the corresponding minimal value of the objective function. However, for all stopping thresholds, NOWPAC requires the smallest number of evaluations of the objective function. In particular, GSS-NLC requires many function evaluations---first, because it must explore the design space in all five coordinate directions, and second, because it repeatedly explores the space while adaptively choosing a suitably high penalty parameter to meet the prescribed tolerances. Also NOMAD requires many function evaluations to find an approximate solution; the achieved accuracy is far less than the solution computed by NOWPAC and COBYLA. On this test example, SDPEN converged prematurely to a non-critical solution.

Next, we again introduce artificial errors of different magnitudes into the objective function and constraint function evaluations. We solve (\ref{eq:exponential_obj}) $1000$ times and report, in Table \ref{tab:exponential_sum_noise}, the average number of function evaluations (for termination after the first noisy iteration), the average number of saved function evaluations, as well as the average absolute error in the computed design and the corresponding objective value. We see that for noise in the objective function in particular, the number of saved function evaluations is significant.

\begin{table}[htbp]
\caption{Summarized performance of the application of NOWPAC to the noisy constrained minimization problem (\ref{eq:exponential_obj}). The stopping criteria is set to $\rho_{min} = 10^{-5}$. $d_x$ and $d_f$ denote the average Euclidean distances of the approximations to the analytical solutions at early termination. $d_x^{end}$ and $d_f^{end}$ denote the average Euclidean distances of the approximations to the analytical solutions without early termination. The solution is computed $1000$ times and the average value of all outcomes is shown. $\#eval$ refers to the average number of function evaluations at early termination, whereas $\#saved$ is the average number of additional evaluations (saved evaluations) performed by the same run without early termination.}\label{tab:exponential_sum_noise}
\begin{center}\footnotesize
\renewcommand{\arraystretch}{1.3}
\begin{tabular}{|c|c||c|c|c|c|c|c|}\hline
$\delta^f_{max}$  & $\delta^c_{max}$  & $\#eval$ & $\#saved$ & $d_x$ & $d_f$ & $d_x^{end}$ & $d_f^{end}$\\\hline\hline
$10^{-2}$   & $10^{-5}$  & $ 62$&$21 $     & $6.43\cdot10^{-2}$  & $1.31\cdot10^{-1}$  & $6.22\cdot10^{-2}$  & $8.04\cdot10^{-2}$\\\hline 
$10^{-3}$   & $10^{-5}$  & $ 75$&$5 $   & $2.81\cdot10^{-2}$  & $1.63\cdot10^{-2}$  & $2.81\cdot10^{-2}$  & $1.60\cdot10^{-2}$\\\hline 
$10^{-2}$   & $10^{-4}$  & $ 60$&$11 $     & $7.51\cdot10^{-2}$  & $1.39\cdot10^{-1}$  & $7.45\cdot10^{-2}$  & $1.28\cdot10^{-1}$\\\hline 
$10^{-3}$   & $10^{-4}$  & $ 62$&$4 $   & $6.31\cdot10^{-2}$  & $9.28\cdot10^{-2}$  & $6.28\cdot10^{-2}$  & $9.00\cdot10^{-2}$\\\hline 
$10^{-3}$   & $0$        & $ 75$&$22 $    & $2.09\cdot10^{-2}$  & $9.96\cdot10^{-3}$  & $2.02\cdot10^{-2}$  & $8.32\cdot10^{-3}$\\\hline 
$0$         & $10^{-4}$  & $ 62$&$4 $   & $5.90\cdot10^{-2}$  & $8.30\cdot10^{-2}$  & $5.90\cdot10^{-2}$  & $8.24\cdot10^{-2}$\\\hline 
$0$         & $0$        & $128$&$-$      & $-$ & $-$ & $3.15 \cdot 10^{-5}$ & $1.63 \cdot 10^{-8}$\\\hline
\end{tabular}
\end{center}
\end{table}

\subsection{Schittkowski benchmark problems}\label{sec:benchmark_test}
In this section we compare the efficiency of NOWPAC with the derivative-free optimization codes COBYLA, NOMAD, SDPEN, and GSS-NLC on a broader set of benchmark problems. For this comparison we chose nonlinear constrained optimization problems from the Schittkowski benchmark problem collection for nonlinear programming \cite{Schittkowski2008}; see also \cite{Hock1981, Schittkowski1987}. The number of design variables in these problems varies between $2$ and $15$, and the number of nonlinear constraints varies between $1$ and $10$.  Results from all the optimization codes run with stopping criteria $SC = 10^{-3}$ and $SC = 10^{-5}$ are summarized in Tables \ref{tab:schittkowskiresults_1e_3} and \ref{tab:schittkowskiresults_1e_5}. In all the test problems, the exact optimal design and the optimal value are known; thus we can report the absolute error $d_x$ in the optimal designs as well as the absolute and relative errors, $d_f^{abs}$ and $d_f^{rel}$, in the objective functions. We see that in many test examples, NOWPAC performs---in terms of required function evaluations---significantly better than the other optimization codes. Moreover, in all cases, the accuracy of the computed optimal solution is better than or comparable to that of the other codes.
 
Note that in some of the test problems, the direct search methods (i.e., NOMAD, SDPEN, and GSS-NLC) find the exact solution. In our benchmark set this is the case if the initial and optimal designs have simple integer values. Nevertheless, despite the fact that these codes find the optimal solutions relatively early, they need a significant number of function evaluations to certify the optimality of the solution. These codes' fast descent of the objective function unfortunately gets lost in many other cases, particularly the moderate-dimensional examples (test problems $100$, $113$, and $285$). In contrast, NOWPAC consistently yields a good reduction of the objective even for increasing dimensionality of the optimization problem.

\begin{table}[htbp]
\caption{Comparison of performance of the optimizers NOWPAC, COBYLA, NOMAD, SDPEN, and GSS-NLC on
test problems in the Hock and Schittkowski benchmark set. TP denotes
the number of the test problem, $n$ and $r$ the number of design variables and nonlinear constraints. The absolute error in the optimal design, and the absolute and relative errors in the optimal objective values are shown in the columns $d_x$, $d_f^{abs}$, and $d_f^{rel}$ respectively. For all solvers a stopping threshold of $10^{-3}$ is used.}\label{tab:schittkowskiresults_1e_3}
\begin{center}\footnotesize
\renewcommand{\arraystretch}{1.3}
\begin{tabular}{|c|c|c||c||c|c|c|c|}\hline
TP & $n$ & $r$ & SOLVER & $\#eval$ & $d_x$ & $d_f^{abs}$ & $d_f^{rel}$\\\hline\hline
\multirow{5}{*}{$29$}&\multirow{5}{*}{$3$} &\multirow{5}{*}{$1$} & NOWPAC & $50$ & $1.0441\cdot10^{-4}$ & $3.2215\cdot10^{-7}$ & $1.4237\cdot10^{-8}$\\\cline{4-8}
& & & COBYLA & $73$ & $7.6908\cdot10^{-4}$ & $5.5365\cdot10^{-7}$ & $2.4468\cdot10^{-8}$\\\cline{4-8}
&&& NOMAD & $277$ & $2.9345\cdot10^{-3}$ & $4.0290\cdot10^{-5}$ & $1.7806\cdot10^{-6}$ \\\cline{4-8}
& & & SDPEN & $115$ & $1.3449$ & $5.8696$ & $2.5940\cdot10^{-1}$\\\cline{4-8}
&&& GSS-NLC & $221$ & $2.9625\cdot10^{-1}$ & $1.6742\cdot10^{-1}$ & $7.3989\cdot10^{-3}$ \\\hline\hline
\multirow{5}{*}{$43$}&\multirow{5}{*}{$4$} &\multirow{5}{*}{$3$}    & NOWPAC & $66$ & $9.8067\cdot10^{-6}$ & $4.4098\cdot10^{-9}$ & $1.0022\cdot10^{-10}$\\\cline{4-8}
& & & COBYLA & $79$ & $1.4830\cdot10^{-3}$ & $1.2760\cdot10^{-6}$ & $2.9000\cdot10^{-8}$\\\cline{4-8}
&&& NOMAD & $234$ & $0$ & $0$ & $0$ \\\cline{4-8}
& & & SDPEN & $174$ & $1.8322$ & $2.1379\cdot10^{1}$ & $4.8588\cdot10^{-1}$\\\cline{4-8}
&&& GSS-NLC & $1111$ & $1.0311\cdot10^{-1}$ & $5.0000\cdot10^{-2}$ & $1.1364\cdot10^{-3}$ \\\hline\hline
\multirow{5}{*}{$100$}&\multirow{5}{*}{$7$} &\multirow{5}{*}{$4$}   & NOWPAC & $155$ & $1.0563\cdot10^{-2}$ & $2.8696\cdot10^{-4}$ & $4.2161\cdot10^{-7}$\\\cline{4-8}
& & & COBYLA & $237$ & $1.3533\cdot10^{-2}$ & $1.1572\cdot10^{-3}$ & $1.7002\cdot10^{-6}$\\\cline{4-8}
&&& NOMAD & $804$ & $2.5400\cdot10^{-1}$ & $2.5474\cdot10^{-1}$ & $3.7427\cdot10^{-4}$ \\\cline{4-8}
& & & SDPEN & $251$ & $8.3322\cdot10^{-1}$ & $4.1820$ & $6.1444\cdot10^{-3}$\\\cline{4-8}
&&& GSS-NLC & $1269$ & $7.6222\cdot10^{-1}$ & $3.5699$ & $5.2451\cdot10^{-3}$ \\\hline\hline
\multirow{5}{*}{$113$}&\multirow{5}{*}{$10$} &\multirow{5}{*}{$8$}  & NOWPAC & $141$ & $7.9201\cdot10^{-4}$ & $3.2979\cdot10^{-6}$ & $1.3568\cdot10^{-7}$\\\cline{4-8}
& & & COBYLA & $335$ & $9.5021\cdot10^{-3}$ & $1.8091\cdot10^{-4}$ & $7.4429\cdot10^{-6}$\\\cline{4-8}
&&& NOMAD & $1199$ & $1.7572$ & $5.5933$ & $2.3012\cdot10^{-1}$ \\\cline{4-8}
& & & SDPEN & $392$ & $1.1885$ & $4.0313$ & $1.6586\cdot10^{-1}$\\\cline{4-8}
&&& GSS-NLC & $4193$ & $8.1851\cdot10^{-1}$ & $2.1238$ & $8.7376\cdot10^{-2}$ \\\hline\hline
\multirow{5}{*}{$227$}&\multirow{5}{*}{$2$} &\multirow{5}{*}{$2$}   & NOWPAC & $18$ & $6.5285\cdot10^{-6}$ & $8.3104\cdot10^{-6}$ & $8.3104\cdot10^{-6}$\\\cline{4-8}
& & & COBYLA & $18$ & $5.8463\cdot10^{-5}$ & $1.0860\cdot10^{-4}$ & $1.0860\cdot10^{-4}$\\\cline{4-8}
&&& NOMAD & $102$ & $0$ & $0$ & $0$ \\\cline{4-8}
& & & SDPEN & $90$ & $4.0569\cdot10^{-3}$ & $5.7538\cdot10^{-3}$ & $5.7538\cdot10^{-3}$\\\cline{4-8}
&&& GSS-NLC & $393$ & $0$ & $0$ & $0$ \\\hline\hline
\multirow{5}{*}{$228$}&\multirow{5}{*}{$2$} &\multirow{5}{*}{$2$}  & NOWPAC & $28$ & $9.8407\cdot10^{-5}$ & $1.1405\cdot10^{-8}$ & $3.8017\cdot10^{-9}$\\\cline{4-8}
& & & COBYLA & $48$ & $1.9082\cdot10^{-3}$ & $3.7562\cdot10^{-6}$ & $1.2521\cdot10^{-6}$\\\cline{4-8}
&&& NOMAD & $101$ & $0$ & $0$ & $0$ \\\cline{4-8}
& & & SDPEN & $4$ & $3.0000$ & $3.0000$ & $1.0000$\\\cline{4-8}
&&& GSS-NLC & $287$ & $0$ & $0$ & $0$ \\\hline\hline
\multirow{5}{*}{$264$}&\multirow{5}{*}{$4$} &\multirow{5}{*}{$3$}   & NOWPAC & $53$ & $1.6402\cdot10^{-4}$ & $3.0458\cdot10^{-7}$ & $6.9222\cdot10^{-9}$\\\cline{4-8}
& & & COBYLA & $77$ & $1.3421\cdot10^{-3}$ & $1.5297\cdot10^{-6}$ & $3.4765\cdot10^{-8}$\\\cline{4-8}
&&& NOMAD & $233$ & $0$ & $0$ & $0$ \\\cline{4-8}
& & & SDPEN & $174$ & $1.8322$ & $2.1379\cdot10^{1}$ & $4.8588\cdot10^{-1}$\\\cline{4-8}
&&& GSS-NLC & $1352$ & $0$ & $0$ & $0$ \\\hline\hline
\multirow{5}{*}{$285$}&\multirow{5}{*}{$15$} &\multirow{5}{*}{$10$} & NOWPAC &  $177$ & $7.8426\cdot10^{-5}$ & $1.0413\cdot10^{-5}$ & $1.2618\cdot10^{-9}$\\\cline{4-8}
& & & COBYLA & $362$ & $1.5656\cdot10^{-3}$ & $3.3056\cdot10^{-4}$ & $4.0058\cdot10^{-8}$\\\cline{4-8}
&&& NOMAD & $787$ & $3.3862\cdot10^{-1}$ & $3.6380\cdot10^{1}$ & $4.4087\cdot10^{-3}$ \\\cline{4-8}
& & & SDPEN & $650$ & $3.1240$ & $3.7707\cdot10^{3}$ & $4.5695\cdot10^{-1}$\\\cline{4-8}
&&& GSS-NLC & $4647$ & $1.2390$ & $5.9800\cdot10^{2}$ & $7.2467\cdot10^{-2}$ \\\hline
\end{tabular}
\end{center}
\end{table}
\begin{table}[htbp]
\caption{Comparison of performance of the optimizers NOWPAC, COBYLA, NOMAD, SDPEN, and GSS-NLC on
test problems in the Hock and Schittkowski benchmark set. TP denotes
the number of the test problem, $n$ and $r$ the number of design variables and nonlinear constraints. The absolute error in the optimal design and the absolute and relative errors in the optimal objective values are shown in the columns $d_x$, $d_f^{abs}$, and $d_f^{rel}$ respectively. For all solvers a stopping threshold of $10^{-5}$ is used.}\label{tab:schittkowskiresults_1e_5}
\begin{center}\footnotesize
\renewcommand{\arraystretch}{1.3}
\begin{tabular}{|c|c|c||c||c|c|c|c|}\hline
TP & $n$ & $r$ & SOLVER & $\#eval$ & $d_x$ & $d_f^{abs}$ & $d_f^{rel}$\\\hline\hline
\multirow{5}{*}{$29$}&\multirow{5}{*}{$3$} &\multirow{5}{*}{$1$} & NOWPAC & $58$ & $1.2405\cdot10^{-5}$ & $3.9934\cdot10^{-10}$ & $1.7648\cdot10^{-11}$\\\cline{4-8}
 & &  & COBYLA & $117$ & $8.7405\cdot 10^{-6}$ & $1.9876\cdot 10^{-10}$ & $8.7839\cdot 10^{-12}$\\\cline{4-8}
 &  &  & NOMAD & $623$ & $2.8503\cdot 10^{-4}$ & $1.8000\cdot 10^{-7}$  & $7.9550\cdot 10^{-9}$\\\cline{4-8}
   & &  & SDPEN   & $154$ & $1.3450$ & $5.8661$ & $2.5925\cdot 10^{-1}$\\\cline{4-8}
 &  &  & GSS-NLC & $696$ & $2.9542\cdot 10^{-1}$ & $1.5742\cdot 10^{-1}$  & $6.9569\cdot 10^{-3}$\\\hline\hline
\multirow{5}{*}{$43$}&\multirow{5}{*}{$4$} &\multirow{5}{*}{$3$} & NOWPAC & $74$ & $9.8067\cdot 10^{-6}$ & $4.4098\cdot 10^{-9}$ & $1.0022\cdot 10^{-10}$\\\cline{4-8}
 & & & COBYLA & $128$ & $5.2627\cdot 10^{-6}$ & $1.2718\cdot 10^{-9}$ & $9.1647\cdot 10^{-10}$\\\cline{4-8}
 &  & &NOMAD & $330$ & $0$ & $0$  & $0$\\\cline{4-8}
 & & & SDPEN & $228$ & $1.8317$ & $2.1367\cdot 10^{1}$ & $4.8562\cdot 10^{-1}$\\\cline{4-8}
 &  & & GSS-NLC & $2332$ & $1.0232\cdot 10^{-1}$ & $5.0000\cdot 10^{-2}$  & $1.1364\cdot 10^{-3}$\\\hline\hline
\multirow{5}{*}{$100$}&\multirow{5}{*}{$7$} &\multirow{5}{*}{$4$} & NOWPAC & $238$ & $6.7890\cdot 10^{-4}$ & $1.2721\cdot 10^{-6}$ & $1.8690\cdot 10^{-9}$\\\cline{4-8}
 & &  & COBYLA & $729$ & $5.0509\cdot 10^{-4}$ & $6.2378\cdot 10^{-7}$ & $9.1647\cdot 10^{-10}$\\\cline{4-8}
 &  &  & NOMAD & $2606$ & $2.2263\cdot 10^{-1}$ & $1.2558\cdot 10^{-1}$  & $1.8450\cdot 10^{-4}$\\\cline{4-8}
    & &  & SDPEN & $360$ & $8.3304\cdot 10^{-1}$ & $4.1351$ & $6.0753\cdot 10^{-3}$\\\cline{4-8}
 &  &  & GSS-NLC & $2769$ & $7.6247\cdot 10^{-1}$ & $3.5699$ & $5.2451\cdot 10^{-3}$ \\\hline\hline
\multirow{5}{*}{$113$}&\multirow{5}{*}{$10$} &\multirow{5}{*}{$8$} & NOWPAC & $188$ & $1.6343\cdot 10^{-4}$ & $3.4602\cdot 10^{-8}$ & $1.4236\cdot 10^{-9}$\\\cline{4-8}
 & &  & COBYLA & $635$ & $1.1470\cdot 10^{-4}$ & $2.3968\cdot 10^{-8}$ & $9.8609\cdot 10^{-10}$\\\cline{4-8}
 &  &  & NOMAD & $1715$ & $1.7512$ & $5.5030$ & $2.2640\cdot 10^{-1}$ \\\cline{4-8}
    & & & SDPEN & $531$ & $1.1885$ & $4.0045$ & $1.6475\cdot 10^{-1}$\\\cline{4-8}
 &  & & GSS-NLC & $7172$ & $8.2086\cdot 10^{-1}$ & $2.1338$ & $8.7788\cdot 10^{-2}$\\\hline\hline
\multirow{5}{*}{$227$}&\multirow{5}{*}{$2$} &\multirow{5}{*}{$2$} & NOWPAC & $31$ & $9.1139\cdot 10^{-12}$ & $1.2971\cdot 10^{-11}$ & $1.2971\cdot 10^{-11}$\\\cline{4-8}
 & & & COBYLA & $26$ & $7.0430\cdot 10^{-9}$ & $1.1950\cdot 10^{-8}$ & $1.1950\cdot 10^{-8}$\\\cline{4-8}
 &  &  & NOMAD & $158$ & $0$ & $0$  & $0$\\\cline{4-8}
    & &  & SDPEN & $130$ & $2.2254\cdot 10^{-5}$ & $3.1472\cdot 10^{-5}$ & $3.1472\cdot 10^{-5}$\\\cline{4-8}
 &  &  & GSS-NLC & $930$ & $0$ & $0$  & $0$\\\hline\hline
\multirow{5}{*}{$228$}&\multirow{5}{*}{$2$} &\multirow{5}{*}{$2$} & NOWPAC & $31$ & $9.8407\cdot 10^{-5}$ & $1.1405\cdot 10^{-8}$ & $3.8017\cdot 10^{-9}$\\\cline{4-8}
 & & & COBYLA & $67$ & $1.2070\cdot 10^{-5}$ & $1.3663\cdot 10^{-10}$ & $4.5543\cdot 10^{-11}$\\\cline{4-8}
&  & & NOMAD & $174$ & $0$ & $0$ & $0$ \\\cline{4-8}
   & &  & SDPEN & $84$ & $0$ & $0$ & $0$\\\cline{4-8}
 &  &  & GSS-NLC & $779$ & $0$ & $0$  & $0$\\\hline\hline
\multirow{5}{*}{$264$}&\multirow{5}{*}{$4$} &\multirow{5}{*}{$3$} & NOWPAC & $63$ & $5.9984\cdot 10^{-6}$ & $2.4949\cdot 10^{-10}$ & $5.6703\cdot 10^{-12}$\\\cline{4-8}
 & & & COBYLA & $135$ & $1.3461\cdot 10^{-5}$ & $1.4405\cdot 10^{-10}$ & $3.2738\cdot 10^{-12}$\\\cline{4-8}
 &  & & NOMAD & $349$ & $0$ & $0$  & $0$\\\cline{4-8}
    & & & SDPEN & $228$ & $1.8317$ & $2.1367\cdot 10^{1}$ & $4.8562\cdot 10^{-1}$\\\cline{4-8}
 &  & & GSS-NLC & $2441$ & $0$ & $0$  & $0$\\\hline\hline
\multirow{5}{*}{$285$}&\multirow{5}{*}{$15$} &\multirow{5}{*}{$10$} & NOWPAC & $209$ & $1.9381\cdot 10^{-5}$ & $3.5764\cdot 10^{-7}$ & $4.3339\cdot 10^{-11}$\\\cline{4-8}
 & & & COBYLA & $614$ & $1.6457\cdot 10^{-5}$ & $2.4447\cdot 10^{-8}$ & $2.9626\cdot 10^{-12}$\\\cline{4-8}
&  &  & NOMAD  & $1246$ & $3.3808\cdot 10^{-1}$ & $3.5863\cdot 10^{1}$  & $4.3460\cdot 10^{-3}$\\\cline{4-8}
   & &  & SDPEN  & $862$ & $3.1240$ & $3.7707\cdot 10^{3}$ & $4.5694\cdot 10^{-1}$\\\cline{4-8}
 &  &  & GSS-NLC  & $9745$ & $1.2307$ & $5.9200\cdot 10^{2}$  & $7.1740\cdot 10^{-2}$\\\hline
\end{tabular}
\end{center}
\end{table}

\subsection{Tar removal process model}\label{sec:num_results_tar_removal_process}
We now discuss optimization of a tar removal process, shown schematically in Figure \ref{fig:tar_removal_process}, which is part of the production of synthesis gas (syngas) in a biomass to liquid (BTL) plant \cite{Srinivas2013}. 
\begin{figure}[!htb]
\vspace{2mm}
\begin{center}
\includegraphics[bb=0 10 396 95, scale=.7]{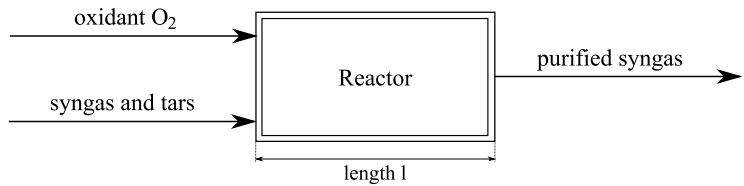}
\vspace{2mm}
\caption{Schematic of the tar removal process in syngas production. Tar-polluted syngas enters the reactor from the left together with oxygen. Chemical reactions to reduce the tars take place in the reactor, and purified syngas exits the process at right.}\label{fig:tar_removal_process}
\end{center}
\end{figure}
The objective is to maximize the flow rate of purified syngas, $\mathcal{F}_{s}$, at the outlet of the reactor by removing tar from the inlet stream as much as possible. The design parameters for the process are the length, $x_1 = l$, of the reactor and the inflow rate of oxygen, $x_2 = \mathcal{F}_{O_2}$. A call to the tar removal process simulator yields the outputs
\[
\left(\mathcal{F}_{s}(x_1, x_2), T_{out}(x_1, x_2)\right),
\]
where $T_{out}$ is the temperature of the purified syngas at the outlet. We remark that the removal of tar is implicitly achieved by maximizing $\mathcal{F}_{s}$ and is therefore not explicitly included in the objective function. The set of feasible design parameters is restricted by physical and economical constraints, as well as safety standards for the operation of the BTL plant. 
On the one hand, the length of the reactor has to be sufficiently large to contain the syngas and allow it to react. On the other hand, building too large of a reactor would result in unallowable material costs. These considerations yield the restrictions $0.5\, \text{[m]} \leq x_1 \leq 2\, \text{[m]}$ for the extent of the reactor. 
The flow rate of the oxygen at the reactor inlet must also obey constraints. A lower bound must be respected in order to sustain the reformer process, while an upper bound is again dictated from an economical perspective, to limit operational costs. We thus impose the constraints $1334375\, \text{[kmol/h]} \leq x_2 \leq 3125000\, \text{[kmol/h]}$. 
Finally, because the reactor vessel might fail when the outlet temperature $T_{out}$ exceeds a limit of $1680$ Kelvin, safety concerns impose the constraint on the temperature $T_{out} \leq 1680\, \text{[K]}$. In summary, the feasible domain $X = \{(x_1, x_2)^T \in \mathbb{R}^2 \; : \; c(x_1, x_2) \leq 0\}$ is given by the constraints 
\[
c(x_1, x_2) = \left(\begin{array}{c}
T_{out}(x_1, x_2) - 1680\\
0.5 - x_1\\
x_1 - 2\\
1334375 - x_2\\
x_2 - 3125000\\
\end{array}\right).
\]
We note that the constraints on the length of the reactor and the flow rate of the oxygen are simple box constraints. However, the constraint on the temperature is more involved since $T_{out}$ is only given by black-box evaluations of the tar removal process simulation. Overall, the constrained optimization problem can be stated as
\begin{equation}\label{eq:gov_eq_tarremovalprocess_opt}
\max\limits_{c(x_1, x_2) \leq 0} \mathcal{F}_{s}(x_1, x_2).
\end{equation}

We choose the starting point $x = (1.0, 1.5\cdot10^6)^T$ and rescale the second design parameter by $10^{-6}$ so that both design parameters are of the same order of magnitude. This scaling reduces the anisotropy of the elliptical trust region and is tailored to the design space $X$. We summarize the performances of NOWPAC, COBYLA, NOMAD, SDPEN, and GSS-NLC in Table \ref{tab:tarremoval_sum}. 
\begin{table}[htbp]
\caption{Summarized performance statistics of NOWPAC, COBYLA, NOMAD, SDPEN, and GSS-NLC in finding an optimal design for the tar removal process model. $SC$ indicates the stopping criteria, which is the $\rho_{min}$ threshold for NOWPAC and the absolute distance in the coordinate directions for COBYLA, NOMAD, SDPEN, and GSS-NLC. For NOWPAC, in parentheses behind the number of model evaluations, we report the number of evaluations saved due to early termination by the error indicator.}\label{tab:tarremoval_sum}
\begin{center}\footnotesize
\renewcommand{\arraystretch}{1.3}
\begin{tabular}{|c||c|c|c|c|}\hline
					& $SC$      & $\#eval$ & $x^\ast$                        & $f^\ast$ \\\hline\hline
%NOWPAC    & $10^{-3}$ & $18$($20$)  & $(2.0000, 2.2199\cdot 10^{6})^T$ & $3.7957 \cdot 10^{3}$ \\\hline
NOWPAC    & $10^{-3}$ & $19$($15$)  & $(2.0000, 2.2199\cdot 10^{6})^T$ & $3.7957 \cdot 10^{3}$ \\\hline
COBYLA    & $10^{-3}$ & $24$        & $(2.0000, 2.2199\cdot 10^{6})^T$ & $3.7953 \cdot 10^{3}$ \\\hline
NOMAD     & $10^{-3}$ & $212$          & $(2.0000, 2.2199\cdot 10^{6})^T$ & $3.7957 \cdot 10^{3}$ \\\hline
SDPEN     & $10^{-3}$ & $53$        & $(2.0000, 2.2188\cdot 10^{6})^T$ & $3.7941 \cdot 10^{3}$ \\\hline
GSS-NLC      & $10^{-3}$ & $102$       & $(2.0000, 2.2198\cdot 10^{6})^T$ & $3.7956 \cdot 10^{3}$ \\\hline\hline
%NOWPAC    & $10^{-4}$ & $18$($28$) & $(2.0000, 2.2199\cdot 10^{6})^T$ & $3.7957 \cdot 10^{3}$ \\\hline
NOWPAC    & $10^{-4}$ & $19$($22$) & $(2.0000, 2.2199\cdot 10^{6})^T$ & $3.7957 \cdot 10^{3}$ \\\hline
COBYLA    & $10^{-4}$ & $26$        & $(2.0000, 2.2201\cdot 10^{6})^T$ & $3.7961 \cdot 10^{3}$ \\\hline
NOMAD     & $10^{-4}$ & $232$          & $(2.0000, 2.2199\cdot 10^{6})^T$ & $3.7957 \cdot 10^{3}$ \\\hline
SDPEN     & $10^{-4}$ & $63$          & $(2.0000, 2.2199 \cdot 10^{6})^T$ & $3.7959 \cdot 10^{3}$ \\\hline
GSS-NLC      & $10^{-4}$ & $197$       & $(2.0000, 2.2198\cdot 10^{6})^T$ & $3.7957 \cdot 10^{3}$ \\\hline\hline
%NOWPAC    & $10^{-5}$ & $18$($37$) & $(2.0000, 2.2199\cdot 10^{6})^T$ & $3.7957 \cdot 10^{3}$ \\\hline
NOWPAC    & $10^{-5}$ & $19$($27$) & $(2.0000, 2.2199\cdot 10^{6})^T$ & $3.7957 \cdot 10^{3}$ \\\hline
COBYLA    & $10^{-5}$ & $32$       & $(2.0000, 2.2199\cdot 10^{6})^T$ & $3.7957 \cdot 10^{3}$ \\\hline
NOMAD     & $10^{-5}$ & $266$      & $(2.0000, 2.2199\cdot 10^{6})^T$ & $3.7957 \cdot 10^{3}$ \\\hline
SDPEN     & $10^{-5}$ & $71$       & $(2.0000, 2.2199\cdot 10^{6})^T$ & $3.7957 \cdot 10^{3}$ \\\hline
GSS-NLC      & $10^{-5}$ & $401$   & $(2.0000, 2.2199\cdot 10^{6})^T$ & $3.7957 \cdot 10^{3}$ \\\hline
\end{tabular}
\end{center}
\end{table}
We use NOWPAC's noise indicator and early termination feature in the solution of the tar removal problem. In Table \ref{tab:tarremoval_sum} we report the number of function evaluations as well as the optimal designs and objective values at termination; we also state the number of saved function evaluations (for NOWPAC) in parentheses. NOMAD and GSS-NLC require considerably more objective function evaluations than NOWPAC, COBYLA, and SDPEN. We remark that COBYLA proposes designs that are not actually feasible, i.e., $T_{out}(x^\ast) -1680 = 7.4 \cdot 10^{-3}$ and $T_{out}(x^\ast) -1680 = 2.0 \cdot 10^{-5}$ for the stopping criteria $SC = 10^{-3}$ and $SC = 10^{-5}$ respectively. Constraint violations like this are, from a practical perspective, not desirable since they may cause failure of the whole reactor. We remark that NOMAD always proposes a feasible optimal design, whereas this is not guaranteed in SDPEN and GSS-NLC. In this example, however, both of the latter codes yield feasible solutions.

The tar removal process model is only available as a black-box simulator that exhibits irreducible numerical errors in its evaluations. Thus, as we have seen in the previous test examples, detection of inexact objective and constraint evaluations is necessary to prevent NOWPAC from performing superfluous iterations. We illustrate the errors of the black-box simulations in Figure~\ref{fig:NoisyTarRespSurfs} by plotting the output $(\mathcal{F}_S, T_{out})$ for various values of the oxygen inflow rates $\mathcal{F}_{O_2} \in [2.2198, 2.2200]$ and lengths of the reactor $l \in \{1.98, 1.99, 2.00\}$. To reveal the errors, we perform a transformation of the outputs; specifically, we subtract the approximated affine part of the solution. In Figure \ref{fig:norm_Hessians_tar}, we plot the norms of the Hessians of the objective function model and the first constraint model for a stopping criterion of $\rho_{min} = 10^{-5}$. We see that the error indicator marks iterations as non-convergent (with circles) as soon as the Hessian norms rise due to the inexact function evaluations. This shows the effectiveness of the error indicator.

% We remark that the corresponding plots for $\rho_{min} = 10^{-3}$ and $\rho_{min} = 10^{-4}$ look exactly the same. This is because our noise indicator terminates NOWPAC at a trust region radius $\rho = 8.4\cdot 10^{-3}$, where neither of the thresholds is reached.
%

\begin{figure}[htbp]
\psfrag{Tag1}[t][t][.85]{$\mathcal{F}_{O_2}$}
\psfrag{Tag2}[b][c][.85]{$\mathcal{F}_{S}$}
\hspace{10mm}
\begin{center}
\includegraphics[bb= 96   250   515   553, scale=.40]{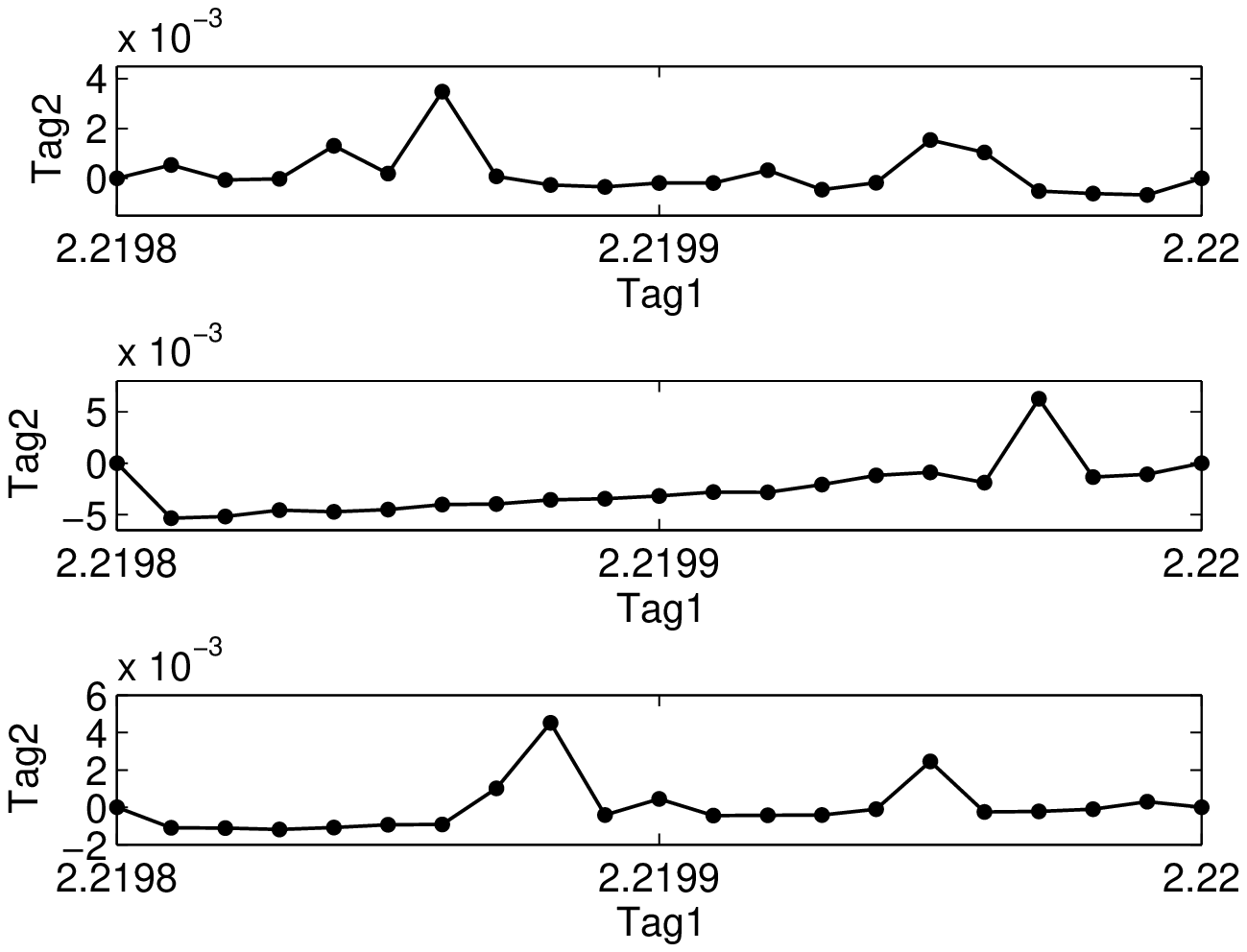}
\psfrag{Tag2}[b][c][.85]{$T_{out}$}
\includegraphics[bb= 96   250   515   553, scale=.40]{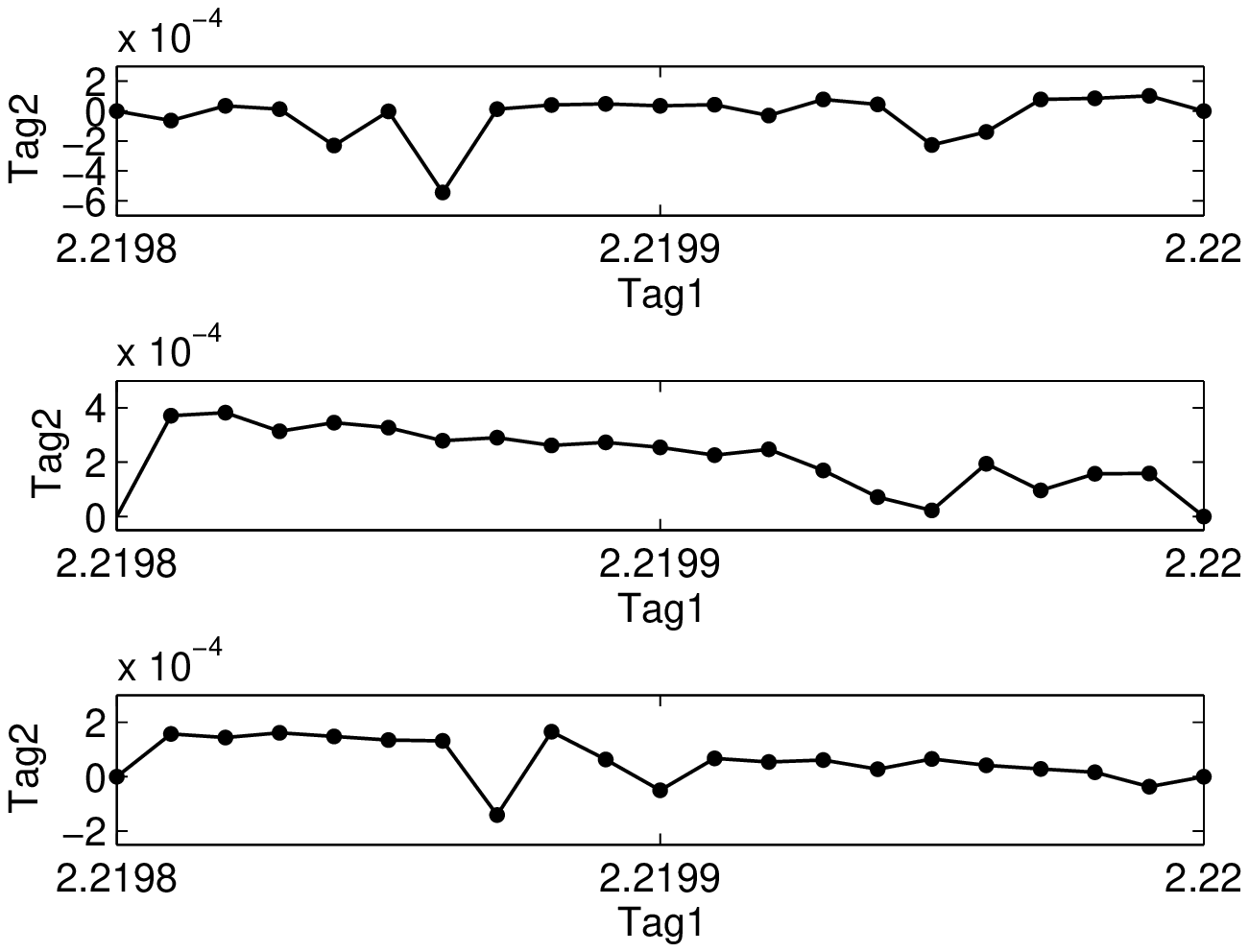}
\caption{Affinely-transformed evaluations of the flowrate of syngas $\mathcal{F}_{S}$ (left) and reactor temperature $T_{out}$ (right) with varying 
oxygen inflow $\mathcal{F}_{O_2} \in [2.2198, 2.2200]$ at reactor lengths of $l = 1.98$ (top), $l = 1.99$ (middle), $l = 2.00$ (bottom).}\label{fig:NoisyTarRespSurfs}
\end{center}
\end{figure}

\begin{figure}[htbp]
\begin{center}
\includegraphics[width=64mm]{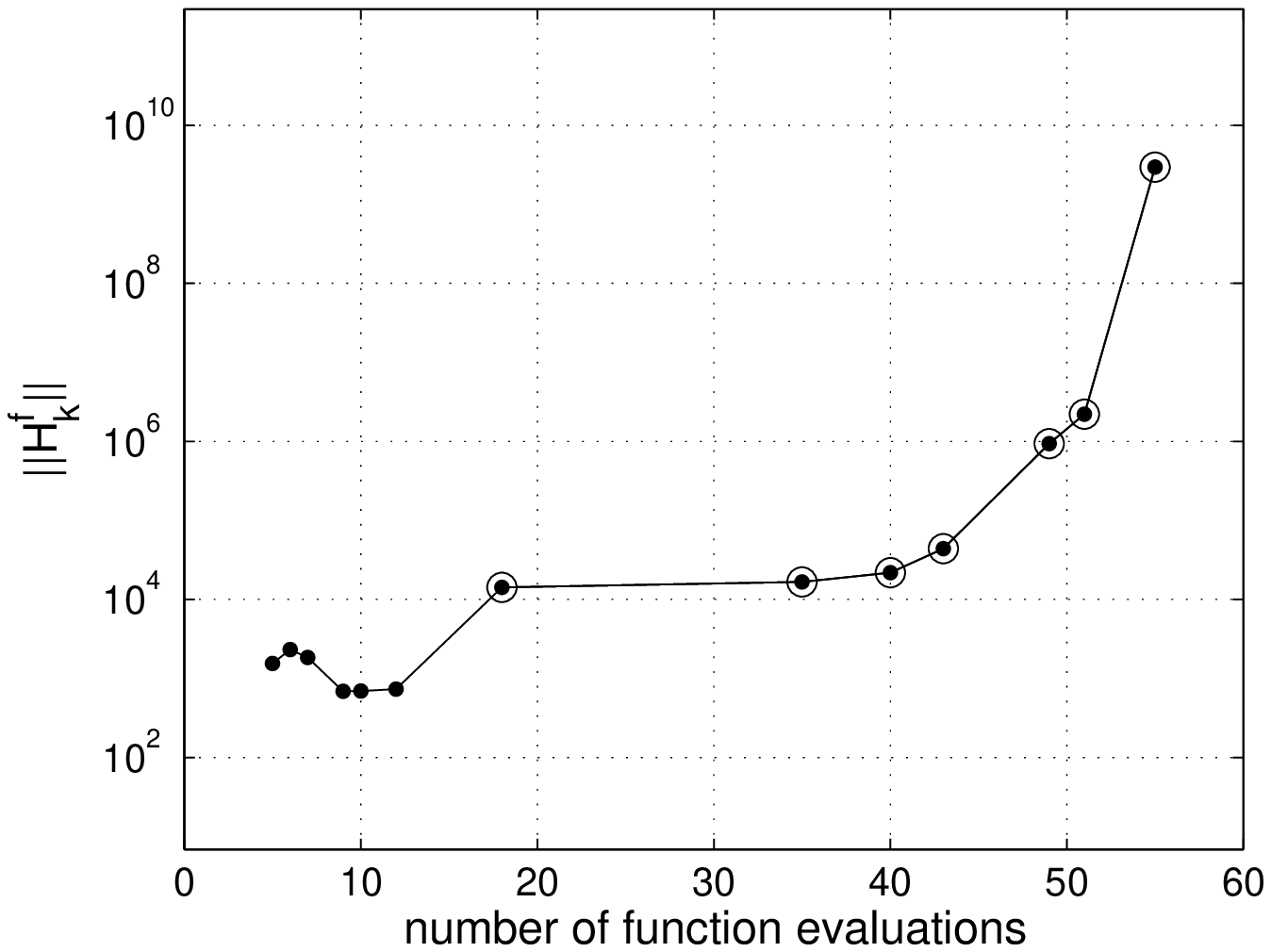}
\includegraphics[width=64mm]{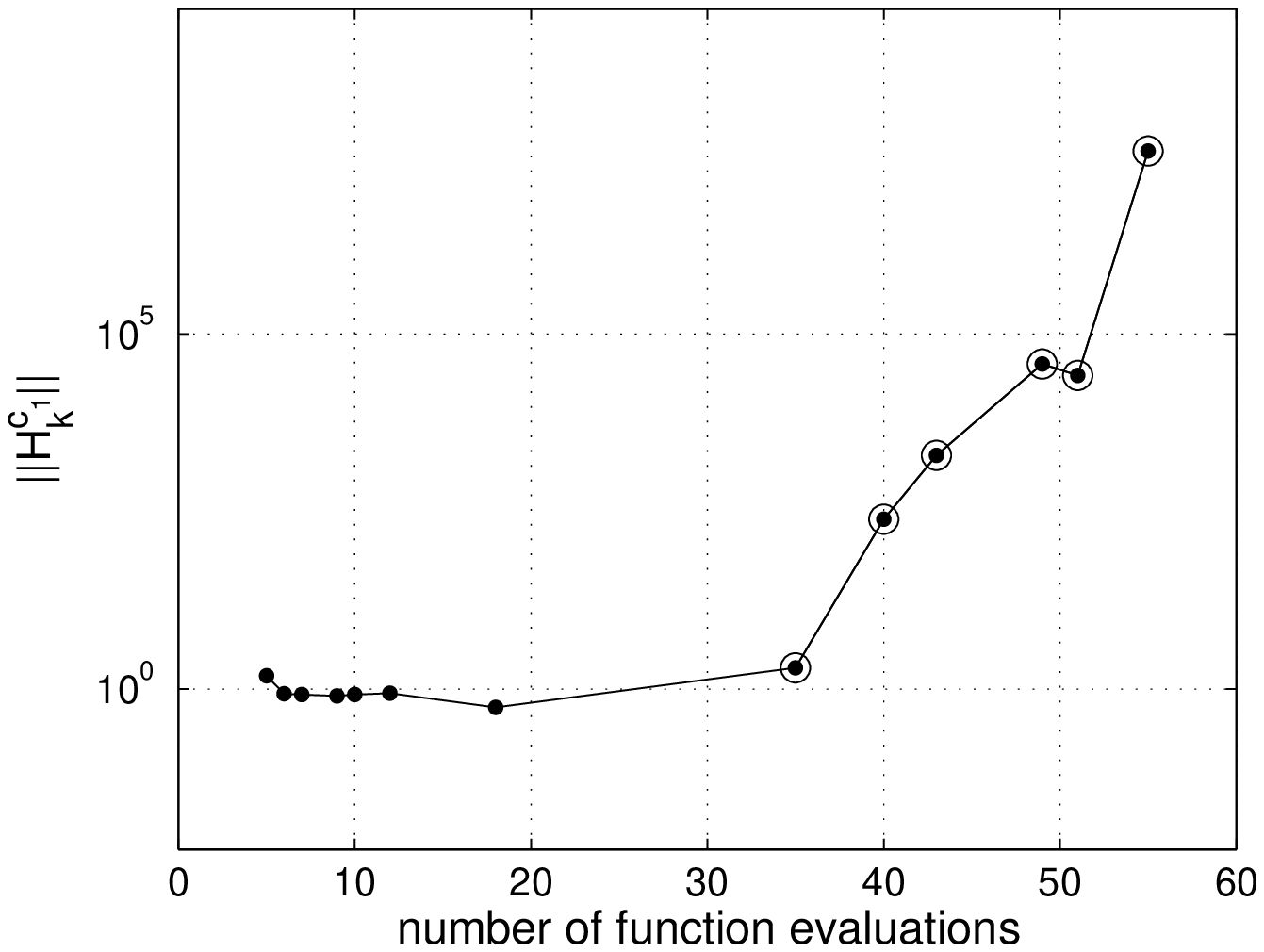}
\caption{Norms of the Hessians of the models of the tar removal process for the objective (left) and the temperature constraint (right). Circles indicate the non-convergent regime.}\label{fig:norm_Hessians_tar}
\end{center}
\end{figure}

\section{Conclusions}\label{sec:conclusions}
This paper has presented a derivative-free trust region method for constrained nonlinear optimization. The method generalizes the work of Conn \textit{et.\ al} \cite{Conn1993, Conn2009a} to handle general black-box constraints without the need for derivative information. We provide a rigorous proof of convergence of the method to first-order critical points. 
This result, given in Section~\ref{sec:convergence_proof}, assumes that evaluations of the objective function and constraints are sufficiently accurate for the trust region surrogate models to satisfy the conditions in (\ref{eq:fully_linear_equations}). In many practical applications, however, where evaluations of the objective function and the constraints are obtained via calls to a black-box simulation, we may only have access to inexact evaluations whose accuracy cannot be tuned. These inaccuracies in the evaluations may corrupt the full linearity properties (\ref{eq:fully_linear_equations}) of the models $m_k^f$ and $\{m_k^{c_i}\}_{i=1}^r$.
In Section \ref{sec:noisy_function_evaluations}, we therefore derive an asymptotic bound on the decay rate of the errors, with respect to the trust region radii, to guarantee convergence. This theoretical analysis leads to the introduction of a error indicator $\tau$, based on the norms of successive model Hessians. The indicator can be used to terminate NOWPAC iterations once they enter a regime where level of inaccuracy in the evaluations of the objective function and constraints is expected to impede further progress of the algorithm.

We note that the termination approach proposed in this paper is local and perhaps conservative. One might encounter situations in which the objective function or the constraints possess regions where the possible local descent is of the order of the accuracy level of the evaluations, yet after passing through this region, significant descent again becomes possible. In these situations, a less conservative behavior of the error indicator would be desirable. We therefore suggest the option of terminating NOWPAC only after a user-prescribed number of non-convergent iterations, essentially to allow for local randomized exploration in flat areas of the objective function.

Since NOWPAC is solely based on evaluations of the objective function and the constraints, it is applicable to a broad class of optimization problems for which no derivative information is available. Moreover, it is guaranteed to converge to first-order critical points without incurring errors due to approximations of the feasible domain. We emphasize that Algorithm \ref{alg:algorithm1} is a skeleton procedure, wherein the user can chose the most appropriate methods for the computation of the trial step and the most suitable approximation method for the surrogates $m_k^f$ and $\{m_k^{c_i}\}_{i=1}^r$.

In future work, we will explore application of the NOWPAC framework to problems in stochastic optimization. As we discussed briefly in Section \ref{sec:noisy_function_evaluations}, in stochastic programming the objective function and/or constraints are often replaced by averages or other measures of variability or risk associated with a lack of knowledge (see, e.g., \cite{Beyer2007, Rockafellar2002}). These quantities are often estimated using sampling strategies that exhibit uncorrelated errors between neighboring designs, and these errors are thus perfectly detectable by our error indicator $\tau$. Our future work may therefore extend the concept of the error indicator into a feedback scheme for adaptively adjusting the accuracy of the objective/constraint evaluations (e.g., choosing the number of samples in a Monte Carlo approximation) to save computational costs while still guaranteeing convergence.

% In our future work we will apply this approach to robust optimization problems, where the objective function and constraints contain probabilistic operators that must also be approximated.\\
% that must be approximated. \\
% whose outputs must be estimated.\\
% whose outputs must also be approximated.\\
% whose outputs must be approximated.\\

\section*{Acknowledgements}
This work was supported by BP under the BP-MIT Conversion Research Program. The authors wish to thank their colleagues from the MIT Energy Initiative for providing their simulation code for the tar removal process. The authors also thank Dr.\ A.\ Conn and Dr.\ L.\ Horesh from the IBM Thomas J.\ Watson Research Center for helpful discussions and comments on this paper. Moreover, the first author wants to thank Prof.\ P.\ Rentrop, Technische Universit{\"a}t M{\"u}nchen, for an interesting discussion on surrogate models for the trust region subproblems. The authors also thank two anonymous referees for valuable comments and suggestions.

\begin{appendix}

\section{Auxiliary lemmas}
In this appendix we provide auxiliary results that justify Assumptions \ref{ass:envelope_of_fully_linear_models} and \ref{ass:convexification_through_ibp}, in order to complete the proofs in the main sections of this paper. First, Lemma \ref{lem:minFrobeniusnormModelsSatisfyAssumpion2.1} shows that minimum-Frobenius-norm models satisfy Assumption \ref{ass:envelope_of_fully_linear_models}. Lemma \ref{lem:local_convexification} shows that local convexification of the feasible domain is possible by choosing a sufficiently large inner boundary path constant $\varepsilon_b > 0$. This results allows us, in  Lemma \ref{lem:upper_lower_semicontinuity}, to prove continuity of the criticality measures used in Lemma \ref{lem:quality_of_approx_criticality_measure}. Lemma \ref{lem:big_A_and_approx_crit_lemma_relation} establishes a relationship between the approximated criticality measure $\alpha_k(\rho_k)$ and its extended version $\mathfrak{A}_2[x_k]$. %, as defined in the proof of Lemma \ref{lem:quality_of_approx_criticality_measure}. 
Thereafter we state, for use in Theorem \ref{thm:overall_convergence_of_exact_criticality meausre}, Corollary \ref{cor:technical_corollary}, which is a direct consequence of statements in the proof of  Lemma \ref{lem:quality_of_approx_criticality_measure}. Finally we present Lemma \ref{lem:completion_of_noise_lemma}, which complements the proof of Theorem \ref{thm:error_decrease_rate}.

\medskip

%%%%%%%%%%%%%%%%%%%%%%%%%%%%%%%%%%%%%%%%%%%%%%%%%%%%%%%%%%%%%%%%%%%%%%%%%%%%%%%%%%%%

\begin{lemma}[Existence of a bounding function satisfying Assumption~\ref{ass:envelope_of_fully_linear_models} for minimum-Frobenius-norm models]
%\begin{lemma}[Quadratic minimum-Frobenius-norm models satisfy Assumption \ref{ass:envelope_of_fully_linear_models}]
\label{lem:minFrobeniusnormModelsSatisfyAssumpion2.1}
Let $m_x^c$ be a quadratic minimum-Frobenius-norm model of the constraint $c$. A bounding function satisfying  Assumption \ref{ass:envelope_of_fully_linear_models} is given by %\todo{I added this preliminary sentence to the statement of the Lemma to make it self-contained. Any modifications? Answer: It is good to refer to the minimum-Frobenius-norm models. I changed the title of the Lemma - is that okay?}
\[
b_c(s; x, \rho) = \|s\|\kappa_{dc}\rho + \left\{\begin{array}{lll} 
\left(\frac12 \kappa_{bh} + \nu^c\right) \|s\|^2 & \quad \text{for} \quad & \|s\| \leq \rho,\\[1mm]
%\nu^c \|s\|^2 & \quad \mbox{for} \quad & \|s\| \leq \rho,\\[1mm]
\bar{b}_c(s; x, \rho)& \quad \text{else} & \\
\end{array}\right.
\] 
with 
%\begin{align*}
%\bar{b}_c(s; x, \rho) := & \frac12 P_0(\|s\|) (\kappa_{bh} + 2\nu^c) \rho^2 + P_1(\|s\|)(1-\rho)(\kappa_{bh} + 2\nu^c)\rho + \\
%&\frac12 P_3(\|s\|)((\kappa_{bh} + 2\nu^c)\rho^2 + 2\bar{\kappa}_{\lambda_1}(1-\rho)\rho) + 
%P_4(\|s\|)(\kappa_{bh} + 2\nu^c)(1-\rho)\rho,
%\end{align*}
\begin{align*}
\bar{b}_c(s; x, \rho) := & 
\left(\kappa_{dc} + \frac12 \kappa_{bh} + \nu^c\right) \rho^2 P_0(\|s\|) +
 \left(\kappa_{dc}+2\left(\nu^c+ \frac12 \kappa_{bh}\right)\right)\rho (1-\rho) P_1(\|s\|)  + \\
& \left( \left(\kappa_{dc} + \frac12 \kappa_{bh} + \nu^c\right)\rho^2 + \bar{\kappa}_{\lambda_1}(1-\rho)\rho\right) P_3(\|s\|) +\\ 
&\left(\kappa_{dc} + 2\left(\nu^c+ \frac12 \kappa_{bh}\right)\right) \rho^2 (1-\rho) P_4(\|s\|),
\end{align*}
where $\bar{\kappa}_{\lambda_1} > 0 $, $\kappa_{dc}$ is the constant from the fully linear property (\ref{eq:fully_linear_dc}), and $\nu^c$ denotes the Lipschitz constant of the gradient of $c$. The functions $P_0, \ldots, P_4$ denote the cubic spline interpolation basis.
\end{lemma}
\begin{proof}
Let $x_k \in X$ and $\rho_k \in \;]0, \,\rho_{max}]$ be arbitrary and consider the approximation error $\varepsilon_c(x_k+s) := c(x_k+s) - m_k^c(x_k+s)$. Adding $c(x_k) - m_k^c(x_k) = 0$ (note that $m_k^c$ interpolates $c$ at the points $x_k$) results in 
\begin{align*}
\varepsilon_c(x_k+s) + c(x_k) - c(x_k+s) & = m_k^c(x_k) - m_k^c(x_k+s)
 = -s^T g_k^c - \frac12 s^T H_k^c s \\
&=  -s^T \left(g_k^c + H_k^c s\right) + \frac12 s^T H_k^c s
 = -s^T \nabla m_k^c(x_k+s) + \frac12 s^T H_k^c s.
%= -s^T\left(g_k^c + H_k^c(x_k+s)\right) - \frac12 s^T H_k^c s.
\end{align*}
%Note that $\varepsilon_c(x_k) = c(x_k) - m_k^c(x_k) = 0$ since minimum-Frobenius norm models $m_k^c$ interpolate $c$ at $x_k$. 
A Taylor expansion of $c(x_k)$ around $x_k + s$,
\[
c(x_k) = c(x_k+s-s) = c(x_k+s) - s^T\nabla c(x_k+s) - s^T \left(\nabla c(x_k + \theta s ) - \nabla c(x_k+s)\right),
\]
for some $\theta \in \; ]0, 1[$, yields
\[
%-s^T(g_k^c + H_k^c(x_k+s)) - \frac12 s^T H_k^c s 
\varepsilon_c(x_k+s) - s^T\nabla c(x_k+s) - s^T \left(\nabla c(x_k + \theta s ) - \nabla c(x_k+s)\right)
= -s^T \nabla m_k^c(x_k+s) + \frac12 s^T H_k^c s,
\]
i.e.,
\begin{align*}
\varepsilon_c(x_k+s) &= -
%s^T\left(g_k^c + H_k^c(x_k+s)\right) + s^T\nabla c(x_k+s) 
%- \frac12 s^T H_k^c s + s^T \left(\nabla c(x_k + \theta s ) - \nabla c(x_k+ s)\right)\\
s^T \nabla m_k^c(x_k+s) + \frac12 s^T H_k^c s + s^T\nabla c(x_k+s) 
+ s^T \left(\nabla c(x_k + \theta s ) - \nabla c(x_k+ s)\right)\\
%&= s^T \left(\nabla c(x_k+s) - \nabla m_k^c \right) - \frac12 s^T H_k^c s + s^T \left(\nabla c(x_k + \theta s ) - \nabla c(x_k+ s)\right).\\
&= s^T \left(\nabla c(x_k+s) - \nabla m_k^c(x_k + s) \right) + \frac12 s^T H_k^c s + s^T \left(\nabla c(x_k + \theta s ) - \nabla c(x_k+ s)\right).\\
\end{align*}
It follows that
\[
\left| \varepsilon_c(x_k+s) \right| \leq \|s\|\kappa_{dc}\rho_k + 
\left(\frac12\kappa_{bh} + \nu^c(1-\theta)\right)
%\nu^c(1-\theta)
\|s\|^2,
\]
where $\nu^c$ is the Lipschitz constant of the gradient of $c$. Since $(x_k, \rho_k)$ was chosen arbitrarily in $X \times [0, \rho_{max}]$, the bounding function restricted to $B(0, \rho)$ is given by
\[
b_c(s; x, \rho) = \|s\|\kappa_{dc}\rho + \left(\frac12 \kappa_{bh} + \nu^c\right) \|s\|^2.
%b_c(s; x, \rho) = \|s\|\kappa_{dc}\rho + \nu^c \|s\|^2.
\] 
We see that $b_c$ is convex and radial symmetric in $s$, as well as constant in $x$. If  $\rho < 1$, we use polynomial Hermite interpolation to smoothly extend the bounding function to $B(0, 1)$ such that $b_c$ is continuous in $(s; x, \rho) \in B(0, 1) \times X \times [0, 1]$.
\end{proof}

\medskip

%%%%%%%%%%%%%%%%%%%%%%%%%%%%%%%%%%%%%%%%%%%%%%%%%%%%%%%%%%%%%%%%%%%%%%%%%%%%%%%%%%%%

\begin{lemma}\label{lem:local_convexification}
Let $c : \mathbb{R}^n \rightarrow \mathbb{R}$ be a continuously differentiable function with Lipschitz continuous gradient on $\mathcal{L} \subset \mathbb{R}^n$. For every $x \in \mathcal{L}$ with $c(x) \leq 0$ and $\rho > 0$ there exists an $\varepsilon_b > 0$ such that $X_{(x,\rho)}^-$, $X_{(x,\rho)}^{+}$, along with $X_x^{ibp}$ and $X_x \cap B(x, 1)$ for all $(x, \rho) \in X \times [0, \rho_{max}]$, are strictly convex sets.
\end{lemma}

\begin{proof}
We first show the existence of a large enough $\varepsilon_{b,1} > 0$ such that $X_x^{ibp}$ is strictly convex. The first-order Taylor approximation of the inner boundary path $h_x(\eta)$, $\eta \in B(x,1)$ around $\xi \in B(x, 1)$ results in
\[
h_x(\eta) - h_x(\xi) = \left\langle \nabla h_x(\xi), \eta - \xi \right\rangle + \varepsilon_{b,1} R(\xi, \eta)
\]
with non-negative $R(\xi, \eta) \in \Omega(\|\eta-\xi\|^2)$. Moreover, due to the constraints $c$ being continuously differentiable with Lipschitz continuous gradient, there exists a function $T(\xi, \eta)$ such that
\[
c(\eta) - c(\xi) = \left\langle \nabla c(\xi), \eta-\xi\right\rangle + T(\xi, \eta)
\]
with $T(\xi, \eta) \in \mathcal{O}(\|\eta-\xi\|^2)$. Thus we obtain
\[
c(\eta) + h_x(\eta) - c(\xi) + h_x(\xi) = \left\langle \nabla c(\xi) + \nabla h_x(\xi), \eta-\xi\right\rangle + T(\xi, \eta) + \varepsilon_{b,1} R(\xi, \eta).
\]
Now, choosing $\varepsilon_{b,1} > 0$ such that $T(\xi, \eta) + \varepsilon_{b,1} R(\xi, \eta) > 0$ results in
\[
c(\eta) + h_x(\eta) - c(\xi) + h_x(\xi) > \left\langle \nabla c(\xi) + \nabla h_x(\xi), \eta-\xi\right\rangle,
\]
which shows the strict convexity of $X_x^{ibp}$. 
The existence of an $\varepsilon_{b,2} > 0$ such that $X_x \cap B(x_x, 1)$ is strictly convex can be shown analogously by replacing the constraints $c$ with $M_x^c$, i.e., the extended surrogate models  (\ref{eq:cont_convex_extension_of_constraint_models}), which have the same smoothness properties as the constraints $c$ themselves. Moreover, since the bounding function $b_c$ is convex, replacing $c$ by $c + b_c$ in the above proof yields the existence of a constant $\varepsilon_{b,+} > 0$ such that $X_{(x,\rho)}^{+}$ is strictly convex. 
For the same choice, $\varepsilon_{b,-} = \varepsilon_{b,+}$, the outer approximation $X_{(x, \rho)}^{-}$ is also strictly convex. This immediately follows by noting that the offset $\kappa_{\lambda_1}\rho$ to the constraints in $X_{(x, \rho)}^-$ is constant in $x$, and thus does not affect the strict convexity of the set.
The assertion of the lemma now follows with $\varepsilon_b := \max\{\varepsilon_{b,1}, \varepsilon_{b,2}, \varepsilon_{b,+}\}$.
\end{proof}
\medskip

%%%%%%%%%%%%%%%%%%%%%%%%%%%%%%%%%%%%%%%%%%%%%%%%%%%%%%%%%%%%%%%%%%%%%%%%%%%%%%%%%%%%

\begin{lemma}\label{lem:upper_lower_semicontinuity}
The set $X_x^{ibp}$ as well as the inner and outer approximations $X_{(x, \rho)}^{\pm}$ are upper and lower semi-continuous in $(x, \rho)$ on the domain $X \times [0, \rho_{max}]$. Moreover, the criticality measure $\mathfrak{A}[x]$ as well as the lower and upper bounds $\mathfrak{A}_1^{\pm}[x,\rho]$ are continuous on $X$ and $X \times [0, \rho_{max}]$ respectively.
\end{lemma}

\begin{proof}
First we discuss the continuity of 
\[
\mathfrak{A}_1^{+}[x, \rho] = -\min\limits_{\xi \in X_{(x,\rho)}^{+}} \left\langle \nabla f(x), \xi - x\right\rangle
\]
in $X \times [0, \rho_{max}]$. For this we use \cite[Thm. 2.1]{Fiacco1990} and remark that the objective $-\left\langle \nabla f(x), \, \xi - x\right\rangle$ is continuous in $(\xi, x, \rho)$. Note that we included $\rho$ for completeness; since the objective is constant in $\rho$, it is trivially continuous with respect to this variable. Thus, we have to show upper and lower semi-continuity of the feasible set 
\[
X_{(x,\rho)}^+ = \{x+d \; : \; c(x+d) + h_x(x+d) + b_c(d;x,\rho) \leq 0\} \cap B(x, 1).
\]
Note that $X_{(x,\rho)}^+ \subseteq B(x,1)$. Let $N(\bar{x}, \bar{\rho})$ be a bounded neighborhood of $(\bar{x}, \bar{\rho})\in X \times [0, \rho_{max}]$. Then there exists a finite bound $d_x > 0$ such that $\|x_1-x_2\| \leq d_{\bar{x}}$ for all $x_1$, $x_2 \in N(\bar{x}, \bar{\rho})$. It therefore holds that  
\[
\bigcup\limits_{(x,\rho)\in N(\bar{x}, \bar{\rho})} X_{(x, \rho)}^+ \subseteq
B(\bar{x}, 1 + d_{\bar{x}}),
\]
i.e., $X^+$ is uniformly compact; see \cite[p. 217]{Fiacco1990}. Moreover, since $\mathcal{L}$ is compact (see Assumption~\ref{ass:general_assumptions}(a)) there exists a ball $B(0, d_{\mathcal{L}}) \subset \mathcal{L}$ such that 
\[
X_{(\bar{x}, \bar{\rho})}^+ = \{\xi \in B(0, d_{\mathcal{L}})  \;: \; 
g(\xi; x, \rho) \leq 0, \; i = 0 \ldots  r\},
\]
with the strictly convex continuous functions $g_0(\xi; x, \rho) := \|\xi-x\|^2 -1$ and
$g_i(\xi; x, \rho) := c(\xi) + b_c(\xi-x; x, \rho) + h_x(\xi)$, $i = 1 \ldots r$; see Assumption~\ref{ass:convexification_through_ibp}. It now follows from \cite[Thm. 2.9]{Fiacco1990} that $X^+$ is open and closed at every $(x, \rho) \in X \times [0, \rho_{max}]$, which in turn implies lower and upper semi-continuity of $X^+$ relative to $X \times [0, \rho_{max}]$; see \cite[p. 217]{Fiacco1990}. Finally, \cite[Thm. 2.1]{Fiacco1990} yields the continuity of $\mathfrak{A}_1^{+}[x, \rho]$ in $X \times [0, \rho_{max}]$.

The continuity of $\mathfrak{A}^-[x, \rho]$ and $\mathfrak{A}[x]$ follows analogously by considering the functions $g_i(\xi; x, \rho) := c(\xi) - \kappa_{\lambda_3}\rho + h_x(\xi)$ and $g_i(\xi; x, \rho) := c(\xi) + h_x(\xi)$ for $i = 1 \ldots r$, respectively.
\end{proof}

\medskip

%%%%%%%%%%%%%%%%%%%%%%%%%%%%%%%%%%%%%%%%%%%%%%%%%%%%%%%%%%%%%%%%%%%%%%%%%%%%%%%%%%%%

\begin{lemma}\label{lem:big_A_and_approx_crit_lemma_relation}
If $X_k$ is convex and $x_k \in X_k$, with the corresponding trust region radius $\rho_k$ defined by Algorithm \ref{alg:algorithm1}, then $\mathfrak{A}_2[x_k] \leq \alpha_k(\rho_k)$ for all $0 < \rho_k < 1$.
\end{lemma}

\begin{proof}
Let us recall the definitions of the intermediate and approximated criticality measures
\[
\mathfrak{A}_2[x_k] = \left| \min\limits_{\substack{x_k + s \in X_k \\ \|s\|\leq 1}} \left\langle g_k^f, s \right\rangle\right| \quad \mbox{and} \quad
\alpha_k(\rho_k) = \frac{1}{\rho_k} \left| \min\limits_{\substack{x_k + d \in X_k \\ \|d\| \leq \rho_k}} \left\langle g_k^f, d \right\rangle\right|.
\]
First we note that in the trivial case of $g_k^f = 0$, $\mathfrak{A}_2[x_k] = \alpha_k(\rho_k) = 0$. From here on, we therefore assume that $g_k^f \neq 0$. 
We denote the optimal solutions of the criticality subproblems by
\[
\hat{s}_k := \argmin\limits_{\substack{x_k + s \in X_k \\ \|s\| \leq 1}} \left\langle g_k^f, s \right\rangle \quad \mbox{and} \quad
\hat{d}_k := \argmin\limits_{\substack{x_k + d \in X_k \\ \|d\| \leq \rho_k}} \left\langle g_k^f, d \right\rangle.
\]
Also, let $\mathcal{N}_k$ denote the normal cone to all constraints $m_k^c + h_k$ that are active in $x_k + \hat{d}_k$. 
We proceed by looking at the three possible positions of the optimal solution $x_k + \hat{d}_k$ individually. Case $1$ addresses situations where $x_k + \hat{d}_k$ lies on the trust region boundary, but not on the boundary of $X_k$. Case $2$ addresses situations where $x_k + \hat{d}_k$ lies on the boundary of  $X_k$ but not on the trust region boundary. Case $3$ considers the situation where $x_k + \hat{d}_k$ is constrained both by the boundary of $X_k$ and by the trust region boundary.
\medskip

%Note that both $\mathfrak{A}_2[x_k]$ and $\alpha_k(\rho_k)$ are bounded from below by $0$ and from %above by $\|g_k^f\|$. We distinguish three cases for $\hat{d}_k$.\\
\begin{description}
\item[Case 1:] $m_k^c(x_k + \hat{d}_k) + h_k(x_k + \hat{d}_k) < 0$:\\ 
In this case only the trust region constraint is active. Thus, $\|\hat{d}_k\| = \rho_k$ and $\hat{d}_k$ is aligned with $-g_k^f$, yielding 
\begin{equation}\label{eq:gradient_equality_for_alpha}
\alpha_k(\rho_k) = -\frac{1}{\rho_k} \left\langle g_k^f, \hat{d}_k \right\rangle =\frac{\left\langle g_k^f, g_k^f \right\rangle}{\|g_k^f\|} = \|g_k^f\|.
\end{equation}
On the other hand we have
\[
\mathfrak{A}_k[x_k] = -\min\limits_{\substack{x_k + s \in X_k \\ \|s\| \leq 1}} \left\langle g_k^f, s \right\rangle \leq
-\min\limits_{x_k + s \in B(x_k, 1)} \left\langle g_k^f, s \right\rangle = \|g_k^f\| = \alpha_k(\rho_k),
\]
where we used (\ref{eq:gradient_equality_for_alpha}) in the last equality.\\
\item[Case 2:] $\hat{d}_k \in \mathcal{N}_k$:\\
In this case $\hat{d}_k$ lies within the normal cone to the boundary of $X_k$. Moreover, relaxing the trust region constraint $\|d\| \leq \rho_k$ to $\|d\| \leq 1$ has no influence on $\hat{d}_k$, i.e., $\hat{s}_k = \hat{d}_k$. Thus, from the definition of $\mathfrak{A}_2[x_k]$ and the approximated criticality measure $\alpha_k(\rho_k)$, we have $\mathfrak{A}_2[x_k] = \rho_k \alpha_k(\rho_k) \leq \alpha_k(\rho_k)$.\\
\item[Case 3:] $\hat{d}_k \not\in \mathcal{N}_k$, $\mathcal{N}_k \neq \{0\}$:\\
First note that $\mathcal{N}_k \neq \{0\}$ implies that $x_k + \hat{d}_k$ does not lie within the interior of $X_k$. Moreover, since $\hat{d}_k \not\in \mathcal{N}_k$, at least one constraint {\it and} the trust region constraint are active at $x_k + \hat{d}_k$. Thus $\|\hat{d}_k\| = \rho_k$ and $\rho_k^{-1}\hat{d}_k$ is a vector of unit length. Since $\hat{d}_k$ is an optimal solution of the criticality subproblem for $\alpha_k$, the negative gradient $-g_k^f$ lies within the normal cone $\mathcal{C}_k$ to the overall set $X_k \cap B(x_k, \rho_k)$ at the point $x_k + \hat{d}_k$, i.e., it can be written as a convex combination of elements in $\mathcal{N}_k$ and $\hat{d}_k$, i.e., $\mathcal{C}_k = \mbox{conv}(\{\hat{d}_k\} \cup \mathcal{N}_k)$. In Figure \ref{fig:angle_relation}.1 the cone $\mathcal{C}_k$ is depicted as a pyramidal cone bounded by the gray cone $\mathcal{N}_k$ and the two white triangular faces. We decompose the gradient $g_k^f$ into two components $g_k^f = g_k^{\mathcal{N}} + g_k^{\mathcal{N}{\perp}}$, where $g_k^{\mathcal{N}} \in \mathcal{N}_k$ and $g_k^{\mathcal{N}{\perp}}$ is orthogonal to $ \mathcal{N}_k$. Note that $g_k^{\mathcal{N}} \neq 0$ and $g_k^{\mathcal{N}^{\perp}} \neq 0$ since $\hat{d}_k \not\in \mathcal{N}_k$ and $g_k^f \in \mathcal{C}_k$.
We denote the set of indices of constraints that are active in $x_k + \hat{d}_k$ by
\[
\mathcal{I}_k := \left\{i \;:\; c_i\left(x_k + \hat{d}_k\right) = 0\right\}.
\]
Consider the linearized feasible domain $\mathcal{T}_k$ defined by the set of points where the linearizations of all constraints that are active in $x_k + \hat{d}_k$ are less or equal than zero. 
Due to the linearity of $\mathcal{T}_k$, the normal cone at any point on the boundary of $\mathcal{T}_k$ is equal to $\mathcal{N}_k$. Now set
\[
t_k(\Delta) := \argmin\limits_{\substack{x_k + t \in \mathcal{T}_k \\ \|t\| \leq \Delta} } \left\langle g_k^f, t \right\rangle,
\]
$\Delta \in [\rho_k, 1]$, and denote the line going through the points $x_k + \hat{d}_k$ and $x_k + t_k(1)$ by
\[
\mathcal{L}_k := \left\{ x \in \mathbb{R}^n \;:\; x = x_k + \hat{d}_k + \sigma \left(t_k(1) - \hat{d}_k\right),\; \sigma \in \mathbb{R}\right\}.
\]
Note that the line $\mathcal{L}_k$ is aligned with $g_k^{\mathcal{N}^{\perp}}$ and lies on the boundary of the linearized domain $\mathcal{T}_k$, since $g_k^{\mathcal{N}} \neq 0$. 
It holds that $t_k(\rho_k) = \hat{d}_k$, and the vector $t_k := t_k(1)$ is the solution of the approximated criticality subproblem subject to the linearized constraints at the point $x_k + \hat{d}_k$, intersected with $B(x_k, 1)$---i.e., subject to the constraint set $\mathcal{T}_k \cap B(x_k, 1)$. Since the line $\mathcal{L}_k$ is aligned with the gradient component $g_k^{\mathcal{N}^{\perp}}$, which in turn is orthogonal to the normal cone $\mathcal{N}_k$, the line  $\mathcal{L}_k$ is itself orthogonal to $\mathcal{N}_k$.

\begin{figure}
\label{fig:angle_relation}
\psfrag{T1}[c][lc][.8]{$x_{k}$}
\psfrag{T2}[bl][l][.8]{$x_{k} + \hat{d}_k$}
\psfrag{T3}[l][l][.8]{$x_{k} + t_k$}
\psfrag{T4}[tc][br][.8]{$\{x_k\} + \mathcal{N}_k$}
\psfrag{T5}[c][cl][.8]{$\mathcal{L}_k$}
\begin{center}
\includegraphics[bb=96   238   516   553, scale=0.43]{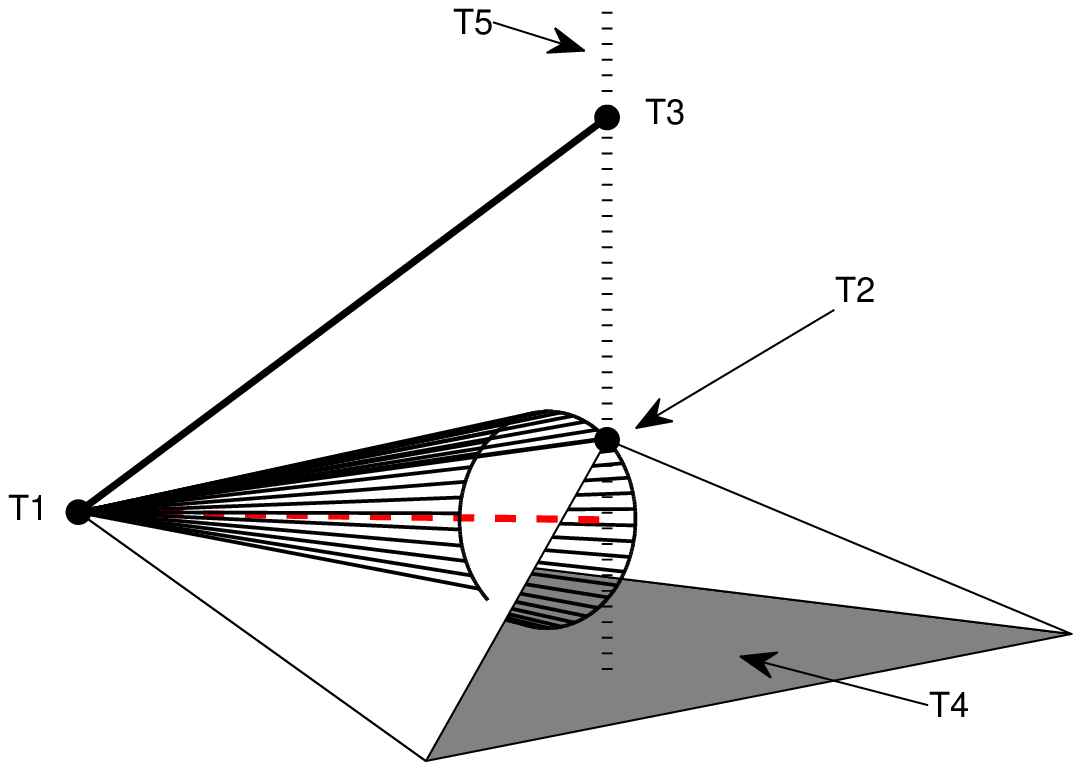}
\includegraphics[bb=96   238   516   553, scale=0.43]{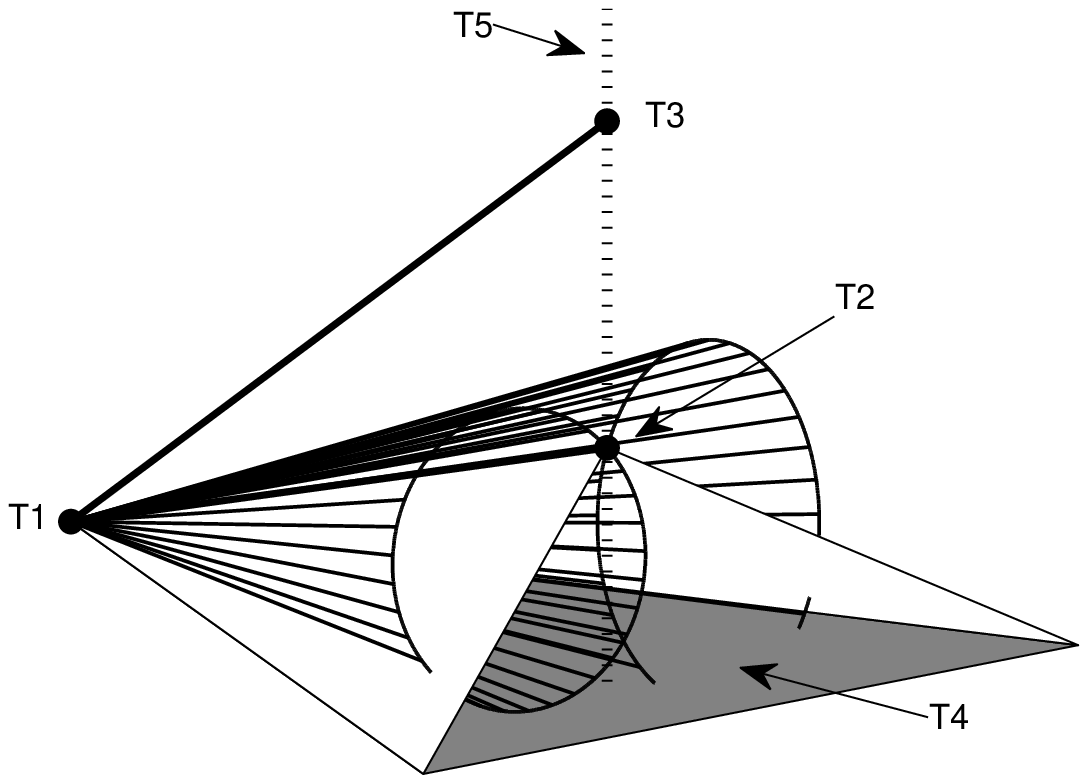}
\caption{Illustration of the (shifted) normal cone $\{x_k\} + \mathcal{N}_k$ (grey area, half hidden underneath the other two faces of $\mathcal{C}_k$) along with $\mathcal{L}_k$ (vertical dotted line). The dashed red line represents one vector $\nu \in \mathcal{C}_n$ and the circular cone $\mathcal{S}_k(\nu)$ surrounding it contains all vectors that enclose an angle with $\nu$ that is less than or equal to $\phi_{d, \nu}$. Right: two circular cones $\mathcal{S}_k(\nu_1)$ and $\mathcal{S}_k(\nu_1)$, $\nu_1, \nu_2 \in \mathcal{C}_n$.}
\end{center}
\end{figure}

\smallskip
Now we will show that the angle $\phi_t$ between $t_k$ and $-g_k^f$ is greater than the angle $\phi_d$ between $\hat{d}_k$ and $-g_k^f$, i.e., $\phi_t \geq \phi_d$. More generally, we denote the angle between $\nu \in \mathcal{C}_n$ and $\hat{d}_k$ by $\phi_{d,\nu}$. For every vector $\nu \in \mathcal{C}_n$ (i.e., for every possible direction of $-g_k^f$) we consider the cylindrical cone $S_k(\nu)$ of all vectors in $\mathbb{R}^n$ that enclose an angle with $\nu$ equal to or smaller than $\phi_{d,\nu}$; see Figure \ref{fig:angle_relation}.1 (left) for an illustration. We define
\[
\mathcal{U}_k := \bigcup\limits_{\nu \in \mathcal{C}_n} \mathcal{S}_k(\nu).
\]
as the set of all vectors $v$ for which there exists a vector $\nu \in \mathcal{C}_n$ such that the angle between $v$ and $\nu$ is smaller than the angle between $\hat{d}_k$ and $\nu$, cf.\ Figure \ref{fig:angle_relation}.1 (right). Since $t_k \not\in \mathcal{C}_n$ and $\mathcal{L}_k$ is orthogonal to $\mathcal{N}_k$ it follows that $t_k \not\in \mathcal{U}_k$. In particular it holds for the special choice of $\nu = -g_k^f \in \mathcal{C}_n$ that $\phi_t > \phi_d$.

\smallskip
We now interpret the inner products in the criticality subproblems as cosines between the corresponding vectors, i.e.,
\[
\rho_k\|g_k^f\|\cos(\phi_d) = -\left\langle g_k^f, \hat{d}_k\right\rangle \quad \mbox{and} \quad
\|t_k\|\|g_k^f\|\cos(\phi_t) = -\left\langle g_k^f, t_k\right\rangle,
\]
and conclude that
\begin{align*}
\alpha_k(\rho_k) &= -\frac{1}{\rho_k}\left\langle g_k^f, \hat{d}_k\right\rangle
=  \frac{1}{\rho_k}\rho_k\|g_k^f\|\cos(\phi_d)\\
&\geq \|t_k\|\|g_k^f\|\cos(\phi_t)
 = -\min_{x_k + t \in \mathcal{T}(x_k + \hat{d}_k)} \left\langle g_k^f, t\right\rangle
\geq \mathfrak{A}_2[x_k].
\end{align*}
In the first inequality we used $\|t_k\| \leq 1$ and $\phi_t \geq \phi_d$, whereas in the last inequality we used the fact that $X_k$ is convex, i.e., $X_k \subseteq \mathcal{T}_k$. 
%
%Relaxing the constraint $\|d\| \leq \rho_k$ to $\|d\| \leq 1$, yields at most in an optimal solution $\hat{s}_k$ having a norm smaller or equal to $1$. Moreover, due to the convexity of $X_k$, the angle $\phi_s$ between $\hat{s}_k$ and $g_k^f$ is greater than $\phi_d$. Note that $\phi_s, \phi_d < \frac{\pi}{2}$ or otherwise the directions $\hat{s}_k$, $\hat{d}_k$ were no descent directions. We now interpret the inner products as cosines between the corresponding vectors, i.e.,
%\[
%\rho_k\|g_k^f\|\cos(\phi_d) = -\left\langle g_k^f, \hat{d}_k\right\rangle \quad \mbox{and} \quad
%\|\hat{s}_k\|\|g_k^f\|\cos(\phi_s) = -\left\langle g_k^f, \hat{s}_k\right\rangle,
%\]
%and conclude that
%\begin{align*}
%\mathfrak{A}_2[x_k] &= -\left\langle g_k^f, \hat{s}_k\right\rangle = \left\|\hat{s}_k\right\|\left\|g_k^f\right\|\cos(\phi_s)\\
%&\leq \left\|g_k^f\right\|\cos(\phi_d) = \frac{1}{\rho_k}\rho_k\left\|g_k^f\right\|\cos(\phi_d)\\
%& = -\frac1{\rho_k}\left\langle g_k^f, \hat{d}_k\right\rangle = \alpha_k(\rho_k).
%\end{align*}
%In the first inequality we used that $\|\hat{s}_k\| \leq 1$ and $\phi_s \geq \phi_d$.\\
\end{description}
\medskip
In summary, we see that in all cases the approximate criticality measure dominates $\mathfrak{A}_2[x_k]$, yielding the assertion of the lemma.
\end{proof}

\medskip

The assertion of the following corollary can be directly found in the proof of Lemma \ref{lem:quality_of_approx_criticality_measure}. However, for better readability, we formulate these results in a separate statement since they are also important for the proof of Theorem \ref{thm:overall_convergence_of_exact_criticality meausre}.

\medskip

\begin{corollary}\label{cor:technical_corollary}
Consider the sequence of points $\{x_k\}_k \subset X$ and the associated trust region radii $\{\rho_k\}_k$ with $\rho_k \rightarrow 0$. For every $\varepsilon > 0$ it holds that 
\[\left| \mathfrak{A}[x_k] - \mathfrak{A}_1[x_k] \right| < \varepsilon \quad
\mbox{and}\quad
\left| \mathfrak{A}_1[x_k] - \mathfrak{A}_2[x_k] \right| < \varepsilon
\]
for $k$ sufficiently large.
\end{corollary}

\begin{proof}
Let us recall the definitions of the intermediate criticality measures $\mathfrak{A}_1$ and $\mathfrak{A}_2$ as given in Lemma \ref{lem:quality_of_approx_criticality_measure},
\[
\mathfrak{A}_1[x_k] = \left|\min\limits_{\substack{x_k + s \in X_k \\ \|s\|\leq 1}} \left\langle \nabla f(x_k), s \right\rangle\right|\quad \mbox{and} \quad
\mathfrak{A}_2[x_k] = \left|\min\limits_{\substack{x_k + s \in X_k \\ \|s\|\leq 1}} \left\langle g_k^f, s \right\rangle\right|.
\]
The first assertion of the corollary follows directly from the first part of the proof of Lemma \ref{lem:quality_of_approx_criticality_measure} where we showed that $\lim\limits_{k\rightarrow 0}\left| \mathfrak{A}[x_k] - \mathfrak{A}_1[x_k] \right| = 0$ if $\rho_k \rightarrow 0$. The second assertion directly follows from the second part of the proof of Lemma \ref{lem:quality_of_approx_criticality_measure}, where we showed that $\lim\limits_{k\rightarrow 0}\left| \mathfrak{A}_1[x_k] - \mathfrak{A}_2[x_k] \right| = 0$ if $\rho_k \rightarrow 0$.
\end{proof}

\medskip

Finally we state Lemma \ref{lem:completion_of_noise_lemma}, which completes the proof of Theorem \ref{thm:error_decrease_rate}. It essentially follows directly from the proof of \cite{Conn2009a}[Lemma 5.4]; however, we want to elaborate on the changes required to yield the particular assertion in which we are interested.

\medskip

\begin{lemma}\label{lem:completion_of_noise_lemma}
Under the assumptions of Theorem \ref{thm:error_decrease_rate} it holds that
\begin{align*}
\left\|\nabla f(x_k + s)  - \nabla f_{n,k}(x_k + s) \right\|  & \in o(\rho_k) \quad \mbox{and}\\
\left\|\nabla c(x_k + s)  - \nabla c_{n,k}(x_k + s) \right\|  & \in o(\rho_k).
\end{align*}
\end{lemma}
\begin{proof}
First we define $y := x_k+s$ and recall that $\delta_k^{f_n,k}(y) = f_{n,k}(y) - f(y)$ and thus
$\nabla \delta_k^{f_{n,k}}(y) = \nabla f_{n,k}(y) - \nabla f(y)$. Let us denote by $\{y_i\}$ the interpolation points on which we construct the minimum-Frobenius-norm model, 
\begin{equation}\label{eq:noise_completion_lemma_eq1}
m_k^{\delta}(y) = c + g^Ty + \frac12 y^T H y,
\end{equation}
of $\delta_k^{f_{n,k}}(y)$. Now we interpret the minimum-Frobenius-norm model as a quadratic approximation of the constant zero function $\mu(y) \equiv 0$ based on inexact evaluations $\delta_k^{f_{n,k}}(y_i) \in o(\rho_k^2)$, i.e.,
\begin{align}
m_k^{\delta}(y) &= \mu(y) + e^{\delta}(y) \quad \mbox{and} \label{eq:noise_completion_lemma_eq2}\\
\nabla m_k^{\delta}(y) &= g + Hy = \nabla \mu(y) + e^g(y).\label{eq:noise_completion_lemma_eq3}
\end{align}
Here, $e^{\delta}(y)$ and $e^{g}(y)$ are the approximation errors that result from the minimum-Frobenius-norm approximation as well as the inexact evaluations. Evaluating (\ref{eq:noise_completion_lemma_eq1}) at all interpolation points $y_i$ and subtracting (\ref{eq:noise_completion_lemma_eq2}) yields
\begin{align*}
\delta_k^{f_{n,k}}(y_i) - \mu(y) - e^{\delta}(y) &= m_k^{\delta}(y_i) - m_k^{\delta}(y)\\
 &= g^T(y_i - y) + \frac12 (y_i - y)^T H (y_i - y) + (y_i - y)^T H y \\
 & = \left(g + Hy\right)^T(y_i - y) + \frac12 (y_i - y)^T H (y_i - y)\\
% &= \nabla m_k^{\delta}(y)^T(y_i - y) + \frac12 (y_i - y)^T H (y_i - y)\\ 
& =\nabla \mu(y)^T(y_i - y) + (e^g)^T(y_i - y) + \frac12 (y_i - y)^T H (y_i - y),
\end{align*}
where we used (\ref{eq:noise_completion_lemma_eq3}) in the last equation. Now note that $\mu(y) = 0$ and $\nabla \mu(y) = 0$; thus
\[
(e^g)^T(y_i - y) + \frac12 (y_i - y)^T H (y_i - y) = \delta_k^{f_{n,k}}(y_i) - e^{\delta}(y).
\]
Subtracting the above equation for $i = 0$ from all the other equations for $i > 0$ yields
\[
(e^g)^T(y_i - y_0) + \frac12 (y_i - y_0)^T H (y_i - y_0) = \delta_k^{f_{n,k}}(y_i) - \delta_k^{f_{n,k}}(y_0) = o(\rho_k^2),
\]
which is exactly the type of equation discussed in the proof \cite{Conn2009a}[Lemma 5.4]. The assertion of this lemma is therefore a straightforward consequence of \cite{Conn2009a}[Lemma 5.4]. Replacing the objective function $f$ with the constraints $c$, the same proof yields the second statement in the assertion of this lemma.
\end{proof}

\end{appendix}

\bibliography{Literatur.bib} 
\end{document}